%% file: Text.tex
\documentclass[a4paper]{article}

\usepackage[a4paper,top=3cm,bottom=3cm,left=3cm,right=3cm,marginparwidth=1.75cm]{geometry}
\usepackage{amsmath}
\usepackage{cite}
\usepackage{tikz}
\usetikzlibrary{matrix}
\newlength\figureheight
\newlength\figurewidth
\usepackage{pgfplots}
\usepackage{amssymb}
\usepackage{mathdots}
\usepackage{multirow}
\usepackage{algorithm}
\usepackage{algpseudocode}

\usepgfplotslibrary{patchplots}
\pgfplotsset{compat=1.14}

\usepackage{booktabs}
\usepackage{subcaption}

\usepackage{url}

\usepackage{amsthm}

\usepackage{epstopdf}

\usepackage{mathtools}

\newtheorem{theorem}{Theorem}[section]
\newtheorem{lemma}[theorem]{Lemma}%
\newtheorem{property}[theorem]{Property}%
\newtheorem{corollary}[theorem]{Corollary}%

\newtheorem{example}{Example}[section]%
\newtheorem{problem}{Problem}[section]%

\numberwithin{equation}{section}

\newcommand{\block}[1]{
	\underbrace{\begin{matrix}0 & \cdots & 0\end{matrix}}_{#1}
}

\begin{document}

	
	\title{A new Legendre polynomial-based approach for non-autonomous linear ODEs}
		\providecommand{\keywords}[1]{\textit{Keywords: } #1}
	
	\author{Stefano Pozza\footnotemark[1] and Niel Van Buggenhout\footnotemark[1]}

	\renewcommand{\thefootnote}{\fnsymbol{footnote}}
	\footnotetext[1]{Charles University, Sokolovská 83, 186 75 Praha 8, Czech Republic. (email: pozza@karlin.mff.cuni.cz; buggenhout@karlin.mff.cuni.cz )}
	
%
	
		\maketitle
	\begin{abstract}
		We introduce a new method with spectral accuracy to solve linear non-autonomous ordinary differential equations (ODEs) of the kind $ \frac{d}{dt}\tilde{u}(t) = \tilde{f}(t) \tilde{u}(t)$, $\tilde{u}(-1)=1$, with $\tilde{f}(t)$ an analytic function.
		The method is based on a new expression for the solution $\tilde{u}(t)$ given in terms of a convolution-like operation, the $\star$-product. This expression is represented in a finite Legendre polynomial basis translating the initial problem into a matrix problem. An efficient procedure is proposed to approximate the Legendre coefficients of $\tilde{u}(t)$ and its truncation error is analyzed. We show the effectiveness of the proposed procedure through some numerical experiments. The method can be easily generalized to solve systems of linear ODEs.
	\end{abstract}
	\indent \keywords{Legendre polynomials; spectral accuracy; ordinary differential equations.}
	

	\section{Introduction}
	Let $\tilde{A}(t)$ be an $N \times N$ analytic matrix-valued function over the interval $[-1, 1]$. Then the unique solution of the ordinary differential equation (ODE) 
	\begin{equation}\label{eq:ode:intro}
		\frac{d}{dt} \tilde{U}(t) = \tilde{A}(t) \tilde{U}(t), \quad \tilde{U}(-1) = I_N, \quad t\in\left[-1,1\right],
	\end{equation}
    is also an analytic $N \times N$ matrix-valued function $\tilde{U}(t)$, where $I_N$ is the identity matrix of size $N\times N$.
	Systems of non-autonomous linear ODEs are ubiquitous in mathematics and related applications. 
	An application of particular interest is nuclear magnetic resonance spectroscopy (NMR).
   	In NMR, the given matrix-valued function is of the form $\tilde{A}(t) = - 2\pi \imath \tilde{H}(t)$, where $\tilde{H}(t)$ is the Hamiltonian of the system \cite{HaSp98,Le08}.
	The Hamiltonian describes the dynamics of the nuclear spins in some sample that is placed in a varying magnetic field. 
	For $\ell$ spins the Hamiltonian is of size $2^\ell \times 2^\ell$.
	So, for a moderate amount of spins, the Hamiltonian becomes extremely large. Thankfully, this matrix is usually sparse since the dominant interactions are those between neighboring spins.  While such large problems will not be considered in this paper, the numerical approach presented here has the ambition to provide a new framework for tackling these challenging problems. This numerical framework arises from a recently developed analytical framework \cite{GiLuThJa15,GiPo20,BonGis2020} in which the solution of \eqref{eq:ode:intro} is given by a simple expression.
 
When $\tilde{A}(\tau_1)\tilde{A}(\tau_2)=\tilde{A}(\tau_2)\tilde{A}(\tau_1)$ for every $\tau_1,\tau_2 \in \left[-1,1\right]$, $\tilde{U}(t)$ can be expressed in the explicit form:
$$\tilde{U}(t)=\exp\left(\int_s^{t} \tilde{A}(\tau)\, \text{d}\tau\right).$$ 
Unfortunately, in general, $\tilde{U}(t)$ has no known simple expression in terms of $\tilde{A}(t)$. 
Nevertheless, a closed form for $\tilde{U}(t)$ exists in a specific algebraic structure of distributions, as shown in \cite{GiLuThJa15} (another possible expression for the solution is the Magnus expansion; see \cite{Blanes2009}).
Let $\tilde{F}(t,s)$ be a matrix-valued function analytic in both variables over $[-1,1]$.
In the previously appeared works on the $\star$-product, the functions are usually assumed to be smooth. Here, we restrict the assumption to analytic for the sake of simplicity since the application we are considering deals with analytic functions.
Moreover, let $\Theta(t-s)$ be the Heaviside step function
    \begin{equation*}
		\Theta(t-s)  = \begin{cases}
			0, \quad \text{if } t<s\\
			1, \quad \text{otherwise}
		\end{cases}.
	\end{equation*}
Let us define  $\mathcal{A}^N_{\Theta}$ as the set of all the distributions of the kind $F(t,s):=\tilde{F}(t,s)\Theta(t-s)$ with size $N \times N$. 
Given $F, G \in \mathcal{A}^N_{\Theta}$, the $\star$-product \cite{GisPozInv19}, denoted by $\star$, is defined as
	\begin{equation*}
		F(t,s) \star G(t,s) := \int_{-1}^{1} G(t,\tau) F(\tau,s) d\tau, \quad \text{with }F,G \in \mathcal{A}^N_{\Theta}.
	\end{equation*}
    From the definition above, by replacing the matrix-matrix product in the integrand with an appropriate one, the $\star$-product can be easily extended to the matrix-vector, matrix-scalar, vector-scalar, and scalar-scalar products. 
    The $\star$-product identity is the distribution $I_\star(t-s) = \delta(t-s)I_N$, where $\delta(t-s)$ is the Dirac delta distribution, and $I_N$ the identity matrix of size $N$ \cite{schwartz1978}.
    Overall, we have defined an algebraic structure (a module) composed of the set of distributions $\mathcal{D}_0^N : = \mathcal{A}_\Theta^N \cup \{I_\star\}$, the $\star$-product (interpreted both as a matrix-matrix and a scalar product), and the usual addition. 
    Table \ref{table:starAlgebra} summarizes the $\star$-product properties and other useful related definitions.
    \begin{table}[ht]
		\centering
		\begin{tabular}{l|l}
			 Operations and related objects & Description \\				\hline  
			  $F(t,s) \star G(t,s)$                            &  matrix-matrix product $\mathcal{D}_0^N \times \mathcal{D}_0^N \rightarrow \mathcal{D}_0^N$ \\
            $f(t,s) \star G(t,s) $      &  (left) scalar product $\mathcal{D}_0^1 \times \mathcal{D}_0^N \rightarrow \mathcal{D}_0^N$ \\
            $G(t,s) \star f(t,s) $      &  (right) scalar product $\mathcal{D}_0^N \times \mathcal{D}_0^1 \rightarrow \mathcal{D}_0^N$ \\
			$f + g$                                          &  usual addition  \\
			$I_{\star} = \delta(t-s) I_N$                    &  identity   (Dirac delta)       \\
			$R_\star(F)(t,s) = I_\star + \sum_{k\geq 1} F^{\star k}(t,s)$   &  $\star$-resolvent
		\end{tabular}
		\caption{Operations and related objects in the $\star$-framework.}	
		\label{table:starAlgebra}
	\end{table}	
 
 The ODE \eqref{eq:ode:intro} can be formulated in the $\star$-framework in terms of $A(t,s) := \tilde{A}(t) \Theta(t-s)$:
	\begin{equation*}
		\frac{d}{dt} U(t,s) = A(t,s) U(t,s), \quad U(s,s) = I_N, \quad t,s\in\left[-1,1\right].
	\end{equation*}
	By Theorem 3.1 in \cite{GiLuThJa15}, the solution to this ODE can be expressed in the form
 \begin{equation}\label{eq:ODE:star:sol}
     U(t,s) = \Theta(t-s) \star R_\star(A)(t,s),
 \end{equation}
 where $R_\star(\cdot)$ is the $\star$-resolvent, i.e., $R_\star(A)(t,s) = I_\star + \sum_{k\geq 1} A^{\star k}(t,s)$, with $A^{\star k}(t,s)$ the $k$th $\star$-power of $A$.
 Once $U(t,s)$ is known, the solution to the original ODE can be obtained by evaluation in $s=-1$, i.e., $\tilde{U}(t) = U(t,-1)$. The symbolic computation of \eqref{eq:ODE:star:sol} leads to complicated expressions involving special functions, even for problems of moderate size.
  
  This paper introduces a new numerical strategy for the computation of $U(t,-1)$ by discretizing the $\star$-product.
	For the sake of simplicity, we will restrict the discussion to the scalar case
	\begin{equation}\label{eq:ODE_univar}
		\frac{d}{dt} \tilde{u}(t) = \tilde{f}(t) \tilde{u}(t),\quad \tilde{u}(-1) = 1,\quad t\in\left[-1,1\right],
	\end{equation}
	where $\tilde{f}(t)$ is an analytical function.
	It is important to remark that the numerical approximation scheme presented in this paper can be easily extended to the matrix case \cite{PoVB22}.
	We also rewrite Equation~\eqref{eq:ODE_univar} in the $\star$-framework as
	\begin{equation}\label{eq:ODE_bivar}
		\frac{d}{dt} u(t,s) = \underbrace{\tilde{f}(t)\Theta(t-s)}_{=:f(t,s)} u(t,s),\quad u(s,s) = 1,\quad t,s\in\left[-1,1\right].
	\end{equation}
    In the following, we denote $\mathcal{A}_\Theta^1$ with the simpler notation $\mathcal{A}_\Theta$.
	In order to devise a numerical approach, we will transform the computation of $\Theta(t-s) \star R_\star(f)(t,s)$ into a linear algebra problem.
	To this end, the operations and objects in the $\star$-algebraic structure are rewritten as equivalent operations and objects in the usual matrix algebra.
	The distributions in $\mathcal{A}_{\Theta}$ are expanded as an infinite double series of Legendre polynomials, and the coefficients in this series are grouped in a matrix.
	The resulting infinite \textit{coefficient matrix} then represents an element in $\mathcal{A}_{\Theta}$, and the operations in Table \ref{table:starAlgebra} have a meaningful equivalent in the matrix algebra of these coefficient matrices.
 
	Section \ref{sec:Legendre} describes the expansion in Legendre basis and shows how the problem of solving ODE \eqref{eq:ODE_univar} can be reformulated as an infinite matrix problem with the coefficient matrix of $f(t,s)\in \mathcal{A}_{\Theta}$.
	The properties of the coefficient matrix are analyzed in detail in Section \ref{sec:coeffMatrix}, and an analytical formula is provided for its entries.
	In Section \ref{sec:practicalComp}, these properties are used to show that the infinite problem can be approximated by a finite problem, and techniques to solve this finite problem efficiently are proposed.
	Section \ref{sec:NumExp} formulates a finite matrix problem corresponding to approximating $\tilde{u}(t)$ and proposes a numerical procedure to solve it.
	The effectiveness of this procedure is illustrated by numerical examples.
	
	\section{From the $\star$-product to the matrix algebra}\label{sec:Legendre}
	The proposed numerical method for the approximation of $\tilde{u}(t)$ is based on the expansion of the distribution $f(t,s)\in \mathcal{A}_{\Theta}$ in a basis of orthonormal Legendre polynomials.
	The distribution $f(t,s)$ can be represented by its \textit{coefficient matrix}, which contains the expansion (Fourier) coefficients of $f(t,s)$.
	The solution $U(t,-1)$ is obtained by exploiting the connection between the $\star$-product and the usual matrix algebra.
	Section \ref{sec:LegendrePoly} discusses the expansion of functions and distributions in the basis of orthonormal Legendre polynomials and defines the coefficient matrix.
	In Section \ref{sec:star_to_matrix}, the connection between the $\star$-product and the matrix algebra is used to reformulate the problem in \eqref{eq:ODE_univar} as an infinite matrix problem.

	\subsection{Legendre polynomials}\label{sec:LegendrePoly}
	The sequence of orthonormal Legendre polynomials $\{p_k\}_{k\geq 0}$ consists of polynomials $p_k$ of exact degree $k$ that satisfy the orthonormality conditions
	\begin{equation*}
		\int_{-1}^1 p_k(t) p_{\ell}(t) dt= \begin{cases}
			0,\quad \text{if } k\neq \ell\\
			1,\quad \text{if } k=\ell
		\end{cases}.
	\end{equation*}
	For a univariate function $\tilde{f}(t)$, its expansion into the Legendre basis is given by
	\begin{equation*}
		\tilde{f}(t) := \sum_{d=0}^\infty \alpha_d p_d(t),\quad \text{with } \alpha_d = \int_{-1}^1 \tilde{f}(t) p_d(t) dt.
	\end{equation*}
	The Fourier coefficients $\{\alpha_d\}_{d\geq 0}$ decay at a rate depending on the smoothness of $\tilde{f}(t)$.
	 Any analytic function over $[-1, 1]$ allows an analytic continuation to a Bernstein ellipse $E_\rho$ for a $\rho>1$ small enough. Therefore, for some constant $C>0$, the Fourier coefficients satisfy
	\begin{equation}\label{eq:coeffDecayRate}
		\vert a_d\vert \leq C \rho^{-d-1};
	\end{equation}
	for details we refer to \cite{Tr13,WaXi12}.
	Moreover, the orthonormal Legendre polynomials can be bounded by
	\begin{equation*}
		\vert p_d(t)\vert \leq \sqrt{\frac{2d+1}{2}} \quad \text{for } t\in\left[-1,1\right],
	\end{equation*}
	and therefore the truncated expansion $\hat{f}_N(t) := \sum_{d=0}^N \alpha_d p_d(t)$ has the error, measured in the maximum norm for $t\in\left[-1,1\right]$,
	\begin{equation}\label{eq:boundLeg}
		\Vert \tilde{f}(t) - \hat{f}_N(t)  \Vert_\infty = \sum_{d=N+1}^\infty \alpha_d p_k(t) \leq \sum_{d=N+1}^\infty \vert \alpha_d \vert \sqrt{\frac{2d+1}{2}}.
	\end{equation}
	Hence, if the (decaying) coefficients $\alpha_N,\alpha_{N+1},\dots$ are smaller than a given threshold, the truncation $\hat{f}_N(t)$ can provide a good approximation to $\tilde{f}(t)$.\\	
	Consider a distribution $f\in \mathcal{A}_\Theta$.
	Its Legendre expansion is
	\begin{equation*}
		f(t,s) = \sum_{k=0}^{\infty} \sum_{\ell=0}^{\infty} f_{k,\ell} p_k(t) p_\ell(s), \quad \text{for every }t\neq s, \quad t,s\in \left[-1,1\right],
	\end{equation*}	
	with Fourier coefficients given by
	\begin{equation}\label{eq:FourCoeffs}
		f_{k,\ell} = \int_{-1}^{1} \int_{-1}^{1} f(\tau,\rho) p_k(\tau) p_\ell(\rho) d\rho d\tau.
	\end{equation}
	We define the \emph{coefficient matrix} $F$ of the distribution $f(t,s)$, which is the infinite matrix composed of the Fourier coefficients $\eqref{eq:FourCoeffs}$
	\begin{equation}\label{eq:coeffMatrix}
		F := \begin{bmatrix}
			f_{k,\ell}
		\end{bmatrix}_{k,\ell=0}^\infty = \begin{bmatrix}
			f_{0,0} & f_{0,1} & f_{0,2} & \dots \\
			f_{1,0} & f_{1,1} & f_{1,2} & \dots \\
			f_{2,0} & f_{2,1} & f_{2,2} & \dots \\
			\vdots & \vdots & \vdots & \ddots
		\end{bmatrix}.
	\end{equation}
	The distribution $f(t,s)$ is only piecewise smooth, because the Heaviside function $\Theta(t-s)$ introduces discontinuities.
	Thus, its Fourier coefficients $f_{k,\ell}$ do not decay at a geometric rate.
	Due to these discontinuities, there is essentially no decay in the coefficients, and the Gibbs phenomenon arises \cite{GoSh97}.
	This means that the reconstruction of  $f(t,s)$ over the entire domain $[-1,1] \times [-1,1]$ by using only the coefficients $f_{k,\ell}$ is not possible.
	For such a reconstruction, there is no convergence at the discontinuities $t=s$, and away from the discontinuities, it converges only linearly to the actual values. 
	There are techniques to resolve the Gibbs phenomenon; see, for example, \cite{GoSh97,GeTa06}.
	In our setting, such techniques are not needed since we only need accurate coefficients $f_{k,\ell}$ representing $f(t,s)$ in the Legendre basis, and we will not use these to reconstruct the function values on the entire domain $(t,s)\in \left[-1,1\right]\times\left[-1,1\right]$, but only on $t\in\left[-1,1\right]$ for $s=-1$, i.e., where the function is analytic in $t$.

	\subsection{A matrix formulation}\label{sec:star_to_matrix}
	The operations of addition and $\star$-multiplication for distributions in $\mathcal{A}_\Theta$ have equivalent operations in the matrix algebra of the associated coefficient matrices, namely the usual matrix addition and matrix-matrix multiplication.
	\begin{lemma}\label{lemma:matrixProduct}
		Consider $f,g\in \mathcal{A}_{\Theta}$ and their respective coefficient matrices $F,G$ in Legendre basis.
		Then:
		\begin{itemize}
			\item $f+g = h \in \mathcal{A}_{\Theta}$ and its coefficient matrix is $H = F+G$.
			\item $f\star g =h \in \mathcal{A}_{\Theta}$ and, assuming the matrix product is well-defined, its coefficient matrix is $H = F G$.
		\end{itemize}
	\end{lemma}
	\begin{proof}
		Addition: from the Legendre expansion of $f$ and $g$ it follows that
		\begin{align*}
			h = f+g &= \begin{bmatrix}
				p_0(t) & p_1(t) & \cdots
			\end{bmatrix} F \begin{bmatrix}
				p_0(s)\\
				p_1(s)\\
				\vdots
			\end{bmatrix} + \begin{bmatrix}
				p_0(t) & p_1(t) & \cdots
			\end{bmatrix} G \begin{bmatrix}
				p_0(s)\\
				p_1(s)\\
				\vdots
			\end{bmatrix}\\
			&= \begin{bmatrix}
				p_0(t) & p_1(t) & \cdots
			\end{bmatrix} \underbrace{(F+G)}_{=H} \begin{bmatrix}
				p_0(s)\\
				p_1(s)\\
				\vdots
			\end{bmatrix}.
		\end{align*}
		Multiplication: plugging in the double series and using the definition of the $\star$-product provides
		\begin{align*}
			h &= f\star g = \int_{-1}^1 \begin{bmatrix}
				p_0(t) & p_1(t) & \cdots
			\end{bmatrix} F \begin{bmatrix}
				p_0(\tau)\\
				p_1(\tau)\\
				\vdots
			\end{bmatrix} \begin{bmatrix}
				p_0(\tau) & p_1(\tau) & \cdots
			\end{bmatrix} G \begin{bmatrix}
				p_0(s)\\
				p_1(s)\\
				\vdots
			\end{bmatrix} d\tau\\
			&= \begin{bmatrix}
				p_0(t) & p_1(t) & \cdots
			\end{bmatrix} F \left(\int_{-1}^1 \begin{bmatrix}
				p_0(\tau) p_0(\tau) & p_0(\tau) p_1(\tau) & \dots\\
				p_1(\tau) p_0(\tau) & p_1(\tau) p_1(\tau) & \dots\\
				\vdots & \vdots & \ddots
			\end{bmatrix} d\tau\right) G \begin{bmatrix}
				p_0(s)\\
				p_1(s)\\
				\vdots
			\end{bmatrix}.
		\end{align*}
		Thanks to the orthonormality of $\{p_k(t)\}_{k\geq 0}$, the matrix in the middle equals the identity matrix
		\begin{equation*}
			\begin{bmatrix}
				\int_{-1}^1 p_0(\tau) p_0(\tau)d\tau & \int_{-1}^1 p_0(\tau) p_1(\tau)d\tau & \dots\\
				\int_{-1}^1 p_1(\tau) p_0(\tau)d\tau & \int_{-1}^1 p_1(\tau) p_1(\tau)d\tau & \dots\\
				\vdots & \vdots & \ddots
			\end{bmatrix}  = \begin{bmatrix}
				1 & 0 & \dots\\
				0 & 1 & \ddots\\
				\vdots & \ddots & \ddots 
			\end{bmatrix}.
		\end{equation*}
		Thus we get
		\begin{equation*}
			h = f\star g = \begin{bmatrix}
				p_0(t) & p_1(t) & \cdots
			\end{bmatrix} \underbrace{FG}_{=H} \begin{bmatrix}
				p_0(s)\\
				p_1(s)\\
				\vdots
			\end{bmatrix},
		\end{equation*}
		that is, the coefficient matrix for $h$ is $H=FG$, under the assumption that this matrix product is well-defined.
	\end{proof}
	In section \ref{sec:completeProof}, the infinite matrix product is discussed and we show that the matrix product between coefficient matrices of distributions $f\in\mathcal{A}_\Theta$ is always well-defined.
	If $f(t,s)$ is bounded for $t,s\in\left[-1,1\right]$, then the $\star$-resolvent of $f(t,s)$, 
	$R_{\star}(f) = 1_\star + \sum_{k\geq 1} f^{\star k}$,
	exists, since the series $\sum_{k\geq 1} f^{\star k}$ uniformly converges in $\mathcal{A}_\Theta$ for every $t,s \in [-1,1]$; see \cite{GiLuThJa15}. 
	Since,
	\begin{equation*}
		R_{\star}(f) \star (1_\star - f) = \left(1_\star + \sum_{k\geq 1} f^{\star k}\right) \star (1_\star - f) = 1_\star,
	\end{equation*}
	 the $\star$-resolvent is the $\star$-inverse of $(1_\star -f)$, i.e., $R_{\star}(f) = (1_\star - f)^{-\star}$.
	Let $g := \sum_{k\geq 1} f^{\star k}$, then $g \in\mathcal{A}_\Theta$ and, hence, we can define its coefficient matrix $G$. 
	Therefore, we have
	\begin{align*}
		R_{\star}(f) &= 1_\star + \sum_{k\geq 1} f^{\star k} = \phi(t)^T\left(I + G \right) \phi(s), \text{ with  } \phi(\tau):= \begin{bmatrix}
			p_0(\tau)\\
			p_1(\tau)\\
			\vdots
		\end{bmatrix},
	\end{align*}
	which allows us to derive the following relation between $(I+G)$ and $(I-F)$,
	\begin{align}\label{eq:reso:mtx:1}
		1_\star =(R_{\star}(f)\star (1_\star - f))(t,s) &= 
		\left(\phi(t)^T 
		\left(I + G \right) 
		\phi(s)\right)
		\star
		\left(\phi(t)^T 
		\left(I - F \right) 
		\phi(s)\right), \\ \label{eq:reso:mtx:2}
		&= 
		\phi(t)^T 
		\left(I + G \right) 
		\left(I - F \right) 
		\phi(s) =  \phi(t)^T I\, \phi(s).
	\end{align}
	As a consequence, we have the following result.
	\begin{lemma}\label{cor:resolvent}
		Consider $f\in \mathcal{A}_{\Theta}$ and its corresponding coefficient matrix $F$.
		If the inverse of the infinite matrix $(I-F)$ exists, then
		\begin{align*}
			R_{\star}(f) &= \begin{bmatrix}
				p_0(t) & p_1(t) & \dots
			\end{bmatrix}(I-F)^{-1} \begin{bmatrix}
				p_0(s)\\
				p_1(s)\\
				\vdots
			\end{bmatrix}.
		\end{align*}
	\end{lemma}
	\begin{proof}
		Let us define $R_{\star}(f) := \phi(t)^T (I-F)^{-1} \phi(s)$, i.e., set $(I+G) = (I-F)^{-1}$. Then, by Equations~\eqref{eq:reso:mtx:1} and \eqref{eq:reso:mtx:2}, we get $R_{\star}(f)\star (1_\star - f) = 1_\star$.
	\end{proof}
	Combining Lemmas~\ref{lemma:matrixProduct} and \ref{cor:resolvent} allows us to obtain an expression for the Legendre coefficients of $\tilde{u}(t)$ in terms of coefficient matrices.
	This expression is the matrix counterpart to the expression for $\tilde{u}(t)$ in the $\star$-framework: $\tilde{u}(t) = u(t,s)\vert_{s=-1} = \Theta(t-s) \star R_\star(f)\vert_{s=-1}$, see \eqref{eq:ODE:star:sol}, and is stated in the following theorem.
	\begin{theorem}\label{theorem:coeffs_smoothu}
		Consider $f\in \mathcal{A}_{\Theta}$ and its corresponding coefficient matrix $F$.
		Let $T$ denote the coefficient matrix of $\Theta(t-s)$, $I$ the identity matrix and $\{p_k\}_{k\geq 0}$ the sequence of orthonormal Legendre polynomials.
		Assume that $(I-F)$ is invertible.
		Then, the Legendre coefficients $\{c_k\}_{k\geq 0}$ of the solution $\tilde{u}(t)$ of the ODE \eqref{eq:ODE_univar} are given by
		\begin{equation}\label{eq:sysEq}
			\begin{bmatrix}
				c_0\\
				c_1\\
				c_2\\
				\vdots
			\end{bmatrix} = T (I-F)^{-1} \begin{bmatrix}
				p_0(-1)\\
				p_1(-1)\\
				p_2(-1)\\
				\vdots
			\end{bmatrix}.
		\end{equation}
	\end{theorem}
	Based on Theorem \ref{theorem:coeffs_smoothu}, we can formulate a matrix problem that is equivalent to the problem of solving ODE \eqref{eq:ODE_univar}.
	
	\begin{problem}[Infinite matrix problem]\label{prob:matrixProblem}
		Given a smooth function $\tilde{f}(t)$, compute the Legendre coefficients $\{c_k\}_{k=0}^\infty$ of the solution $\tilde{u}(t)$ to the ODE \eqref{eq:ODE_univar}.
		By \eqref{eq:sysEq} this corresponds to three matrix problems:
		\begin{enumerate}
			\item Construct the infinite coefficient matrix $F=\begin{bmatrix}
				f_{k,\ell}
			\end{bmatrix}_{k,\ell=0}^\infty$ of Fourier coefficients in Legendre basis $f_{k,\ell} = \int_{-1}^1\int_{-1}^1 \tilde{f}(\tau)\Theta(\tau-\rho) p_k(\tau)p_{\ell}(\rho) d\rho d\tau$.
			\item Solve the infinite linear system of equations $(I-F)x = \phi(-1)$ for $x$. The right hand side is the column vector $\phi(-1) = \begin{bmatrix}
				p_k(-1)
			\end{bmatrix}_{k=0}^{\infty}$ and $I$ is the infinite identity matrix.
			\item Compute the matrix-vector product $T x = \begin{bmatrix}
				c_0& c_1& c_2& \cdots
			\end{bmatrix}^\top$, where $T$ is the coefficient matrix of $\Theta(t-s)$.
		\end{enumerate}
	\end{problem}
	Problem \ref{prob:matrixProblem} is the main problem to solve.
	In the remainder of this paper, we develop a numerical scheme to approximate its solution and investigate the conditions under which this approximation is expected to converge.
 	
	\section{The coefficient matrix and its properties}\label{sec:coeffMatrix}
	Since the coefficient matrix is central to our analysis and to the proposed procedure, we study its structure.
	In Section \ref{sec:coeffMatrix_decay}, an analytical expression for the entries of the coefficient matrices is presented and we show that the entries decay along the diagonals.
	Section \ref{sec:coeffMatrix_band} proves that the coefficient matrices can be approximated by a banded matrix.
	These results allow us to show, in Section \ref{sec:completeProof}, that the matrix-matrix product between two coefficient matrices of distributions in $\mathcal{A}_{\Theta}$ is well-defined, see Lemma~\ref{lemma:matrixProduct}.
	
	\subsection{Formula for the Fourier coefficients}\label{sec:coeffMatrix_decay}
	The Fourier coefficients $f_{k,\ell}$ \eqref{eq:FourCoeffs} of $f(t,s)= \tilde{f}(t) \Theta(t-s)\in \mathcal{A}_{\Theta}$ are studied by relying on the Legendre expansion of the analytical function $\tilde{f}(t) = \sum_{d=0}^{\infty} \alpha_d p_d(t)$, its fast decaying coefficients $\{\alpha_d\}_{d\geq 0}$ and the coefficient matrices of $p_d(t)\Theta(t-s)\in \mathcal{A}_{\Theta}$.
	The Fourier coefficients of $p_d(t)\Theta(t-s)$ can be computed via an analytical formula, see Theorem~\ref{theorem:FormulaLegCoeffs}.
	This formula follows from combining the two known properties of Legendre polynomials below.
	\begin{property}[Integral of a Legendre polynomial on a subinterval {\cite[p.178]{Sa77}}]\label{prop:int_shift_Legendre}
		Let $p_\ell(t)$ denote the orthonormal Legendre polynomial of degree $\ell$.
		For $\ell = 0$ it holds that
		\begin{equation*}
			\int_{-1}^{\tau} p_0(\rho)d\rho = \frac{1}{\sqrt{3}} p_1(\tau) + p_0(\tau),
		\end{equation*}
		and for $\ell>0$
		\begin{equation*}
			\int_{-1}^{\tau} p_\ell(\rho)d\rho = \frac{1}{\sqrt{2\ell+1}} \left(\frac{1}{\sqrt{2\ell+3}}p_{\ell+1}(\tau) - \frac{1}{\sqrt{2\ell-1}} p_{\ell-1}(\tau)  \right).
		\end{equation*}
	\end{property}
	
	\begin{property}[Integral of the triple product of Legendre polynomials \cite{GiJeZe88}]\label{prop:intLeg}
		Let $p_\ell$ be the orthonormal Legendre polynomial of degree $\ell$.
		Consider integers $a,b,c\geq 0$ and set $s:=\frac{a+b+c}{2}$ and $\alpha:=\vert b-c\vert$.
		The integral of the product of three orthonormal Legendre polynomials is
		\begin{align*}
			\mathcal{F}_{a,b,c} &:= \int_{-1}^{1} p_a(\rho) p_b(\rho) p_c(\rho) d\rho \\
			&= 
			\begin{cases}
				0,\quad \text{if } a+b+c \text{ odd},\\
				0, \quad \text{if } s<\max(a,b,c),\\
				0, \quad \text{if } a<\vert b-c \vert,\\
				\frac{\sqrt{(2a+1)(2b+1)(2c+1)}}{\sqrt{2}(a+b+c+1)}
				\begin{psmallmatrix}
					2s-2a\\
					s-a
				\end{psmallmatrix} \begin{psmallmatrix}
					2s-2b\\
					s-b
				\end{psmallmatrix} \begin{psmallmatrix}
					2s-2c\\
					s-c
				\end{psmallmatrix} \begin{psmallmatrix}
					2s\\
					s
				\end{psmallmatrix}^{-1}, \quad \text{else}
			\end{cases}\\
			&= \begin{cases}
				0,\quad \text{if } a+b+c \text{ odd},\\
				0, \quad \text{if } b+c<a\\
				0, \quad \text{if } a<\alpha,\\
				\frac{\sqrt{(2a+1)(2b+1)(2c+1)}}{2^{(2a+1/2)}} \frac{\prod_{j=1}^{a}\frac{-a+b+c+2j}{-a+b+c+2j-1}}{(a+b+c+1)}  {\frac{\prod_{j=(\frac{a+\alpha}{2}+1)}^{a+\alpha}j^2}{\prod_{j=1}^{\frac{a-\alpha}{2}}j^2 \prod_{j=(a-\alpha+1)}^{a+\alpha} j}} ,\quad \text{else}.
			\end{cases}
		\end{align*}
	\end{property}
	
	\begin{theorem}[Coefficients of a Legendre polynomial in $\mathcal{A}_\Theta$]\label{theorem:FormulaLegCoeffs}
		Let $p_d(t),p_k(t),p_\ell(s)$ be the orthonormal Legendre polynomials of degree $d,k,\ell$, respectively, and $\mathcal{F}_{a,b,c}$ as in Property \ref{prop:intLeg}.
		Then the coefficients $b_{k,\ell}^{(d)}$ of the Legendre expansion of $p_d(t)\Theta(t-s)$ are given, for $\ell=0$, by		
		\begin{equation*}
			b^{(d)}_{k,0}= \frac{1}{\sqrt{3}}\mathcal{F}_{d,k,1}+\mathcal{F}_{d,k,0},
		\end{equation*}
		and, for $\ell>0$, by
		\begin{equation}\label{eq:LegBasisCoeffs}
			b^{(d)}_{k,\ell}=\frac{1}{\sqrt{2\ell+1}} \left(\frac{1}{\sqrt{2\ell+3}}\mathcal{F}_{d,k,\ell+1}-\frac{1}{\sqrt{2\ell-1}}\mathcal{F}_{d,k,\ell-1}\right).
		\end{equation}
	\end{theorem}
	\begin{proof}
		By orthonormality of the Legendre polynomials, the Fourier coefficients for $\ell>0$ are given by
		\begin{align*}
			b^{(d)}_{k,\ell} &:= \int_{-1}^1 \int_{-1}^1 p_d(\tau) \Theta(\tau-\rho) p_k(\tau) p_\ell(\rho) d\rho d\tau\\
			&= \int_{-1}^1  p_d(\tau) p_k(\tau) \left(\int_{-1}^1\Theta(\tau-\rho)  p_\ell(\rho) d\rho\right) d\tau
			= \int_{-1}^1  p_d(\tau) p_k(\tau) \underbrace{\left(\int_{-1}^\tau p_\ell(\rho) d\rho\right)}_{\textrm{Apply Property \ref{prop:int_shift_Legendre}}} d\tau \\
			&=  \frac{1}{\sqrt{2\ell+1}} \left[ \frac{1}{\sqrt{2\ell+3}}\int_{-1}^1  p_d(\tau) p_k(\tau) p_{\ell+1}(\tau) d\tau - \frac{1}{\sqrt{2\ell-1}}\int_{-1}^1  p_d(\tau) p_k(\tau)  p_{\ell-1}(\tau)  d\tau\right]\\
			&= \frac{1}{\sqrt{2\ell+1}} \left[\frac{1}{\sqrt{2\ell+3}} \mathcal{F}_{d,k,\ell+1} -\frac{1}{\sqrt{2\ell-1}} \mathcal{F}_{d,k,\ell-1} \right].
		\end{align*}
		For $\ell=0$ the same derivation holds, by using the formula for this case stated in Property \ref{prop:int_shift_Legendre}.
	\end{proof}
	Denote the coefficient matrix of $p_d(t)\Theta(t-s)$ by $B^{(d)} := \left[b_{k,\ell}^{(d)}\right]_{k,\ell=0}^{\infty}$, with $b_{k,\ell}^{(d)}$ as in Theorem \ref{theorem:FormulaLegCoeffs}.
	We will call such a matrix the \textit{Legendre basis matrix of degree $d$}.
	Along a diagonal of $B^{(d)}$ the entries decay linearly, this is formally stated in Lemma \ref{lemma:LegCoeffsDecay}.
	
	\begin{lemma}[Decay of Legendre basis coefficients]\label{lemma:LegCoeffsDecay}
		For $b_{k,l}^{(d)}$, as in Theorem~\ref{theorem:FormulaLegCoeffs}, it holds that
		\begin{equation*}
			\lim_{\substack{k,\ell\rightarrow\infty\\ \vert k-\ell\vert \text{ constant}}} \vert b^{(d)}_{k,\ell} \vert \sim \mathcal{O}(1/\ell).
		\end{equation*}
	\end{lemma}
	\begin{proof}
		In the last equality in the formula in Property \ref{prop:intLeg}, the last fraction is constant since $\alpha = \vert b-c\vert$ is constant and $a=d$ is fixed.
		Then, it is straightforward to see that there is no decay in the expression of the integral over the triple product,
		\begin{equation*}
			\lim_{\substack{k,\ell\rightarrow\infty\\ \vert k-\ell\vert \text{ constant}}} \mathcal{F}_{d,k,\ell} \sim \mathcal{O}(1).
		\end{equation*}
		Since 
		\begin{equation*}
			\lim_{\substack{k,\ell\rightarrow\infty\\ \vert k-\ell\vert \text{ constant}}} \frac{1}{\sqrt{(2\ell+1)(2\ell-1)}} \sim \mathcal{O}(1/\ell),
		\end{equation*}
		the statement follows from \eqref{eq:LegBasisCoeffs}.
	\end{proof}
	The coefficient matrix $F := \left[f_{k,\ell}\right]_{k,\ell=0}^{\infty}$ of $f(t,s) = \tilde{f}(t)\Theta(t-s)\in \mathcal{A}_\Theta$ can be written as $F = \sum_{d=0}^{\infty} \alpha_d B^{(d)}$, where $\{\alpha_d\}_{d\geq 0}$ are the Legendre coefficients of $\tilde{f}(t)$.
	Thus, we can relate properties of $B^{(d)}$ to properties of $F$.
	Namely, the fact that along a diagonal of $F$ the entries decay linearly follows from Lemma \ref{lemma:LegCoeffsDecay} and is illustrated in Example \ref{example:decay}.
	\begin{corollary}[Decay of expansion coefficients]\label{cor:expCoeffsDecay}
		Let $\alpha=\vert k-\ell\vert$ be constant as $k,\ell$ go to infinity. 
		Then the coefficients $f_{k,\ell}$ of the Legendre expansion of $f(t,s)\in \mathcal{A}_{\Theta}$ decay asymptotically at the rate $\frac{1}{\ell}$.
	\end{corollary}
	\begin{example}\label{example:decay}
		The polynomial of degree one $\tilde{f}(t) = -\imath \tau (t+1)$, with $\tau>0$, can be written as a linear combination of $p_0(t)$ and $p_1(t)$, namely $-\imath \tau (t+1) = -2\imath \tau p_0(t) - \sqrt{\frac{2}{3}} \imath \tau p_1(t)$.
		Thus, its coefficient matrix is $F = -2\imath \tau B^{(0)} - \sqrt{\frac{2}{3}} \imath \tau B^{(1)}$, which is a pentadiagonal matrix.
		The order of magnitude of the entries of $F$ for $\tau=4$ are shown in Figure \ref{fig:band_ITVOLT}.
		A linear decay is observed in this figure.
		
		\begin{figure}[!ht]
			\begin{subfigure}{.49\textwidth}
				\centering
				\setlength\figureheight{.6\textwidth}
				\setlength\figurewidth{.6\textwidth}
				\input{figs/ITVOLT_contour_tend=4.tikz}
			\end{subfigure}
			\begin{subfigure}{.49\textwidth}
				\centering
				\setlength\figureheight{2cm}
				\setlength\figurewidth{.85\textwidth}
				\input{figs/diagdecay_ITVOLT.tikz}
			\end{subfigure}
			\caption{Left: the order of magnitude of the entries $f_{k,\ell}$ of $F$, the coefficient matrix of $f(t,s) = \left[-\imath 4(t+1)\right]\Theta(t-s)$. Right: the magnitude of the entries on the first superdiagonal $\vert f_{\ell,\ell+1}\vert$ together with the predicted decay rate $\mathcal{O}(\frac{1}{\ell})$.}
			\label{fig:band_ITVOLT}
		\end{figure}
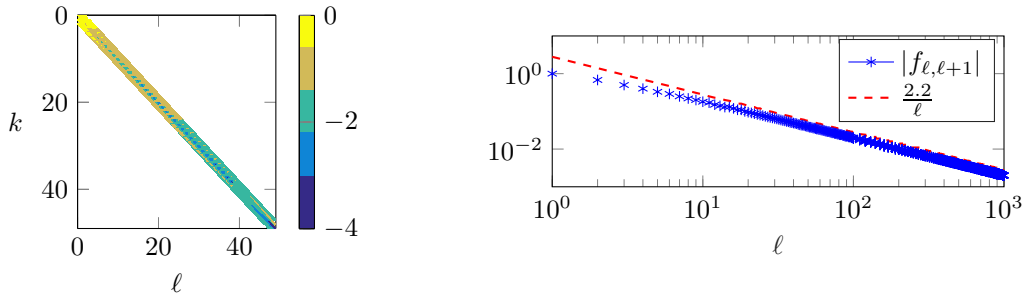
	\end{example}

	\subsection{Banded coefficient matrix}\label{sec:coeffMatrix_band}
	A key property of the coefficient matrices of distributions in $\mathcal{A}_{\Theta}$ is that they are numerically banded.
	That is, they can be approximated by a banded matrix for any given threshold, e.g., machine precision.
	A matrix $A$ is said to be an $N$-banded matrix, or, to have bandwidth $N$, if $a_{k,\ell} = 0$, for $\vert k-\ell \vert>N$.
	In this convention a diagonal matrix is a $0$-banded matrix.
	The following corollary follows from Property~\ref{prop:intLeg} and Theorem \ref{theorem:FormulaLegCoeffs}.
	\begin{corollary}[Bandedness of Legendre basis matrix $B^{(d)}$]\label{cor:banded_Bd}
		Consider the coefficient matrix $B^{(d)}$ of $p_d(t)\Theta(t-s)$, where $p_d(t)$ is the orthonormal Legendre polynomial of degree $d$.
		Then $B^{(d)}$ is a $(d+1)$-banded matrix, i.e.,
		\begin{equation*}
			b_{k,\ell}^{(d)} = 0, \quad \text{for } \vert k-\ell\vert >d+1.
		\end{equation*}
	\end{corollary}	
	We would like to truncate the infinite series ${F}:= \sum_{d=0}^{\infty} \alpha_d B^{(d)}$ to a finite series ${F}^{(N)}:= \sum_{d=0}^{N} \alpha_d B^{(d)}$, which is justified if $\hat{F}^{(N)}$ is in some sense close to $F$.
	Closeness will be expressed in terms of the maximum norm $\Vert F-F^{(N)} \Vert_\infty$.
	In order to bound this quantity we state an upper bound on the maximum norm of the Legendre basis matrices $B^{(d)}$ in the following lemma.
  \begin{lemma}\label{lemma:infNormBd}
		Consider the Legendre basis matrix $B^{(d)}$, its maximum norm can be bounded by
  \begin{align*}
        \Vert B^{(d)} \Vert_\infty \leq 3d+2.
  \end{align*}
 \end{lemma}
	The proof of this lemma is technical and lengthy and is, therefore, postponed to Appendix~\ref{app:A}.
    We remark that we have observed, by numerical computations, that the infinity norm $\Vert B^{(d)} \Vert_\infty$ can be bounded by a constant.
	So the bound is too pessimistic. However, it is sufficient to prove the following result.
	\begin{theorem}\label{theorem:bandedMatrix}
		Consider the coefficient matrix $F$ of $\tilde{f}(t)\Theta(t-s) =f(t,s)\in \mathcal{A}_{\Theta}$.
		For any given tolerance $\delta_{\textrm{tol}}>0$ the matrix $F$ can be approximated by an $(N+1)$-banded matrix $F^{(N)}$ satisfying
		\begin{equation*}
			\Vert F-F^{(N)}\Vert_{\infty} \leq \delta_{\textrm{tol}}.
		\end{equation*}
		In other words, $F$ is a numerically banded matrix.
	\end{theorem}
	\begin{proof}
		Let $\tilde{f}(t) = \sum_{d=0}^{\infty} \alpha_d p_d(t)$ be the Legendre series of the function $\tilde{f}(t)$, analytic in the Bernstein ellipse $E_\rho$, $\rho>1$.
		Its coefficient matrix is $F = \sum_{d=0}^{\infty} \alpha_d B^{(d)}$.
		Using the bound~\eqref{eq:coeffDecayRate} and Lemma~\ref{lemma:infNormBd}, we have, for $F^{(N)} := \sum_{d=0}^{N} \alpha_d B^{(d)}$, that
		\begin{equation}\label{eq:infNormF}
			\Vert F-F^{(N)}\Vert_\infty = \left\Vert \sum_{d=N+1}^{\infty} \alpha_d B^{(d)} \right\Vert_\infty \leq \sum_{d=N+1}^{\infty} \vert \alpha_d\vert \Vert B^{(d)} \Vert_\infty \leq \sum_{d=N+1}^{\infty} C \rho^{-d-1}(3d+2).
		\end{equation}
		Therefore, there exists an $N$ for which $\sum_{d=N+1}^{\infty} C \rho^{-d-1}(3d+2)\leq \delta_{\textrm{tol}}$. This proves the statement.
	\end{proof}
	In the proof above, note that the truncated series
	\begin{equation}\label{eq:bandMatrix}
		F^{(N)} = \sum_{d=0}^{N} \alpha_d B^{(d)},
	\end{equation}
	defines an $(N+1)$-banded matrix sufficiently close to $F$ for $N$ large enough.
	The numerical bandedness of $F$ and the bound in Equation \eqref{eq:infNormF} for $F^{(N)}$ are illustrated in the following example.
	\begin{example}\label{example:bandedness}
		Consider the function $\tilde{f}(t) = -\imath \omega \sin(\omega t)$, where $\omega$ controls the oscillation of the function.
		This function is not a polynomial, so we cannot expect a banded coefficient matrix, however, it is numerically banded.
		For $\omega = 1$, Figure \ref{fig:numBand_nu1} shows (left) the norm $\Vert F-F^{(N)}\Vert_\infty$ and the upper bound \eqref{eq:infNormF} for increasing $N$, and (right) the order of magnitude of the entries of $F$, i.e., $\log_{10}(\vert f_{k,\ell}\vert)$.
		We see a clear numerical band structure of the matrix $F$ and that the upper bound holds.
		With the given threshold chosen equal to machine precision $\delta_{\textrm{tol}}=\epsilon_{\textrm{mach}}$, the bandwidth is $N=14$.
		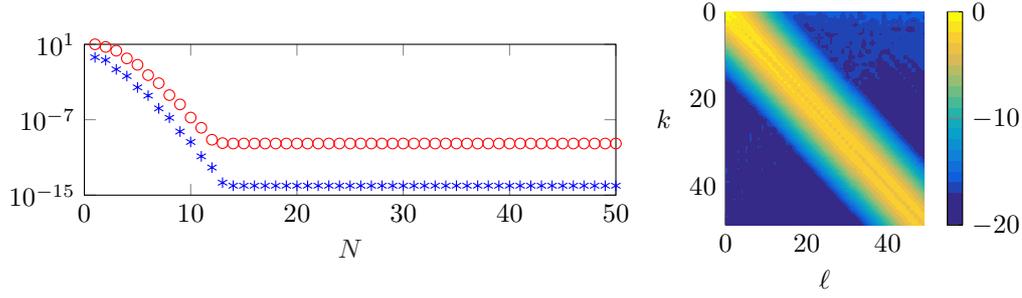
\begin{figure}[!ht]
			\centering
			\begin{subfigure}{.49\textwidth}
				\centering
				\setlength\figureheight{2cm}
				\setlength\figurewidth{\textwidth}
				\input{figs/maxnorm_DE_m=50nu=1.tikz}
			\end{subfigure}
			\begin{subfigure}{.49\textwidth}
				\centering
				\setlength\figureheight{.6\textwidth}
				\setlength\figurewidth{.6\textwidth}
				\input{figs/contour_DE_m=50nu=1.tikz}
			\end{subfigure}
			\caption{Coefficient matrix $F$ of $\left[-\imath \omega \sin(\omega(t+1))\right] \Theta(t-s)$ with $\omega=1$.
				Left: maximum norm $\Vert F-F^{(N)}\Vert_\infty$ (${\color{blue} \ast}$) and the upper bound $\sum_{d=N+1}^{\infty} \vert \alpha_d\vert (3d+2)$ ($\color{red} \circ$).
				Right: order of magnitude of the entries $f_{k,\ell}$.}
			\label{fig:numBand_nu1}
		\end{figure}
		
		For $\omega=5$, Figure \ref{fig:numBand_nu5} shows the same.
		We notice that $\Vert F-F^{(N)}\Vert_\infty$ reaches machine precision at $N= 24$ and again this corresponds to the numerical bandwidth of $F$.
		The bandwidth has increased compared to the less oscillatory function with $\omega=1$, since now we require more Legendre coefficients to represent the function accurately.
		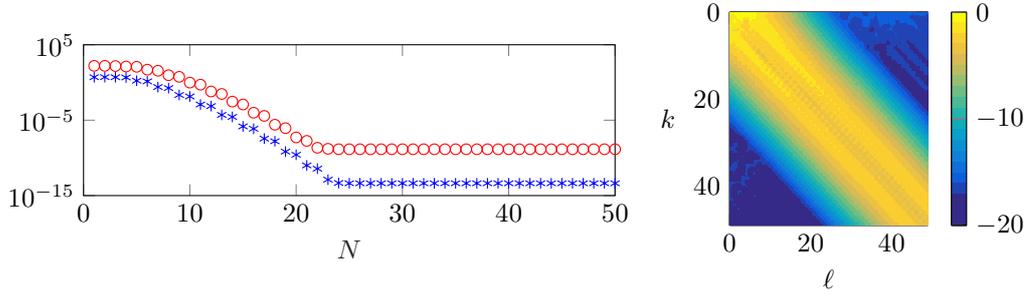
\begin{figure}[!ht]
			\centering
			\begin{subfigure}{.49\textwidth}
				\centering
				\setlength\figureheight{2cm}
				\setlength\figurewidth{\textwidth}
				\input{figs/maxnorm_DE_m=50nu=5.tikz}
			\end{subfigure}
			\begin{subfigure}{.49\textwidth}
				\centering
				\setlength\figureheight{.6\textwidth}
				\setlength\figurewidth{.6\textwidth}
				\input{figs/contour_DE_m=50nu=5.tikz}
			\end{subfigure}
			\caption{Coefficient matrix $F$ of $\left[-\imath \omega \sin(\omega(t+1))\right] \Theta(t-s)$ with $\omega=5$.
				Left: maximum norm $\Vert F-F^{(N)}\Vert_\infty$ (${\color{blue} \ast}$) and the upper bound $\sum_{d=N+1}^{\infty} \vert \alpha_d\vert (3d+2)$ ($\color{red} \circ$).
				Right: order of magnitude of the entries $f_{k,\ell}$.}
			\label{fig:numBand_nu5}
		\end{figure}
	\end{example}

	\subsection{Well-defined matrix product}\label{sec:completeProof}
	Given the distributions $f, g \in \mathcal{A}_\Theta$ and their coefficient matrices $F, G$, until now, and in particular in Lemma~\ref{lemma:matrixProduct}, we have assumed that the matrix product $FG$ is well-defined.
	Thanks to the properties above, we can show that each element of $FG$ exists, i.e., the matrix product is well-defined.
	Consider the Legendre series 
	\begin{equation*}
		\tilde{f}(t) = \sum_{d=0}^\infty \alpha_d p_d(t), \quad \tilde{g}(t) = \sum_{d=0}^\infty \beta_d p_d(t),
	\end{equation*}
	and the related matrix expansions
	\begin{equation*}
		F = \sum_{d=0}^\infty \alpha_d B^{(d)}, \quad G = \sum_{d=0}^\infty \beta_d B^{(d)}.
	\end{equation*}
	Using Corollary~\ref{cor:banded_Bd} and Lemma \ref{lemma:infNormBd}, for $k,\ell=0, 1, \dots$ we get bounds on the $(k,\ell)$-entry of these coefficient matrices
	\begin{equation*}
		|f_{k,\ell}| \leq \sum_{d=|k-\ell|-1}^\infty |\alpha_d| (3d+2) , \quad |g_{k,\ell}| \leq \sum_{d=|k-\ell|-1}^\infty |\beta_d| (3d+2),
	\end{equation*}
	(by convention, when $k=\ell$, $d$ is set to start from $0$).
	As recalled in Equation \eqref{eq:coeffDecayRate}, $\vert \alpha_d\vert \leq C_f \rho_f^{-d-1}$, with $C_f>0$, $\rho_f>1$.
	Therefore, there exists $K_f>0$ such that
	\begin{align*}
		|f_{k,\ell}| &\leq C_f \sum_{d=|k-\ell|-1}^\infty \rho_f^{-d-1} (3d+2) \leq C_f \rho_f^{-|k-\ell|} \sum_{d=0}^\infty \rho_f^{-d-2} (3d + 3|k-\ell|-1) \leq K_{f} \rho_f^{-|k-\ell|}.
	\end{align*}
	The same applies to $|g_{k,\ell}|$ for some $K_g>0$ and $\rho_g>1$. Note that this shows that $F$ is characterized by an off-diagonal exponential decay.
	As a consequence, there exist constants $\rho > 1$ and $C>0$ such that
	\begin{align*}
		\left|(FG)_{k,\ell}\right| &= \left| \sum_{j=0}^{\infty} F_{k,j} G_{j, \ell}\right| \leq \sum_{j=0}^{\infty} |F_{k,j}| |G_{j, \ell}| \leq \sum_{j=0}^\infty C \rho^{-(|j-k| + |j-\ell|)} < \infty,
	\end{align*}
	proving that, for every $k,\ell = 0, 1, \dots$, $(FG)_{k,\ell}$ is well-defined.
	Hence, Lemma~\ref{lemma:matrixProduct} holds for all the distributions in $\mathcal{A}_{\Theta}$.
	This allows us to state that there is a correspondence between the $\star$-product algebraic structure and the matrix (sub)algebra of Legendre coefficient matrices.
	This is summarized in Table \ref{table:matrixAlgebra}.
	\begin{table}[!ht]
		\centering
		\begin{tabular}{l|l}
			$\star$-framework & matrix algebra\\				\hline
			$f(t,s) \star g(t,s)$             &  $F G$ \\
			$f + g$           & $F+G$  \\
			$1_{\star} = \delta(t-s)$ &     $I$       \\
			$R_\star(f)(t,s) = 1_\star +  \sum_{k\geq 1} f^{\star k}(t,s)$   &  $(I-F)^{-1}$
		\end{tabular}
		\caption{The $\star$-operations for distributions $f,g\in \mathcal{A}_{\Theta}$ and the associated matrix algebra operations for the respective coefficient matrices $F,G$.}	
		\label{table:matrixAlgebra}
	\end{table}	

\section{Practical computation in the matrix algebra}\label{sec:practicalComp}
	The second matrix problem in Problem \ref{prob:matrixProblem} is to find the solution $x$ to the infinite system of equations
	\begin{equation*}
		(I-F)x = \phi(-1).
	\end{equation*}
	From Section \ref{sec:coeffMatrix_band} we know that we can accurately represent the infinite numerically banded coefficient matrix $F$ by an infinite banded matrix $F^{(N)}$.
	In Section \ref{sec:field_of_values}, we will discuss the existence of $(I-F^{(N)})^{-1}$.
	That is, the banded infinite system of equations
	\begin{equation*}
		(I-F^{(N)}) x = \phi(-1)
	\end{equation*}
	has a unique solution $x$.
	This solution can be arbitrarily close to the solution of the original system by choosing $N$ appropriately and, therefore, we make a slight abuse of notation by using the same variable $x$ for both these solutions in favor of an easier notation.
 
	To be able to use standard linear algebra techniques we would like to work with a finite system of equations instead of an infinite system.
	In Section \ref{sec:truncErr}, we discuss how an accurate approximation $\dot{x}$ to the first entries of $x$ can be obtained by solving the finite system
	\begin{equation*}
		(I_M - F_M^{(N)}) \dot{x} = \phi_M(-1),
	\end{equation*}
	with $\phi_M(-1)= \begin{bmatrix}
		p_0(-1) & p_1(-1) &	\dots & p_{M-1}(-1)
	\end{bmatrix}^\top$ and  $F_M^{(N)}$ the $M\times M$ leading principal submatrix of $F^{(N)}$.
	In Section \ref{sec:LegCoeffs}, we elaborate on choosing an appropriate value for $M$ and how Legendre coefficients can be obtained from $\dot{x}$.
	Section \ref{sec:fastComp} explores further improvements to compute the solution $\dot{x}$ more efficiently based on exploiting hidden matrix structures.

	\subsection{Resolvent existence and decay phenomenon}\label{sec:field_of_values}
	In this section, we deal with the existence of the resolvent $(I-F^{(N)})^{-1}$. Addressing this problem means discussing the invertibility of an operator. More precisely, consider the Hilbert space $\mathcal{H}$ with orthonormal basis $\{\dot e_0, \dot e_1, \dot e_2, \dots\}$. The coefficients $f_{k,\ell}$ of the (banded) matrix $F^{(N)}$ define the operator $\mathcal{R}: \mathcal{H} \rightarrow \mathcal{H}$ as follows
	\begin{align}\label{eq:operator}
		\mathcal{R} \, \dot e_\ell = \sum_{k=\max\{\ell - N,0\}}^{\ell+N} f_{k,\ell} \, \dot e_k.
	\end{align}
	Denoting by $\mathcal{H}_M$ the linear span of $\{\dot e_0, \dots, \dot e_{M-1} \}$ and with $\mathcal{P}_M : \mathcal{H} \rightarrow \mathcal{H}_M$ the related orthogonal projection, we can define the finite-dimensional operator $\mathcal{R}_M = \mathcal{P}_M \mathcal{R} \mathcal{P}_M$. The operator $\mathcal{R}_M$ is then represented by the matrix $F_M^{(N)}$.
	Theorem 3.1 in \cite{KulRadSar22} shows that the operator $\mathcal{R}$ is invertible under the following conditions:
	\begin{enumerate}  
		\item $F^{(N)}$ is banded;
		\item For every $M=1,2,\dots$, and for $j=1, \dots, (N+1)$, there exist positive constants $K_j, L_j$, such that 
		$$\left\| \left(I_M-F_M^{(N)}\right)^{-1} e_{M-j} \right\|_2 \leq K_j, \quad \left\| \left(I_M-\left(F_M^{(N)}\right)^{H}\right)^{-1} e_{M-j} \right\|_2 \leq L_j.$$
	\end{enumerate}
	In the following, we demonstrate that Condition~2 is satisfied under certain assumptions on the \emph{field of values} of the matrices $F_M^{(N)}$, i.e., the convex set in $\mathbb{C}$ defined as
	$$ W(F_M^{(N)}) := \{ v^H F_M^{(N)} \, v, \|v\|_2 = 1 \}. $$
	Under these assumptions, we show that the matrix $(I_{M}-F_M^{(N)})^{-1}$ is characterized by the so-called \emph{decay phenomenon} (e.g., \cite{BenCetraro16,BenUMI21,BenBoi14}), i.e., the magnitude of its elements decay exponentially as we move away from the band of $F_M^{(N)}$.
	\begin{lemma}\label{lemma:decay}
		Let $A$ be a matrix with bandwidth $N+1$ and so that $W(A)$ is contained in $D(0,r)$, a disk with radius $r<1$ centered at the origin. Then 
		$$ \left| (I-A)^{-1}_{k,\ell} \right| \leq  C \mu^{d(k,\ell)}, \quad d(k,\ell) := \frac{|k-\ell|}{(N+1)}, $$
		for every $r<\mu<1$ and $C$ a constant determined by $\mu$.
	\end{lemma}
	The proof follows immediately from Theorem 2.3 in \cite{PozSim19}; see also \cite{BenBoi14}.
	Assume that $W(F_M^{(N)}) \subset D(0,r)$ for a fixed $r<1$ for $M \geq M_0$, with $M_0$ large enough. Then, for every $\ell=0,1\dots, M-1$, we get
	\begin{align*}
		\left\| \left(I_M-F_M^{(N)}\right)^{-1} e_{\ell} \right\|_2^2 &=
		\sum_{k=0}^M \left| \left(I_M-F_M^{(N)}\right)^{-1}_{k,\ell} \right|^2 \\
		& \leq C^2 \sum_{k=0}^M  \mu^{2 d(k,\ell)} 
		\leq C^2 \sum_{k=0}^M  \tau^{|k-\ell|}, \quad \tau = \mu^{2/(N+1)} < 1, \\
		&\leq C^2 \sum_{k=0}^\infty  \tau^{|k-\ell|} =: K_\ell< \infty.
	\end{align*}
	Note that $K_\ell$ is independent from $M$, proving that Condition~2 above holds (a similar argument holds for the Hermitian transposed case).
	Therefore, by Theorem 3.1 in \cite{KulRadSar22}, we proved the following result.
	\begin{theorem}
		Assume that $W(F_M^{(N)}) \subset D(0,r)$, with $r<1$, for every $M>M_0$, with $M_0$ large enough. Then the operator $\mathcal{R}$ defined in \eqref{eq:operator} is invertible. 
		Moreover, consider the operator equations $\mathcal{R} x = y$ and $\mathcal{R}_M x_M = \mathcal{P}_M y$. If $y$ is in the range of $\mathcal{R}$, then
		$x_M \rightarrow x$ in the norm topology.
	\end{theorem}
	
	As a final step, we need to determine conditions on the function $\Tilde{f}(t)$ so that the coefficient matrix $F_M^{(N)}$ of $\tilde{f}(t)\Theta(t-s)$ satisfies
	$W(F_M^{(N)}) \subset D(0,r)$ for every $M$ large enough.
	Since these matrices are usually characterized by a field of values with a disk shape, we estimate it by using the so-called
	\emph{numerical radius}, which is defined, for a given matrix $A$, as
	$$ \nu(A) := \sup\{|\lambda|, \lambda \in W(A) \}. $$
	Note that $W(A) \subseteq D(0,\nu(A))$.
	Moreover, the numerical radius can also be expressed via the formula
	\begin{equation}
		\nu(A) \leq \max_{k} \left( \sum_{\ell} \frac{|a_{k,\ell}| + |a_{\ell,k}| }{2} \right);
	\end{equation}
	see, e.g., \cite[Corollary 5.2-3]{GusRaoDug97}.
	Unfortunately, obtaining analytic bounds for the numerical radius has proved to be difficult. Therefore, for the moment, we rely on arguments based on numerical observations.
	
	First, consider $B^{(d)}$, the Legendre basis matrix of degree $d$, and the related truncated matrix $B_M^{(d)}$. For $M=2000$ and $d=0,\dots,500$, we observed numerically that
	$$ \max_{k} \left( \sum_{\ell} \frac{|(B_{2000}^{(d)})_{k,\ell}| + |(B_{2000}^{(d)})_{\ell,k}| }{2} \right) \leq 0.87. $$
	Moreover, for all the tested matrices, the maximum was obtained for $k=0$, the first row.
	As the magnitude of the elements of $B_M^{(d)}$ tends to decay along the diagonal (see Lemma \ref{lemma:LegCoeffsDecay}), we can bound the numerical radius by 
	$$\nu(B^{(d)}) \leq 0.87, \quad d=0,\dots,500.$$ 
	
	Finally, by the Legendre expansion of $\tilde{f}(t) = \sum_{d=0}^{\infty} \alpha_d p_d(t)$ we get the bound
	$$ \nu(F^{(N)}) \leq \sum_{d=0}^N |\alpha_d| \nu(B^{(d)}). $$
	For the reasons given above, we conjecture that as long as 
	\begin{equation}\label{eq:f:cond:decay}
		\sum_{d=0}^N |\alpha_d| \leq 1.1494,     
	\end{equation}
	we get the inclusion:
	\begin{equation*}
		W(F_M^{(N)}) \subseteq D(0, 0.87), \quad M=0, 1, \dots  .
	\end{equation*}
	Nevertheless, we have often observed that the solution $x_M$ can still converge even when $\nu(F_M^{(N)})>1$.
	The condition in \eqref{eq:f:cond:decay} is very restrictive, it gives the impression that the techniques presented in this paper are applicable only to slow oscillating, low amplitude functions $\tilde{f}(t)$.
	However, in numerical experiments (Section~\ref{sec:NumExp}) we observe that even when the sum of coefficients is orders of magnitude larger than $1.1494$, or $\nu(F_M^{(N)})>1$, the matrix $(I-F^{(N)}_M)$ is still invertible and its inverse still has an off-diagonal decay.
	Hence, both conditions are nondescriptive in our context, they are too pessimistic.
	More descriptive conditions might~ be obtained by looking at the pseudospectrum of $F^{(N)}$.
	Exploring this path is part of ongoing research.

	\subsection{Truncation error}\label{sec:truncErr}
	The finite banded system is obtained by taking the $M\times M$ leading principal submatrix of the matrix $(I-F^{(N)})$:
	\begin{equation}\label{eq:sys_fin}
		(I_M-F^{(N)}_M) \dot{x} = \phi_M(-1).
	\end{equation}
	We analyze the error $\vert \dot{x}-x_M\vert $, where $x_M$ denotes the vector containing the first $M$ entries of the infinite solution $x$.\\
	First, we derive an analytical expression for $x_M$ in terms of submatrices of the infinite matrix $(I-F^{(N)})$.
	Set, for a more compact notation, $A:=(I_M-F^{(N)}_M)\in\mathbb{C}^{M\times M}$, and the matrices $B,C,D$ as in Figure \ref{fig:Schur}.
	In this figure, the colored region contains generic nonzeros, white indicates zeros, and arrows are used to emphasize the size of a block or region.
	Note that $N+2$ is equal to the bandwidth of the matrix plus one.
	\begin{figure}[!ht]
		\centering
		\input{figs/SchurComplement.tikz}
		\caption{Block subdivision of an infinite banded matrix $(I-F^{(N)})$, the colored region shows the band of the matrix. The left upper block is $A:=(I_M-F^{(N)}_M)\in\mathbb{C}^{M\times M}$, the dimensions of the other blocks follow immediately from this. Open lines indicate that the row and/or column index goes to infinity.}
		\label{fig:Schur}			
	\end{figure}
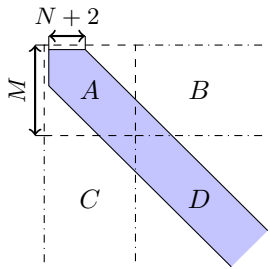
	
	This block subdivision allows us to rewrite the infinite system of equations as
	\begin{equation*}
		(I-F^{(N)}) x = \phi(-1) \Leftrightarrow \begin{bmatrix}
			A & B\\
			C & D
		\end{bmatrix} \begin{bmatrix}
			x_M\\
			z
		\end{bmatrix} = \begin{bmatrix}
			\dot{y}\\
			v
		\end{bmatrix},
	\end{equation*}
	with $x_M\in\mathbb{C}^M$, $\dot{y} = \phi_M(-1)\in\mathbb{C}^M$ and $v=\begin{bmatrix}
		p_M(-1) & p_{M+1}(-1) & \dots
	\end{bmatrix}^{\top}$.
	Assume that $A$ and $D$ are invertible, then the first $M$ entries of the solution are
	\begin{equation*}
		x_M =(I_M-A^{-1}BD^{-1}C)^{-1} A^{-1}\dot{y} - (I_M-A^{-1}BD^{-1}C)^{-1}A^{-1}B D^{-1} v.
	\end{equation*}
	The object under study is the error $\vert \dot{x}-x_M\vert$, where $\dot{x} = A^{-1}\dot{y} $, which is given by
	\begin{equation}\label{eq:error}
		\left\vert \left[(I_M-A^{-1}BD^{-1}C)^{-1} - I_M\right] \dot{x} - (I_M-A^{-1}BD^{-1}C)^{-1}A^{-1} B D^{-1} v\right\vert.
	\end{equation}
	To study this error we look at the matrix structures of the matrices appearing in Equation \eqref{eq:error}.
	Assume that $A^{-1}=(I_M-F^{(N)}_M)^{-1}$ is a numerically banded matrix, see Section \ref{sec:field_of_values}.
	In the following, matrix entries with magnitude below a given threshold are  truncated. As a consequence, the matrix $A^{-1}$ is a $K$-banded matrix.
	Since $W(D)\subseteq W(I-F^{(N)})$, if $F^{(N)}$ shows a decay then, by Lemma \ref{lemma:decay}, $D^{-1}$ also shows a decay and can be approximated accurately by an $L$-banded matrix.
	The values $K$ and $L$ can be estimated a priori by using spectral information of $F^{(N)}$ using, e.g., Lemma~\ref{lemma:decay}.
	In Figure \ref{fig:ABD}, the structure of the matrix products $A^{-1}BD^{-1}$ and $A^{-1}B D^{-1}C$ appearing in \eqref{eq:error} is derived.
	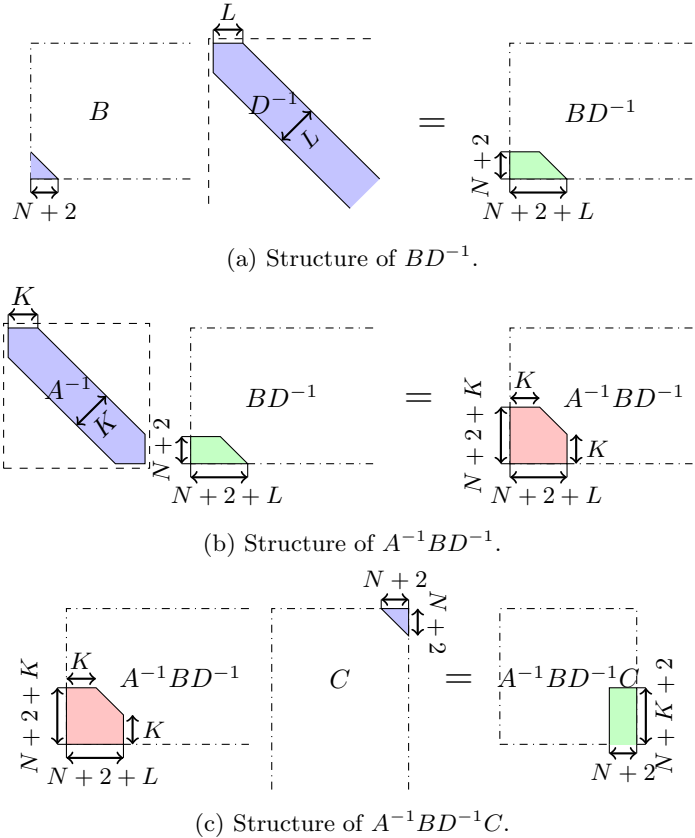
\begin{figure}[!ht]
		\begin{subfigure}{\textwidth}
			\centering
		\input{figs/BD.tikz}
			\caption{Structure of $BD^{-1}$.}
		\end{subfigure}
		\begin{subfigure}{\textwidth}
			\centering
		\input{figs/ABD.tikz}
			\caption{Structure of $A^{-1}BD^{-1}$.}
		\end{subfigure}
		\begin{subfigure}{\textwidth}
			\centering
		\input{figs/ABDC2.tikz}
			\caption{Structure of $A^{-1}BD^{-1}C$.}
		\end{subfigure}
		\caption{Structure of the matrices appearing in the error analysis of $\dot{x}$. Colored regions indicate generic nonzeros and dashed lines indicate the boundary of the matrix, the lack of a dashed line indicates that the row and/or column index goes to infinity.}
		\label{fig:ABD}
	\end{figure}

	The structure of the error is now easily determined from the structure of these matrices. 
	Namely, plug the matrices into the formula $\vert x_M-\dot{x}\vert$ and we obtain that (at least) the first $M-N-K-2$ entries are computed accurately, see Figure~\ref{fig:vectorError}.
	\begin{figure}[!ht]
		\centering
	\input{figs/structure_xdotError.tikz}
		\caption{Structure of the truncation error $\vert x_M - \dot{x} \vert =\vert \left[(I-A^{-1}BD^{-1}C)^{-1} - I \right] \dot{x} - (I-A^{-1}B D^{-1} C)^{-1} A^{-1} B D^{-1} v \vert$.
			Colored regions indicate generic nonzeros and dashed lines indicate the boundary of the matrix, the lack of a dashed line indicates that the row and/or column index goes to infinity.}
		\label{fig:vectorError}			
	\end{figure}
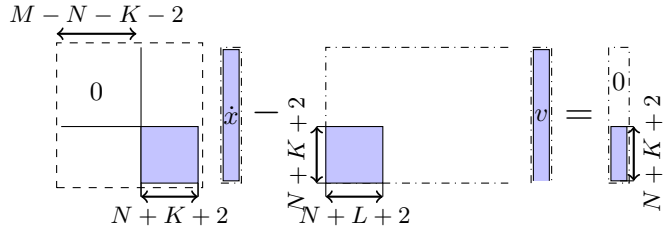		
	In practice, thanks to the decay phenomenon, more than $M-N-K-2$ entries might be computed accurately.
	This is illustrated in Example \ref{example:error}.
	\begin{example}\label{example:error}
		Consider $\tilde{f}(t) = -\imath \sin(t+1)$.
	 For $M=50$, Figure \ref{fig:matrixStructure_contours} shows the corresponding matrix $(I_M-F_M)$, its inverse and the $50\times 50$ leading principal submatrix of $D^{-1}$.
	All these matrices are numerically banded and for $\delta_{\textrm{tol}} = \epsilon_{\textrm{mach}}$ they have the bandwidth $N+2=16$, $K=22$ and $L=16$, respectively.
		\begin{figure}[!ht]
			\centering
			\setlength\figureheight{.28\textwidth}
			\setlength\figurewidth{.28\textwidth}
	\input{figs/Structure_contour_DE_m=100nu=1.tikz}			
			\caption{Order of magnitude of the entries of the matrices $(I_M-F_M)$, $(I_M-F_M)^{-1}$ and $D^{-1}$, respectively, for the function $\tilde{f}(t) = -\imath \sin(t+1)$ and $M=50$. On the far right the colorbar indicates the colors corresponding to the orders of magnitude.}	
			\label{fig:matrixStructure_contours}	
		\end{figure}
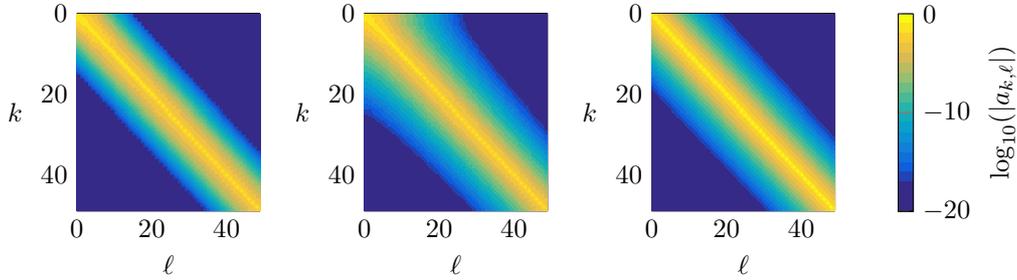
		
		From the above error analysis, we expect that the first $M-N-K-2$ entries of $\dot{x}$ are close to those of the exact (infinite) solution $x$.
		In our example, this would be $50-14-22-2 = 12$ entries.
		The predicted structure in Figure \ref{fig:vectorError} is verified by computing these matrices numerically.
		This is shown in Figure \ref{fig:truncError}, where we have chosen to set elements smaller than $10^{-19}$ to zero (i.e., the color white in the colobar).
		The observed matrices adhere to the predicted structure, but thanks to the decay of the entries it does not fill the whole predicted submatrix with large elements.
		As a consequence, more than the predicted 12 entries of $\dot{x}$ are computed accurately; we observe that 30 entries are computed up to machine precision.
		\begin{figure}[!ht]
			\centering
			\setlength\figureheight{.28\textwidth}
			\setlength\figurewidth{.28\textwidth}
		\input{figs/truncError.tikz}
			\caption{Order of magnitude of the entries of the matrices $(I-A^{-1}BD^{-1}C)^{-1} -I$ and $(I-A^{-1}BD^{-1})^{-1} A^{-1}B D^{-1}$ and of the error vector $\vert x_M-\dot{x}\vert$ for the function $\tilde{f}(t) = -\imath \sin(t+1)$, from left to right, respectively. On the far right the colorbar indicates the colors corresponding to the orders of magnitude.
				Dashed red lines indicate the predicted structure as shown in Figure \ref{fig:vectorError}.}	
			\label{fig:truncError}	
		\end{figure}
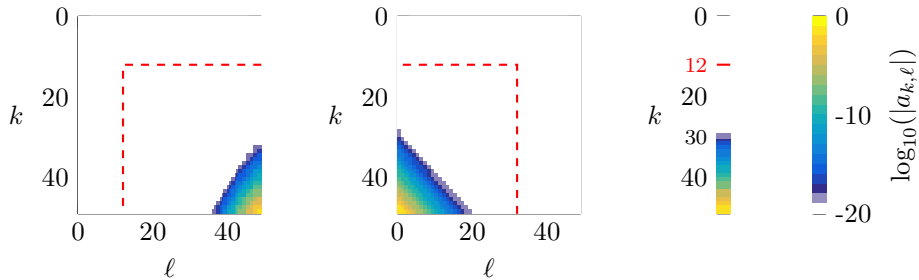		
	\end{example}

	The above example numerically validates the error analysis in this section.
	Thus, the truncation to the leading principal submatrix is both theoretically and numerically justified.
	We would like to stress that the inverse of $A=(I_M-F_M)$ is not computed explicitly in our proposed procedure, instead the system of equations $(I_M-F_M^{(N)})\dot{x} = \phi_M(-1)$ is solved to obtain $\dot{x}$.
	Once $\dot{x}$ is available, the approximate Legendre coefficients are obtained as $\dot{c} = T_M \dot{x}$.
	Since $T_M$ is a tridiagonal matrix, it follows immediately that at least $M-N-K-3$ coefficients of $\dot{c}$ are computed accurately.

	\subsection{Finding the accurate Legendre coefficients}\label{sec:LegCoeffs}
	From the analysis in the above section two questions arise
	\begin{enumerate}
		\item What is an optimal choice for $M$, large enough to accurately compute a sufficient amount of Legendre coefficients and as small as possible, to reduce computational cost?
		\item After computing the Legendre coefficients $\dot{c}$, how many should we keep?
	\end{enumerate}

	The first question is the most challenging and might not have a precise answer.
	We sketch two possible strategies.
	The first strategy can be based on Lemma~\ref{lemma:decay}, which makes it possible to obtain an estimate for the bandwidth $K$ and thus to obtain an estimate for $M-N-K-3$, i.e., the number of accurately computed Legendre coefficients for a system of size $M$.
	For a good estimation of $K$ we require spectral information of $F^{(N)}$, but the bound from Section \ref{sec:field_of_values} is too pessimistic for practical use.
	Moreover, it is difficult to predict how many Legendre coefficients would suffice to accurately represent the unknown solution $\tilde{u}(t)$.\\
	The second strategy follows the strategy employed in chebfun \cite{AuTr17} and it exploits the bound in Equation~\eqref{eq:boundLeg}.
	The idea behind this approach is to first compute the Legendre coefficients for some chosen $M$, and check whether the Legendre coefficients have reached a given tolerance, if not, try $2M$ and repeat until the tolerance is reached.
	The reason such an approach works well for chebfun is because computing the Chebyshev coefficients can be done at a computational complexity of $\mathcal{O}(N\log N)$.
	Thus, in order to make such a method feasible for our procedure we must have an efficient algorithm to solve Problem~\ref{prob:matrixProblem}.
	The next section proposes some directions which can lead to an efficient procedure to compute the Legendre coefficients. \\
	In the sequel we assume that $M$ is chosen appropriately.
	
	The second question is answered by a particular truncation combined with a numerical method that chooses the right amount of coefficients automatically.
	We describe the truncation using an example.
	Consider the function $\tilde{f}(t)= -\imath \omega \sin(\omega(t+1))$, for which we approximate the solution to the ODE \eqref{eq:ODE_univar}, that is $\tilde{u}(t)=\exp\left(-\imath(1-\cos(\omega t + \omega))\right)$.
	For $\omega=1$ this function is the one considered in Example~\ref{example:error}, from which we know that for $M=50$ we compute about $30$ Legendre coefficients up to machine precision, which suffices to represent $\tilde{u}(t)$.
	In Figure \ref{fig:coeffs_DE}, we show the Legendre coefficients obtained in the following two ways:
	\begin{enumerate}
		\item Solve $(I-F_M^{(N)}) \dot{x}=\phi_M(-1)$ and compute $\dot{c} = T_M \dot{x}$.
		\item Solve $(I-\underline{F}_M^{(N)})\underline{\dot{x}}=\phi_M(-1)$ and compute $\underline{\dot{c}} = \underline{T}_M \underline{\dot{x}}$, where the entries of the coefficient matrices are set to zero in the last $K$ rows, where $K$ is the bandwidth of the matrix.
		This means that the first $M-N-1$ rows of $\underline{F}_M^{(N)}$ equal those of $F_M^{(N)}$ and the last $N+1$ rows are all zeros.
		Since $T_M$ has bandwidth equal to one, omitting its last row gives us $\underline{T}_M$.
	\end{enumerate}
	For $\omega=1$ and $M=50$, the Legendre coefficients $\dot{c}$ without truncation lead to the first 30 coefficients being accurately computed and after this the coefficients increase in amplitude, resulting in a \textit{u}-shaped curve for the coefficients.
	This \textit{u}-shape does not correspond to the actual Legendre coefficients and it makes it more difficult to determine how many coefficients one would keep to obtain an accurate approximation to $\tilde{u}(t)$.
	Because, using more than 30 coefficients will cause the approximation to deteriorate.
	
	The Legendre coefficients $\underline{\dot{c}}$ with truncation of the coefficient matrices is equally accurate as $\dot{c}$ for the first $30$ coefficients and does not have an increase after this, the truncation in fact pushes these last coefficients to zero.
	Thus $\underline{\dot{c}}$ is easier to use, since it has a simpler shape allowing chopping of the series by, e.g., the procedure in \cite{AuTr17}, moreover, using more coefficients than 30 will not deteriorate the approximation to $\tilde{u}(t)$.
	
	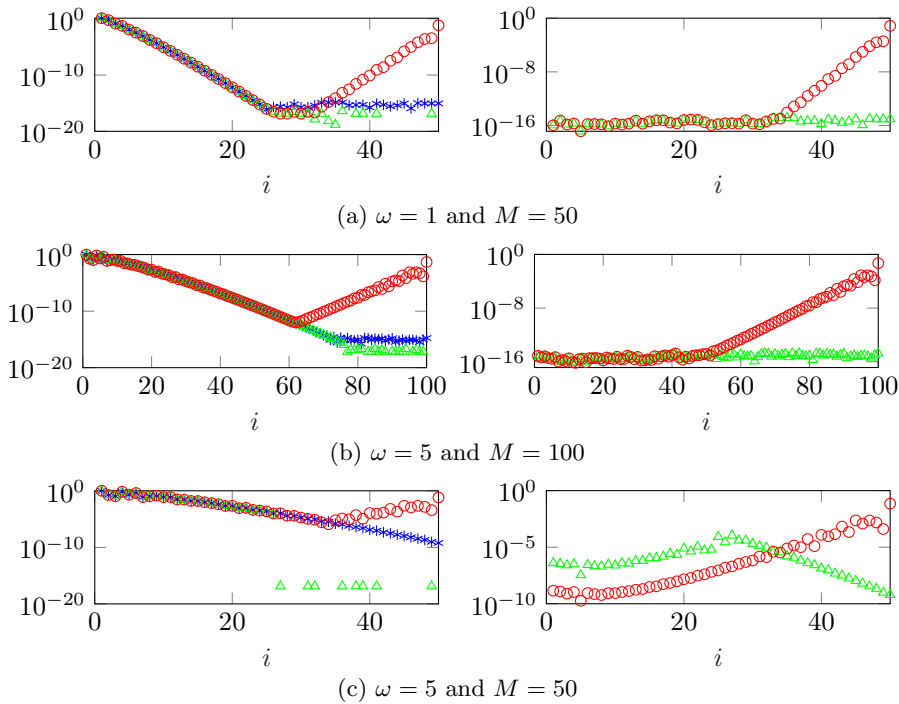
\begin{figure}[!ht]
		\begin{subfigure}{\textwidth}
			\centering
			\setlength\figureheight{1.5cm}
			\setlength\figurewidth{11cm}
			\input{figs/coeffs_DE_M=50nu=1.tikz}
			\vspace{-0.2cm}
			\caption{$\omega=1$ and $M=50$}
		\end{subfigure}
		\begin{subfigure}{\textwidth}
			\centering
			\setlength\figureheight{1.5cm}
			\setlength\figurewidth{11cm}
			\input{figs/coeffs_DE_M=100nu=5.tikz}
			\vspace{-0.2cm}
			\caption{$\omega=5$ and $M=100$}
		\end{subfigure}
		\begin{subfigure}{\textwidth}
			\centering
			\setlength\figureheight{1.5cm}
			\setlength\figurewidth{11cm}
			\input{figs/coeffs_DE_M=50nu=5.tikz}
			\vspace{-0.2cm}
			\caption{$\omega=5$ and $M=50$}
		\end{subfigure}
		\caption{Legendre coefficients obtained by system solve for the coefficient matrix $F^{(N)}_M$ for $\tilde{f}(t) = -\imath \omega \sin(\omega (t+1))$. Left: Legendre coefficients, exact $c$ (${\color{blue} \ast}$), approximation $\dot{c}$ (${\color{red} \circ}$), approximation with truncation $\underline{\dot{c}}$ (${\color{green} \triangle}$).
			Right: the error on the computed coefficients, $\vert c-\dot{c}\vert$ ({\color{red}$\circ$}) and $\vert c - \underline{\dot{c}}\vert$ ({\color{green} $\triangle$}).}
		\label{fig:coeffs_DE}
	\end{figure}

	In Figure \ref{fig:coeffs_DE}, we also show the Legendre coefficients and their error for two other choices of parameters.
	For $\omega=5$ and $M = 100$, we observe similar behavior as for $\omega=1$.	
	For $\omega=5$ and $M = 50$, $M$ is chosen smaller than the minimal required size to represent $\tilde{u}(t)$ up to machine precision.
	Now the first coefficients in $\underline{\dot{c}}$ are computed less accurately than $\dot{c}$ due to the truncation.
	This might be because the last $N=24$ rows of the coefficient matrix $F^{(N)}_M$ of size $M=60$ are truncated, thereby discarding too many entries.

	\subsection{Fast computation of Fourier coefficients}\label{sec:fastComp}
	The first subproblem in Problem \ref{prob:matrixProblem} is to construct the coefficient matrix $F$.
	We have shown that in fact it suffices to approximate $F$ by the finite banded matrix
	\begin{equation*}
		F_M^{(N)}=\sum_{d=0}^{N} \alpha_d B^{(d)}_M.
	\end{equation*}
	The efficient construction of the coefficient matrix thus requires efficient computation of $\{\alpha_d\}_{d\geq 0}$ and $\{B^{(d)}_M\}_{d\geq 0}$.
	Since $B_M^{(d)}$ are the coefficient matrices of the basis, they must be computed only once for each $d$ and can then be reused for different functions.\\
	First an efficient algorithm to approximate the Legendre coefficients $\{\alpha_d\}_{d=0}^{N}$ of $\tilde{f}(t)$ is discussed.
	We use functions that are available in the MATLAB package \verb|chebfun| \cite{DrHaTr14} .
	The Legendre coefficients for the smooth function $\tilde{f}(t)$ are given by $\alpha_d= \int_{-1}^{1} \tilde{f}(t) p_d(t)$ and will be approximated by $\{\hat{\alpha}_d\}_{d=0}^N$ as follows:
	\begin{enumerate}
		\item Using \verb|chebfun|, we compute the coefficients of the interpolating Chebyshev series $\sum_{k=0}^N \hat{c}_d T_d(t) \approx \tilde{f}(t)$. Given a required accuracy, an appropriate truncation value $N$ is chosen automatically \cite{AuTr17}. 
		The coefficients $\{\hat{c}_d\}_{d=0}^N$ are obtained at complexity $\mathcal{O}(N \log(N))$.
		For details on the error incurred by interpolation, instead of by computing the integral $c_d = \int_{-1}^{1} \frac{\tilde{f}(t) T_d(t)}{\sqrt{1-t^2}}dt$, we refer to the book by Trefethen \cite[Chapter 4]{Tr13}.
		\item The Chebyshev coefficients $\{\hat{c}_d\}_{d=0}^N$ can be transformed into Legendre coefficients $\{\hat{\alpha}_d\}_{d=0}^N$ at complexity $\mathcal{O}(N \log^2(N))$ by using the method proposed by Townsend, et al. \cite{ToWeOl18}.
		In \verb|chebfun|, this method is available under the name \verb|cheb2leg|.
		This transformation from Chebyshev to Legendre coefficients is expected to have a worst case error growth of $\mathcal{O}(\sqrt{N}\log(N))$, and for a fast decaying set of coefficients $\{\hat{c}_d\}_{d=0}^{N}$ Townsend and collaborators have observed numerically that there is no error growth with $N$.
	\end{enumerate}
	Thus, at an overall complexity of $\mathcal{O}(N\log^2 (N))$ we are able to compute the coefficients $\{\hat{\alpha}_d\}_{d=0}^N$ representing $\tilde{f}(t)$ in the Legendre basis.
	The coefficient matrix $F^{(N)}_M = \sum_{d=0}^{N} \alpha_d B^{(d)}_M $ can be accurately approximated by $\hat{F}_{M}^{(N)} = \sum_{d=0}^N \hat{\alpha}_d B^{(d)}_M$.
	
	Second, we want to compute and store the Legendre basis matrices $B^{(d)}_M$ efficiently.
	Equation~\eqref{eq:LegBasisCoeffs} provides an expression for the entries $b_{k,\ell}^{(d)}$ in terms of integrals of the triple product of Legendre polynomials $\mathcal{F}_{a,b,c}$.
	Our approach is based on expressing the matrix $\left[\mathcal{F}_{a,b,c}\right]_{b,c=0}^{M-1}$ as
	\begin{equation*}
		\left[\mathcal{F}_{a,b,c}\right]_{b,c=0}^{M-1} = \sqrt{2a+1} \left(C_M\circ H_M\circ T_M\right),
	\end{equation*}
	with $C_M = \left[\sqrt{(2b+1)}\sqrt{(2c+1)}\right]_{b,c=0}^{M-1}$, $H_M$ a Hankel matrix and $T_M$ a Toeplitz matrix.
	The Hankel matrix depends on the value $\gamma := b+c$ and is characterized completely by its last row $r_H^\top$ and first column $c_H$.
	Let us define a function that generates the entries in $H_M$,
	\begin{equation*}
		h(a,\gamma) :=  \frac{1}{(a+\gamma+1)} \left(\prod_{j=1}^{a}\frac{-a+\gamma+2j}{-d+\gamma+2j-1}\right),
	\end{equation*}
	then, for $(m+a)$ even, the last row and first column are:
	\begin{align*}
		r_H^\top &= \begin{bmatrix}
			h(a,m) & 0 & h(a,m+2) & 0 & \dots & h(a,2m-2) & 0 & h(a,2m) 
		\end{bmatrix}^\top,\\
		c_H &= \begin{bmatrix}
			\smash[b]{\block{a}} & h(a,a) & 0 & h(a,a+2) & 0 & \dots & h(a,m-2) & 0 & h(a,m) 
		\end{bmatrix}^\top \\
		&\quad
	\end{align*}
	and, for $(m+a)$ odd:
	\begin{align*}
		r_H^\top &= \begin{bmatrix}
			h(a,m) & 0 & h(a,m+2) & 0 & \dots & 0 & h(a,2m-1) & 0 
		\end{bmatrix}^\top,\\
		c_H &= \begin{bmatrix}
			\smash[b]{\block{a}} & h(a,a) & 0 & h(a,a+2) & 0 & \dots & 0 & h(a,m-1) & 0 
		\end{bmatrix}^\top.\\
		&\quad
	\end{align*}
	The Toeplitz matrix depends on $\alpha := \vert b-c\vert$ and since the Toeplitz matrix $T_M$ is symmetric, it is characterized its first column $c_T$.
	Using the following function generating the entries of $T_M$,
	\begin{equation*}
		t(a,\alpha) := \frac{1}{2^{(2a+1/2)}} \frac{\prod_{j=(\frac{a+\alpha}{2}+1)}^{a+\alpha}j^2}{\prod_{j=1}^{\frac{a-\alpha}{2}}j^2 \prod_{j=(a-\alpha+1)}^{a+\alpha} j},
	\end{equation*}
	the first column of the Toeplitz matrix is given, for $a$ odd, by:
	\begin{equation*}
		c_T =\begin{bmatrix}
			t(a,0) & 0 & t(a,2) & 0 & \dots & t(a,a)  &	\smash[b]{\block{m-a}}\\
		\end{bmatrix}^\top
	\end{equation*}
	and, for $a$ even, by:
	\begin{align*}
		c_T = \begin{bmatrix}
			0 & t(a,1) & 0 & t(a,3) & \dots & 0 & t(a,a)  & 	\smash[b]{\block{m-a}}\\
		\end{bmatrix}^\top.\\
		\quad
	\end{align*}
	Using this expression, the matrices $\left[\mathcal{F}_{d,b,c}\right]_{b,c=0}^{M-1}$ for $d=0,1,\dots,N$ can be stored by $3(N+1)M+M^2$ numbers, instead of $(N+1)M^2$ if each matrix is stored naively.
	Further reduction of memory cost can be obtained by exploiting the zero structure of the Hankel and Toeplitz matrix.	\\
	The Legendre basis matrix of degree $d$ can now be written as
	\begin{equation}\label{eq:HadBasisMatrix}
		B_M^{(d)} = \left[b_{k,l}^{(d)}\right]_{k,l=0}^{M-1} = \sqrt{2d+1} \left(\tilde{C}_M\circ \left((\underline{H}_M\circ \underline{T}_M) Z_M\right)\right),
	\end{equation}
	where $\tilde{C}_M := \left[\frac{\sqrt{2k+1}}{\sqrt{2\ell+1}}\right]_{k,\ell=0}^{M-1}\in\mathbb{R}^{M\times M}$, $\underline{H}_M$ and $\underline{T}_M$ are, respectively, $H_{M+1}$ and $T_{M+1}$ with the last row removed and
	\begin{align*}
		Z_M &:= \begin{bmatrix}
			1 & -1 & \\
			1 & 0 & -1 &\\
			& 1 & 0 & -1 & &  \\
			&  & \ddots & \ddots & \ddots & \\
			& & & 1 & 0 & -1 \\
			&  & &  & 1 & 0\\
			& &  & & &  1
		\end{bmatrix}\in\mathbb{R}^{(M+1)\times M}.
	\end{align*}	
	The representation of the Legendre basis matrices $B_M^{(d)}$ given by Equation \eqref{eq:HadBasisMatrix} creates possibilities for the development of efficient methods for storing and computing the coefficient matrix $\hat{F}^{(N)}_M$; see, for example, the procedure proposed in \cite{ToWeOl18} for the positive semidefinite case.
	Exploring and implementing efficient memory and computational schemes further is subject of ongoing research and is out of the scope of this paper.

	\section{Proposed numerical procedure}\label{sec:NumExp}
	Using the presented results, we can replace the infinite matrix problem formulated in Problem \ref{prob:matrixProblem} by the following finite matrix problem.
	\begin{problem}[Finite matrix problem]\label{prob:matrixProblem_finite}
		Given a smooth function $\tilde{f}(t)$ and a tolerance $\delta_{\textrm{sol}}$, compute the approximate Legendre coefficients $\{\hat{c}_k\}_{k=0}^{\hat{M}}$ of the solution $\tilde{u}(t)$ to the ODE \eqref{eq:ODE_univar} such that they satisfy $\Vert \tilde{u}(t) - \sum_{d=0}^{\hat{M}}\hat{c}_d p_d(t)\Vert_{\infty} \lesssim  \delta_{\textrm{sol}}$ on the interval $t\in\left[-1,1\right]$.
		This corresponds to solving five subproblems:
		\begin{enumerate}
			\item Compute the interpolating Legendre coefficients $\{\hat{\alpha}_k\}_{k=0}^{N}$ of $\tilde{f}(t)$ for an appropriate value of $N$.
			\item Determine an appropriate value for the size $M$ of the coefficient matrix $F$ such that enough Legendre coefficients, $\hat{M}$, are computed accurately in order to reach the given tolerance, see Section \ref{sec:truncErr}.
			\item Construct the finite banded coefficient matrix $\hat{F}^{(N)}_M=\sum_{d=0}^{N} \hat{\alpha}_d B^{(d)}_M$.
			\item Solve the finite linear system of equations $(I_M-\hat{F}^{(N)}_M)\dot{x} = \phi_M(-1)$ for $\dot{x}$. The right hand side is the column vector $\phi_M(-1) = \begin{bmatrix}
				p_k(-1)
			\end{bmatrix}_{k=0}^{M-1}$ and $I_M$ is the identity matrix.
			\item Compute the finite matrix vector product $T_M \dot{x}=\hat{c}$, thereby obtaining the approximate Legendre coefficients $\{\hat{c}_k\}_{k=0}^{M}$.
		\end{enumerate}
	\end{problem}
	The procedure developed in this paper proposes to solve the subproblems in Problem \ref{prob:matrixProblem_finite} as follows:
	\begin{enumerate}
		\item Using \verb|chebfun|, the coefficients $\{\hat{\alpha}_k\}_{k=0}^{N}$ can be computed at complexity $\mathcal{O}(N\log^2(N))$, as described in Section \ref{sec:fastComp}. Moreover, for a given tolerance the value for $N$ is chosen automatically.
		\item This remains an open problem for which two possible strategies are proposed in Section \ref{sec:LegCoeffs}.
		\item Compute the sum $\hat{F}^{(N)}_M=\sum_{d=0}^{N} \hat{\alpha}_d B^{(d)}_M$. An analytical formula for the entries of $B^{(d)}_M$ is stated in Property \ref{prop:intLeg}.
		Equation \eqref{eq:HadBasisMatrix} provides a formulation for the construction of $B^{(d)}_M$ that is more memory and computationally efficient.
		\item Solve the finite system of equations $(I_M-\underline{\hat{F}}^{(N)}_M)\underline{\dot{x}} = \phi_M(-1)$, where $\underline{F}^{(N)}_M$ equals $F^{(N)}_M$ with its last $N+1$ rows set equal to zero. See Section \ref{sec:truncErr} and Section \ref{sec:LegCoeffs} for details.
		The system is solved in MATLAB by the \verb|backslash| function.
		\item Form the product $\underline{T}_{M}\underline{\dot{x}}=\hat{c}$, where $\underline{T}_M$ equals $T_M$ with its last row set to zero. 
		If requested, the Legendre series can be chopped by applying the procedure proposed by Aurentz and Trefethen \cite{AuTr17}.
		See also Section \ref{sec:LegCoeffs}.
	\end{enumerate}
	In the following, we present numerical experiments which will confirm the validity of this procedure.
	The invertibility of $(I_M-F_M^{(N)})$ is studied by its spectral properties.
	We report the numerical radius $\nu(F_M^{(N)})$ and the upper bound \eqref{eq:f:cond:decay}.
	However, these quantities are not descriptive of the observed numerical behavior, therefore we also report the pseudospectra of $(I_M-F_M^{(N)})$ which might provide a better description.
	We denote by $\sigma(A)$ the spectrum of $A$, then, for $\epsilon>0$, the $\epsilon$-pseudospectrum of $(I_M-F_M^{(N)})$ is defined by
	\begin{equation*}
		\sigma_\epsilon(I_M-F_M^{(N)}) = \{z\in\mathbb{C}\vert z\in\sigma(I_M-F_M^{(N)}+E) \text{ for some }E \text{ with }\Vert E\Vert \leq \epsilon \}.
	\end{equation*}
	See, for example, the book by Trefethen and Embree \cite{TrEm05} for details.
	
	Since, for the scalar ODEs the analytical solution $\tilde{u}(t)$ is available, we can compute the error of our approximation $\hat{u}(t):=\sum_{d=0}^M \hat{c}_d p_d(t)$. 
	This error is expressed in the maximum norm 
	\begin{equation*}
		\textrm{err}_\textrm{f} := \frac{\Vert \tilde{u}(t) - \hat{u}(t)\Vert_\infty}{\Vert \tilde{u}(t)\Vert_\infty}.
	\end{equation*}
	This error is estimated by evaluating the solution and approximation in $10M$ equidistant nodes in $t\in \left[-1,1\right]$.
	Denote by $c$ the vector of the first $M$ Legendre coefficients of $\tilde{u}(t)$, then the error on the computed Legendre coefficients $\hat{c}$ is quantified by
	\begin{equation*}
		\textrm{err}_c = \frac{\vert c-\hat{c}\vert }{\Vert c\Vert_\infty}.
	\end{equation*}

	Timings are not performed since in this paper we focus on the validity and accuracy of the procedure.
	For the scalar case studied here we do not expect to outperform state-of-the-art methods.
	However, thanks to the straightforward generalization of our procedure to the matrix case we aim to develop a procedure for the matrix ODE competitive with the state-of-the-art.
	Understanding the applicability and accuracy of the scalar case is a fundamental step towards developing a competitive procedure for the matrix case.	
	
	\subsection{Toy problem}
	The following function is constructed so that we have control over its behavior
	\begin{equation*}
		\tilde{f}(t) = -\imath \frac{\omega}{\beta} \sin(\omega (t+1)).
	\end{equation*}
	Parameter $\omega$ controls the oscillation of the function and $\beta$ controls its amplitude.
	The solution to the ODE
	\begin{equation*}
		\frac{d}{dt} \tilde{u}(t) = -\imath \frac{\omega}{\beta} \sin(\omega (t+1)) \tilde{u}(t), \quad \tilde{u}(-1) = 1,\quad \text{on } t\in \left[-1,1\right]
	\end{equation*}
	is
	\begin{equation*}
		\tilde{u}(t) = \exp\left(-\frac{\imath}{\beta} (1-\cos(\omega t+\omega))\right).
	\end{equation*}
	We report results for three different choices of parameters.
	The first choice is $\omega=5$, $\beta=10$ and we take $M=100$.
	The condition \eqref{eq:f:cond:decay} is satisfied, namely $\sum_{d=0}^N \vert \alpha_d\vert = 1.0909$ and the numerical radius is $\nu(F_{100}) = 0.2151$.
	Thus, the matrix $(I_M-F_M^{(N)})$ is nonsingular and subproblem 4 in Problem \ref{prob:matrixProblem_finite} has a unique solution.
	The approximation has accurate Legendre coefficients, $\max(\textrm{err}_c) =  1.7828\textrm{e}-15$, and a function error of $\textrm{err}_\textrm{f} = 1.3345e-15$.\\
	The second choice is $\omega=5$, $\beta=1$ for $M=100$.
	Condition \eqref{eq:f:cond:decay} is not satisfied, $\sum_{d=0}^{N} \vert \alpha_d \vert = 10.909$ and the numerical radius of the coefficient matrix is $\nu(F_{100}^{(24)}) = 2.151$, thus the existence of $(I_{100}-F_{100}^{(24)})^{-1}$ cannot be guaranteed by looking at these quantities.
	Nevertheless, it exists and the approximation obtained is accurate, $\textrm{err}_\textrm{f} = 1.8621\textrm{e}-15$ and $\max(\textrm{err}_c) =  2.5823\textrm{e}-15$.\\
	As a final choice we take a more oscillatory function $\nu=100$ and $\beta=1$ and compute an approximation for $M=1500$.
	The numerical radius $\nu(F^{(148)}_{1500}) = 45.11$ is much smaller than the upper bound $\sum_{d=0}^{N}\vert \alpha_d \vert = 796.7$.
	The approximation is accurate, $\textrm{err}_\textrm{f} = 9.9812\mathrm{e}-14$ and $\max(\textrm{err}_c) =  3.6107\mathrm{e}-14$.
	The numerical radius for the second and third choice of parameters does not guarantee the existence of $(I_{M}-F_{M}^{(N)})^{-1}$.
	Therefore, in Figure \ref{fig:pseudoSpectrum_DE}, we show the pseudospectra for several levels for $(I_M-F_M^{(N)})$ for these choices with $M=1500$.
	They indicate that even for perturbations $E$ that are relatively large in norm (up to $10^{-5}$) the spectrum of $(I_{1500}-F_{1500}^{(N)}+E)$, for both functions, remains contained in a disk centered at $(1,0)$ with a radius equal to one.
	\begin{figure}[!ht]
		\centering
		\setlength\figureheight{4.5cm}
		\setlength\figurewidth{5.4cm}
		\input{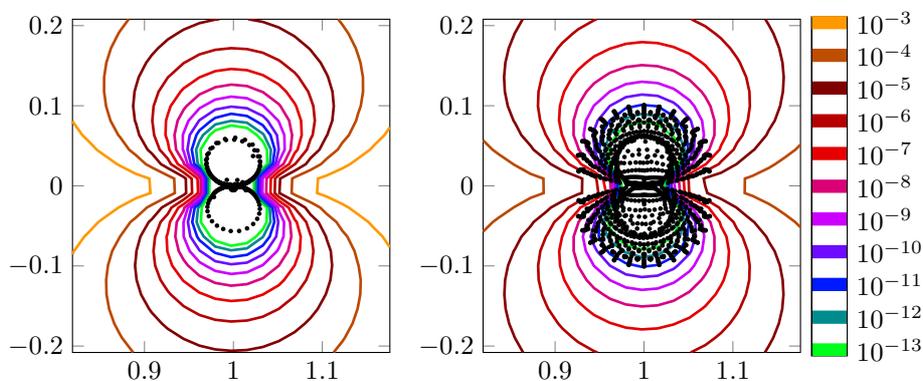}
		\caption{Spectrum ({\color{black} $\cdot$}) and pseudospectra of $(I-F_M)$ for the coefficient matrix $F_M$ of $\tilde{f}(t) = -\imath \frac{\nu}{\beta} \sin(\nu (t+1))$.
			Left: $\nu=5$, $\beta=1$ and $M=1500$. Right: $\nu=100$, $\beta=1$ and $M=1500$.}
		\label{fig:pseudoSpectrum_DE}
	\end{figure}	
	
	\subsection{A polynomial problem}
	Consider the following ODE, which appeared, e.g., in \cite{ScGhSc22},
	\begin{equation*}
		\frac{d}{d\tau} \tilde{u}(\tau) = -\imath \tau \tilde{u}(\tau), \quad  \tilde{u}(0) = 1, \quad \text{on }\tau\in\left[0,\tau_{\textrm{end}} \right]
	\end{equation*}
	the function $\tilde{f}(\tau) = -\imath \tau$ is a degree one polynomial and the solution is $\tilde{u}(\tau) = \exp(-\imath \tau^2)$.
	Since the Legendre polynomials are defined on $\left[-1,1\right]$, we perform a transformation on the ODE, mapping $\tau\in\left[0,\tau_{\textrm{end}}\right]$ onto $t\in \left[-1,1\right]$.
	This transformation is $t = \frac{2\tau}{\tau_{\textrm{end}}}-1$ and its inverse $\tau = (t+1)\frac{\tau_{\textrm{end}}}{2}$, and leads to
	\begin{equation*}
		\frac{d}{d t} \tilde{u}(t) = -\imath \left(\frac{\tau_{\textrm{end}}}{2}\right)^2 (t+1) \tilde{u}(t) ,
	\end{equation*}
	with solution $\tilde{u}(t) = \exp\left(-\frac{\imath}{2} \left(\frac{\tau_{\textrm{end}}}{2}\right)^2 (t+1) \right)$.
	The coefficient matrix $F$ of $f(t,s)=\tilde{f}(t)\Theta(t-s)$ with $\tilde{f}(t) = -\imath \left(\frac{\tau_{\textrm{end}}}{2}\right)^2 (t+1)$ is pentadiagonal and is discussed in Example \ref{example:decay}.\\
	For a fixed size of the coefficient matrix, $M=1000$, we run our procedure for $\tau_{\textrm{end}} = 25$ and $\tau_{\textrm{end}} = 50$.
	Table \ref{table:ITVOLT} shows the metrics for these cases; a good approximation is obtained in both cases even though the numerical radius is much larger than one.
	
	\begin{table}[ht]
		\centering
		\begin{tabular}{l|llll}
			$\tau_{\textrm{end}}$ & $\vert \alpha_0\vert+\vert \alpha_1\vert$ & $\nu(F_{1000}^{(1)})$ & $\mathrm{err}_c$      & $\textrm{err}_\textrm{f}$ \\ \hline
			25                    & 348.5                               & 147.6                & $5.228\mathrm{e}-14$  & $1.067\mathrm{e}-13$      \\
			50                    & 1394                                & 590.6                & $3.210\mathrm{e}-13$ & $3.008\mathrm{e}-13$    
		\end{tabular}
	\caption{Metrics for the function $\tilde{f}(t)=-\imath \left(\frac{\tau_{\textrm{end}}}{2}\right)^2 (t+1)$ and for the approximation to $\tilde{u}(t)$ obtained for $M=1000$.}
	\label{table:ITVOLT}
	\end{table}
	In Figure \ref{fig:pseudoSpectrum_ITVOLT}, the spectrum and pseudospectra for the two choices are shown.
	The eigenvalues of $(I-F_{1000}^{(N)})$ lie on a circle and, as $\tau_{\textrm{end}}$ increases, the radius of this circle increases; however, the center of the circle also shifts and the circle does not cross the real line. Thus $(I-F_{1000}^{(N)})^{-1}$ exists in both cases.
	The pseudospectra that contain the origin are those associated with large perturbations of the coefficient matrix.	
	In Figure \ref{fig:decay_ITVOLT}, the order of magnitude of the entries of the inverse for both cases is shown.
	It shows that the inverse is characterized by the decay phenomenon.
	\begin{figure}[!ht]
		\centering
		\setlength\figureheight{4.5cm}
		\setlength\figurewidth{5.4cm}
		\input{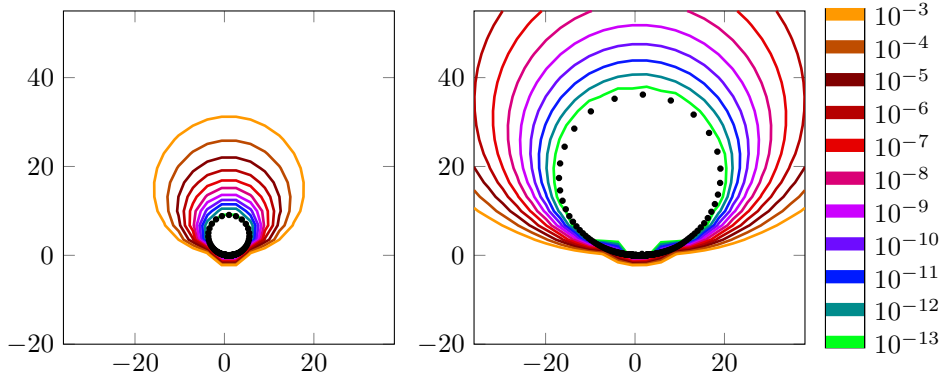}
		\caption{Spectrum ({\color{black} $\cdot$}) and pseudospectra of $(I_M-F_M^{(N)})$ for the coefficient matrix $F_M^{(N)}$ of $\tilde{f}(t) =  -\imath \left(\frac{\tau_{\textrm{end}}}{2}\right)^2 (t+1)$.
			Left: $\tau_{\textrm{end}}= 25$ and $M=1000$. Right: $\tau_{\textrm{end}}=50$ and $M=1000$.}
		\label{fig:pseudoSpectrum_ITVOLT}
	\end{figure}

	\begin{figure}[!ht]
		\centering
		\includegraphics[width=0.4\textwidth]{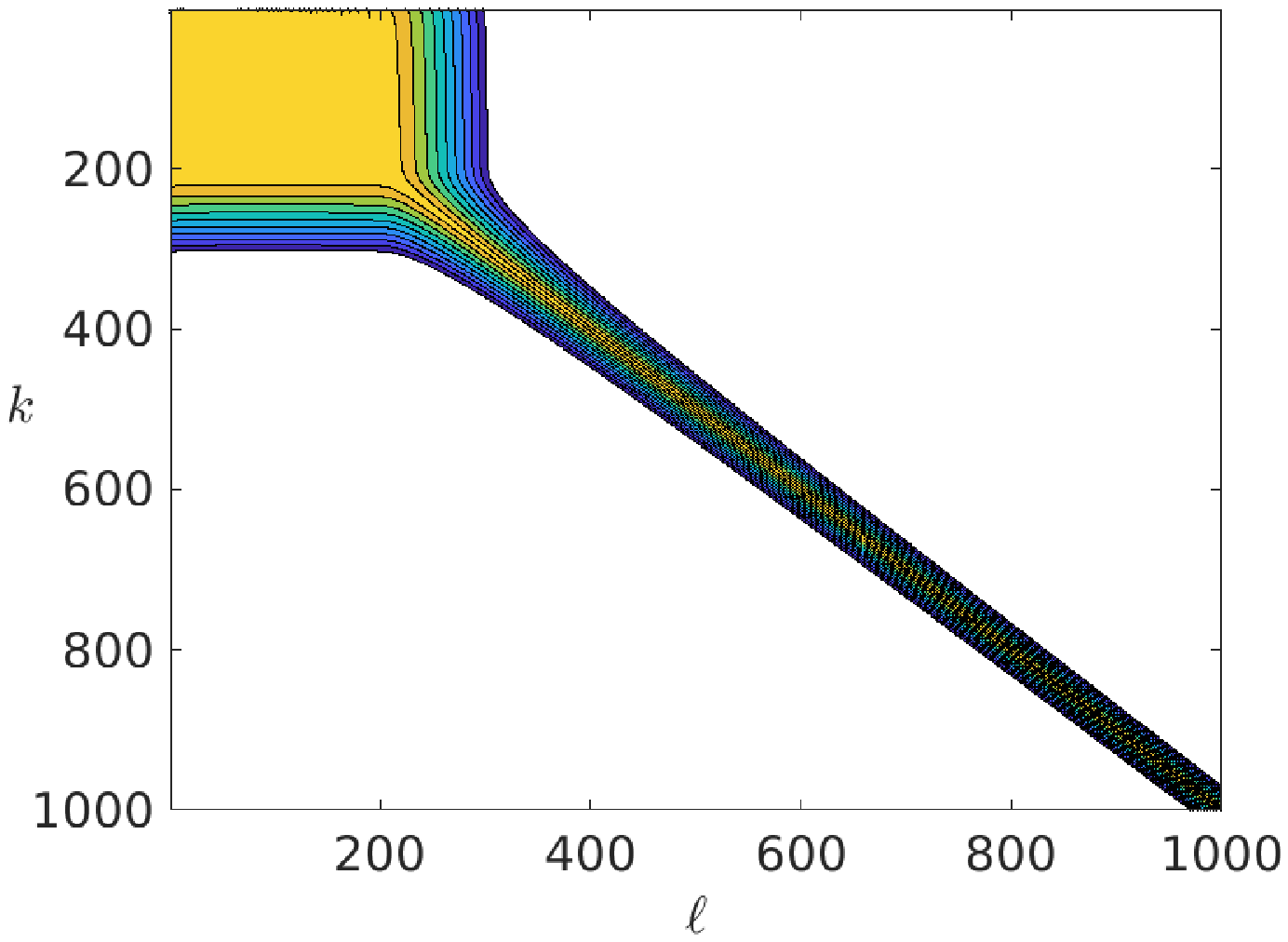}
		\includegraphics[width=0.4\textwidth]{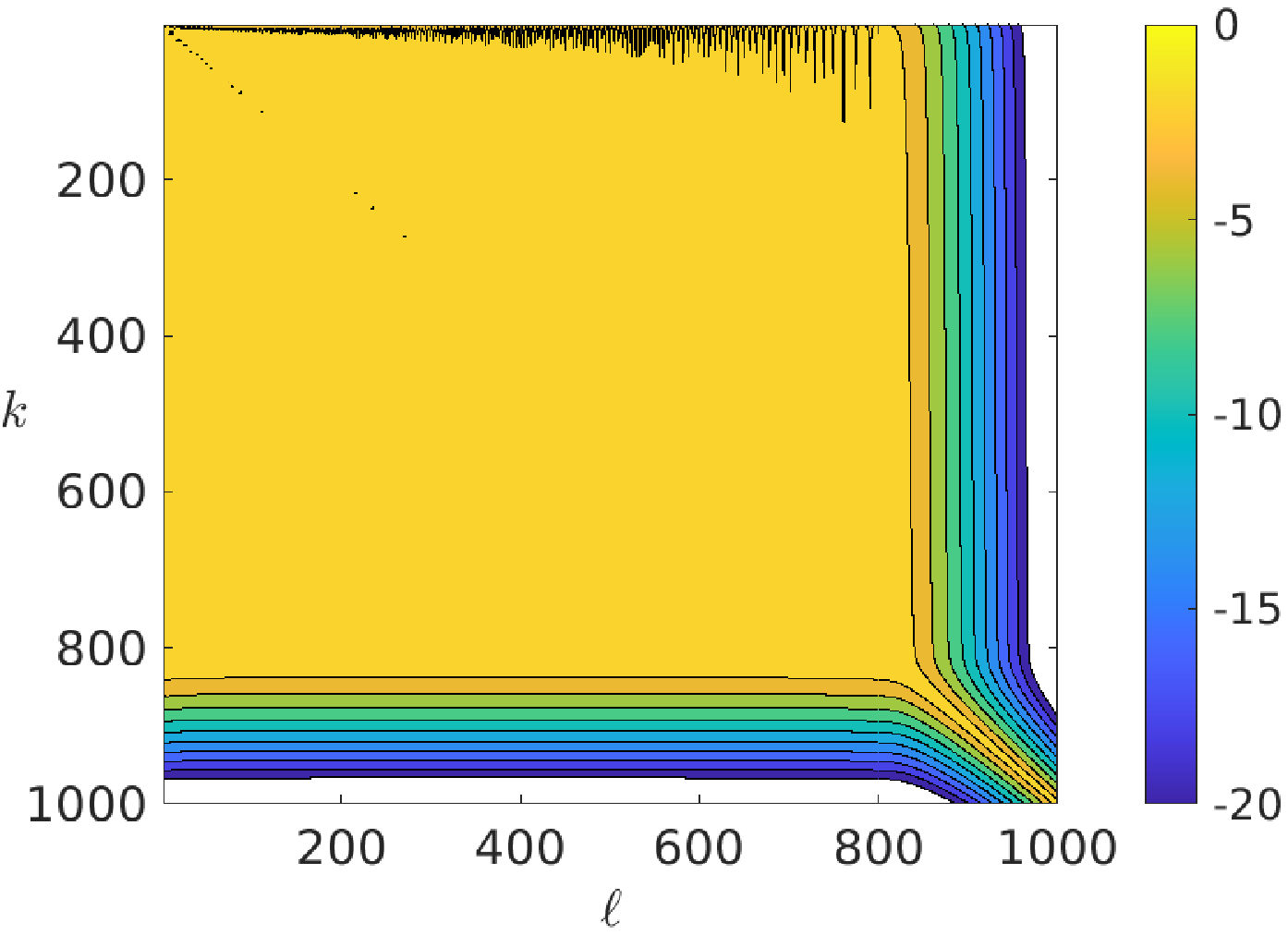}
		\caption{Order of magnitude of the entries of $(I_M-F_M^{(N)})^{-1}$ for $M=1000$, where $F_M^{(N)}$ is the coefficient matrix of $\tilde{f}(t) =  -\imath \left(\frac{\tau_{\textrm{end}}}{2}\right)^2 (t+1)$.
			Left: $\tau_{\textrm{end}}= 25$. Right: $\tau_{\textrm{end}}=50$.}
		\label{fig:decay_ITVOLT}
	\end{figure}

	Tables \ref{table:conv25} and \ref{table:conv50} show for both the choices $\tau_{\textrm{end}}=25$ and $\tau_{\textrm{end}}=50$ the function error $\textrm{err}_\textrm{f}$ and the amplitude of the last accurately computed Legendre coefficient, respectively.
	The last accurately computed Legendre coefficient for both choices is $\hat{c}_{M-1}$.
	From the tables we can conclude that, for $M$ large enough, an estimate for $\textrm{err}_\textrm{f}$ can be obtained by looking only at the Legendre coefficients, which is the expected behavior in function approximation with Legendre polynomials \cite{Tr13}.
	\begin{table}[!ht]
		\begin{minipage}[t]{0.48\textwidth}
		\centering
		\captionsetup{width=.8\linewidth}
		\begin{tabular}{l|ll}
			$M$ & $\textrm{err}_{\textrm{f}}$ & $\vert \hat{c}_{M-1}\vert$  \\ \hline
			200 & $1.8\textrm{e+}00$                & $2.7\textrm{e}-02$                            \\
			210 & $3.3\textrm{e}-01$                & $1.3\textrm{e}-02$                       \\
			220 & $1.6\textrm{e}-02$                & $1.9\textrm{e}-03$                \\
			230 & $4.6\textrm{e}-04$                & $9.0\textrm{e}-05$                                 \\
			240 & $8.0\textrm{e}-06$                & $2.2\textrm{e}-06$                                  \\
			250 & $8.5\textrm{e}-08$                & $2.9\textrm{e}-08$                        \\
			260 & $5.9\textrm{e}-10$                & $2.4\textrm{e}-10$                   \\
			270 & $2.8\textrm{e}-12$                & $1.2\textrm{e}-12$                      \\
			280 & $9.9\textrm{e}-14$                & $2.1\textrm{e}-14$                      \\
			290 & $8.4\textrm{e}-14$                & $1.2\textrm{e}-14$                     \\
			300 & $8.7\textrm{e}-14$                & $1.2\textrm{e}-14$                    
		\end{tabular}
		\caption{Error on the Legendre coefficients and the magnitude of the last accurate Legendre coefficient $\hat{c}_{M-N}$ for increasing $M$ for $\tau_{\textrm{end}} = 25$.}
		\label{table:conv25}
		\end{minipage}
		\hspace{-0.3cm}
		\begin{minipage}[t]{0.48\textwidth}
		\centering
		\captionsetup{width=.75\linewidth}
			\begin{tabular}{l|ll}
				$M$ & $\textrm{err}_{\textrm{f}}$ & $\vert \hat{c}_{M-1}\vert$ \\ \hline
				830 & $8.3\textrm{e}-02$                & $2.4\textrm{e}-03$                \\
				840 & $1.1\textrm{e}-02$                & $5.5\textrm{e}-04$                \\
				850 & $1.1\textrm{e}-03$                & $8.4\textrm{e}-05$                \\
				860 & $9.9\textrm{e}-05$                & $9.3\textrm{e}-06$                \\
				870 & $7.0\textrm{e}-06$                & $8.0\textrm{e}-07$                \\
				880 & $4.0\textrm{e}-07$                & $5.4\textrm{e}-08$                \\
				890 & $1.9\textrm{e}-08$                & $3.0\textrm{e}-09$                \\
				900 & $7.7\textrm{e}-10$                & $1.3\textrm{e}-10$                \\
				910 & $2.6\textrm{e}-11$                & $5.0\textrm{e}-12$                \\
				920 & $9.6\textrm{e}-13$                & $1.6\textrm{e}-13$                \\
				930 & $3.1\textrm{e}-13$                & $1.6\textrm{e}-14$               
			\end{tabular}
		\caption{Error on the Legendre coefficients and the magnitude of the last accurate Legendre coefficient $\hat{c}_{M-1}$ for increasing $M$ for $\tau_{\textrm{end}} = 50$.}
		\label{table:conv50}
		\end{minipage}
	\end{table}

	\subsection{NMR-inspired problem}
	The following experiment is inspired by a problem in nuclear magnetic resonance spectroscopy (NMR) where the matrix ODE
	\begin{equation*}
		\frac{d}{dt} \tilde{A}(t) = -2\imath \pi \tilde{H}(t) \tilde{A}(t), \quad \left[0,t_{\textrm{end}}\right],
	\end{equation*}
	governs the dynamics of, e.g., a magic angle spinning experiment.
	The matrix valued function $\tilde{H}(t)$ is the Hamiltonian and is of size $2^\ell\times 2^\ell$, where $\ell$ is the number of spins in the sample \cite{Le08}.
	The functions appearing in $\tilde{H}(t)$ are of the form
	\begin{equation}\label{eq:ftil_NMR}
		\tilde{f}(t) = -2 \imath \pi (\alpha + \beta \cos(2\pi \nu t) + \gamma \cos(4\pi \nu t)),
	\end{equation}
	where $\alpha\in \left[-1,1\right]$ and $\beta,\gamma\in \left[100,5000\right]$ are typical ranges for these parameters.
	In a magic angle spinning experiment \cite{HaSp98}, the sample spins at an angular velocity $\nu\in \left[5000, 120 000\right]$ chosen by the user. The experiment typically runs for about $t_{\textrm{end}}=10^{-2}$ seconds.
	Here, we consider the simpler problem of a scalar ODE 
	\begin{equation*}
		\frac{d}{dt} \tilde{u}(t) = \tilde{f}(t) \tilde{u}(t), \quad \left[0,t_{\textrm{end}}\right].
	\end{equation*}
	We lose the connection to NMR, but studying this problem provides insight into the capabilities of our proposed procedure to tackle the physically relevant matrix case.
	For $\alpha = 0.05$, $\beta=\gamma=3450$, $\nu=5000$ and $M=1500$ we compute an approximation to $\tilde{u}(t)$.
	The function approximation error is $\textrm{err}_{\textrm{f}} = 1.5994\textrm{e}-04$, increasing $M$ will improve on this error.
	However, an accuracy of the order $1\textrm{e}-4$ suffices for NMR experiments, where one is limited by the accuracy of the measurements.\\
	If the frequency is increased to $\nu=120 000$ the function and solution require more coefficients to be represented accurately, i.e., we require a larger $M$.
	Alternatively, the interval $\left[0,t_{\textrm{end}}\right]$ can be split into smaller subintervals and the ODE is then solved on each of these subintervals separately.
	Suppose we split the interval into 20 equal subintervals, then the first subproblem is the ODE on the interval $\left[0,\frac{t_{\textrm{end}}}{20}\right] = \left[0, 5 \textrm{e}{-04}\right]$.
	The approximate solution to this subproblem obtained for $M=1500$ has $\textrm{err}_\textrm{f} = 1.4101\textrm{e}-07$ and $\textrm{err}_c= 1.4087\textrm{e}-08$.
	Thus, by splitting the interval into 20 subintervals the ODE can be solved accurately for the highest frequency of interest.
	
	The spectrum and pseudospectra for both functions are shown in Figure \ref{fig:spectrum_NMR}.
	\begin{figure}[!ht]
		\centering
		\setlength\figureheight{4.5cm}
		\setlength\figurewidth{5.4cm}
		\input{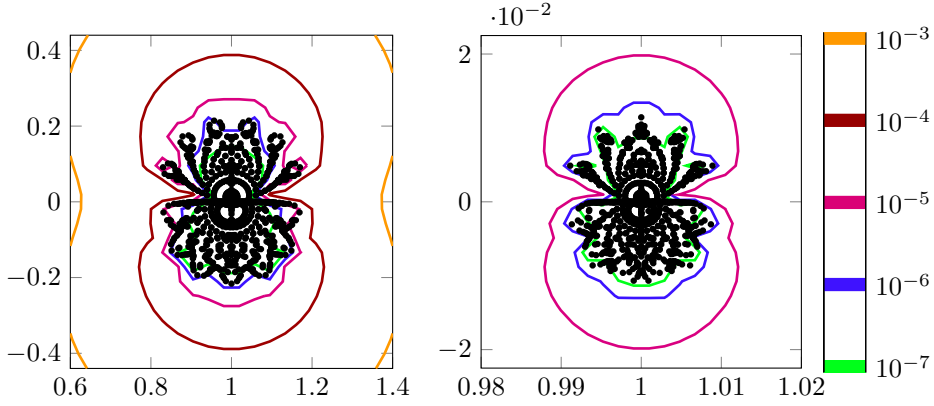}
		\caption{Spectrum ({\color{black} $\cdot$}) and pseudospectrum of $(I-F_M)$ for the coefficient matrix $F_M$ of $\tilde{f}(t)$ \eqref{eq:ftil_NMR}.
			Left: $\nu=5000$, $t_\textrm{end}=10^{-2}$ and $M=1500$. Right: $\nu=120000$, $t_\textrm{end}=\frac{1}{20}10^{-2}$ and $M=1500$.}
		\label{fig:spectrum_NMR}
	\end{figure}	
\newpage
	\section{Conclusion}
	We presented a new approach for the solution of linear non-autonomous scalar ODEs based on the discretization of the $\star$-product by using expansions in series of Legendre polynomials. This approach effectively transforms operations defined on bivariate distributions to operations in a (sub)algebra of infinite matrices.
	We studied the properties of such matrices and used them to prove the existence of the ODE solution in the infinite matrix algebra. Once the Legendre polynomial series is truncated, the ODE solution is accessible by solving a (finite) linear system. We studied the truncation error, proving that obtaining accurate approximations from this finite system is possible. We also presented effective methods to compute the discretization and tested the method on several numerical examples.
	
	The new method was numerically analyzed in the scalar case. The scalar analysis is a fundamental step toward understanding the more general case of systems of non-autonomous linear ODEs \cite{GiPo20,PoVB22}. In fact, the authors are developing a method for this more general case whose analysis and understanding will be built on the crucial results presented here. 
	
	\appendix
	\section{Proof of Lemma \ref{lemma:infNormBd}}\label{app:A}
	In this proof, it is easier to work with the formula based on Legendre polynomials $\dot{p}_k(t)$, which are normalized such that $\dot{p}_k(1)=1$.
	Then, setting $s=(a+b+c)/2$, we have
	\begin{align*}
		\dot{\mathcal{F}}_{a,b,c} &:= \int_{-1}^{1} \dot{p}_a(\tau) \dot{p}_b(\tau) \dot{p}_c(\tau) d\tau\\
		&= 
		\begin{cases}
			0\quad \text{if } a+b+c \text{ odd},\\
			0 \quad \text{if } s<\max(a,b,c),\\
			0 \quad \text{if } a<\vert b-c \vert,\\
			\frac{2}{a+b+c+1}\begin{pmatrix}
				2(s-a)\\
				s-a
			\end{pmatrix} \begin{pmatrix}
				2(s-b)\\
				s-b
			\end{pmatrix} \begin{pmatrix}
				2(s-c)\\
				s-c
			\end{pmatrix} \begin{pmatrix}
				2s\\
				s
			\end{pmatrix}^{-1} \quad \text{else}.
		\end{cases}\\
		&= \begin{cases}
			0\quad \text{if } a+b+c \text{ odd},\\
			0 \quad \text{if } b+c<a\\
			0 \quad \text{if } a<\alpha:=\vert b-c\vert,\\
			\frac{1}{2^{(2a-1)}}  {\frac{1}{(a+b+c+1)} \left(\prod_{j=1}^{a}\frac{-a+b+c+2j}{-a+b+c+2j-1}\right)} {\frac{\prod_{j=(\frac{a+\alpha}{2}+1)}^{a+\alpha}j^2}{\prod_{j=1}^{\frac{a-\alpha}{2}}j^2 \prod_{j=(a-\alpha+1)}^{a+\alpha} j}}\quad \text{else}.
		\end{cases}
	\end{align*}
	First, we need the following property.
	\begin{property}\label{prop:derivative}
		The following equality holds for $ x \in \mathbb{R}$ and $d=1,2,\dots$,
		\begin{equation*}
			\frac{\partial}{\partial x}\left( \prod_{j=1}^{d}\frac{2x+2j}{2x+2j-1}\right)
			= -2  \prod_{j=1}^{d}\frac{2x+2j}{2x+2j-1} \sum_{j=1}^{d} \frac{1}{(2x+2j)(2x+2j-1)}.
		\end{equation*}
	\end{property}
	\begin{proof}
		Follows by induction on $d$.
		For $d=1$:
		\begin{equation*}
			\frac{\partial}{\partial x} \frac{2x+2}{2x+1} = -2 \frac{2x+2}{2x+1} \frac{1}{(2x+1)(2x+2)}.
		\end{equation*}
		Assume the equality holds for $d$, consider $d+1$:
		\begin{align*}
			\frac{\partial}{\partial x} \left(\prod_{j=1}^{d+1} \frac{2x+2j}{2x+2j-1}\right) &= \frac{\partial}{\partial x} \left(\frac{2x+2d+2}{2x+2d+1} \prod_{j=1}^{d} \frac{2x+2j}{2x+2j-1}\right)\\
			&= -2 \prod_{j=1}^{d+1} \frac{2x+2j}{2x+2j-1} \sum_{j=1}^{d+1} \frac{1}{(2x+2j)(2x+2j-1)}.
		\end{align*}
		This proves the statement.
	\end{proof}
	
		\begin{lemma}[Monotonous decay along diagonals]\label{lemma:monoDecay}
		For given integers $d\geq 0$ and $k\leq d$ , the following equality is satisfied for $i=1,2,\dots$,
		\begin{equation*}
			\dot{\mathcal{F}}_{d,k,d-k} > \dot{\mathcal{F}}_{d,k+i,d-k+i} > 0.
		\end{equation*}
	\end{lemma}
	\begin{proof}
		For $d=0$ the integral of the triple product is
		\begin{equation*}
			\dot{\mathcal{F}}_{0,b,c}\int_{-1}^1 \dot{p}_0(\tau) \dot{p}_b(\tau) \dot{p}_c(\tau) d\tau = \int_{-1}^1 \dot{p}_b(\tau) \dot{p}_c(\tau) d\tau = \frac{2}{b+c+1},
		\end{equation*}
		which clearly satisfies $\dot{\mathcal{F}}_{0,0,0} > \dot{\mathcal{F}}_{0,i,i} > 0$.
		Next, we prove the statement for $d\geq 1$.
		The elements $\dot{\mathcal{F}}_{d,k+i,d-k+i}$ are clearly positive and nonzero.
		Set $\alpha =  \vert 2k-d\vert$, then
		\begin{equation*}
			\dot{\mathcal{F}}_{d,k+i,d-k+i} = \frac{1}{2^{(2d-1)}} \frac{1}{d+2i+1} \prod_{j=1}^{d}\frac{2i+2j}{2i+2j-1} \frac{\prod_{j=\frac{d+\alpha}{2}+1}^{d+\alpha}j^2}{\prod_{j=1}^{\frac{d-\alpha}{2}}j^2 \prod_{j=d-\alpha+1}^{d+\alpha}j}.
		\end{equation*}
		The term $C(d,\alpha):= \frac{1}{2^{(2d-1)}}\frac{\prod_{j=\frac{d+\alpha}{2}+1}^{d+\alpha}j^2}{\prod_{j=1}^{\frac{d-\alpha}{2}}j^2 \prod_{j=d-\alpha+1}^{d+\alpha}j}$ is independent of $i$.
		To show that the expression decreases as $i$ increases, we take the derivative with respect to $0\leq x\in\mathbb{R}$ and use Property \ref{prop:derivative}
		\begin{align*}
			&\frac{\partial}{\partial x}\dot{\mathcal{F}}_{d,k+x,d-k+x} = C(d,\alpha)  \frac{\partial}{\partial x}\left(\frac{1}{d+2x+1} \prod_{j=1}^{d}\frac{2x+2j}{2x+2j-1}\right)\\
			& = -\frac{2 C(d,\alpha)}{d+2x+1}   \prod_{j=1}^{d}\frac{2x+2j}{2x+2j-1} \left(\frac{1}{d+2x+1}+\sum_{j=1}^{d} \frac{1}{(2x+2j)(2x+2j-1)}\right)\\
			&< 0.
		\end{align*}
		Since $\dot{\mathcal{F}}> 0$ for $i=0,1,\dots$, replacing $x$ with integers $i\geq 0$ proves the statement.
	\end{proof}

 \begin{lemma}\label{lemma:max:F}
 For given integers $d\geq 0$ and $k \leq d$, 
 \begin{equation*}
     \dot{\mathcal{F}}_{d,k,d-k} \leq \frac{2}{2d + 1}.
 \end{equation*}
 \end{lemma}
 \begin{proof}
   In order to prove this lemma it is sufficient to show that
   \begin{equation*}
   \max_{k=0,\dots,d}
       \begin{pmatrix}
				2(d-k)\\
				d-k
			\end{pmatrix} \begin{pmatrix}
				2k\\
				k
			\end{pmatrix} 
   = \begin{pmatrix}
				2d\\
				d
			\end{pmatrix}.
   \end{equation*}
   Note that, because of symmetry, this is equivalent to showing that, for $d/2 \leq k < d$ and $d>0$, it holds that
      \begin{equation*}
          \begin{pmatrix}
				2(d-k)\\
				d-k
			\end{pmatrix} \begin{pmatrix}
				2k\\
				k
			\end{pmatrix} 
   \geq        \begin{pmatrix}
				2(d-k) -2\\
				d-k -1
			\end{pmatrix} \begin{pmatrix}
				2k-2\\
				k-1
			\end{pmatrix}.
   \end{equation*}
   The equation above can be reformulated as the quadratic problem
   \begin{equation*}
       4(2d-2k-1)(2k-1) \geq k(d-k).
   \end{equation*}
   For $d\geq2$, the inequality above is satisfied for every $d/2 \leq k < d$. The proof is then concluded since the cases $d=0, 1$ are trivial.
 \end{proof}

	\begin{proof}[Proof of Lemma \ref{lemma:infNormBd}]
		The coefficients of $B^{(d)} = \left[b_{k,\ell}^{(d)} \right]_{k,\ell=0}^\infty$ are given by the formula
		\begin{equation*}
			b_{k,\ell}^{(d)} = \frac{\sqrt{(2d+1)(2k+1)}}{\sqrt{8} \sqrt{2l+1}} \left( \dot{\mathcal{F}}_{d,k,\ell+1} - \dot{\mathcal{F}}_{d,k,\ell-1}\right).
		\end{equation*}
		Then, by the definition of the infinity norm, we have
		\begin{align*}
			\Vert B^{(d)} \Vert_\infty &= \max_{k\geq 0} \sum_{\ell = 0}^{\infty} \vert b_{k,\ell}^{(d)} \vert = \frac{\sqrt{2d+1}}{\sqrt{8}} \max_{k\geq 0} \sum_{\ell=0}^{\infty} \left( \frac{\sqrt{2k+1}}{\sqrt{2\ell+1}} \left\vert \dot{\mathcal{F}}_{d,k,\ell+1} - \dot{\mathcal{F}}_{d,k,\ell-1} \right\vert \right) \\
			&\leq \frac{\sqrt{2d+1}}{\sqrt{8}} \max_{k\geq 0} \sum_{\ell=0}^{\infty} \left( \frac{\sqrt{2k+1}}{\sqrt{2\ell+1}} \left\vert \dot{\mathcal{F}}_{d,k,\ell+1} \right\vert +  \left\vert \dot{\mathcal{F}}_{d,k,\ell-1} \right\vert \right) \\
			&= \frac{\sqrt{2d+1}}{\sqrt{8}} \left(\max_{k\geq 0} \sum_{\ell=0}^{\infty}  \frac{\sqrt{2k+1}}{\sqrt{2\ell+1}} \left\vert  \dot{\mathcal{F}}_{d,k,\ell+1} \right\vert + \max_{k\geq 0} \sum_{\ell=0}^{\infty}  \frac{\sqrt{2k+1}}{\sqrt{2\ell+1}}  \left\vert \dot{\mathcal{F}}_{d,k,\ell-1} \right\vert \right).
		\end{align*}
		Consider the first term and use Lemma \ref{lemma:monoDecay}, which implies that $\dot{\mathcal{F}}_{d,k+i,d-k+i}$ for any $i\geq 0$ can be bounded from above by $\dot{\mathcal{F}}_{d,k,d-k}$.
		We can also bound the term $\frac{\sqrt{2k+1}}{\sqrt{2\ell+1}}$:
		\begin{align*}
			\underset{\ell = d-k}{\max_{0\leq k\leq d}} \frac{2k+1}{2\ell +1} = \max_{0\leq k\leq d} \left(\frac{2d+2}{2d-(2k-1)}\right) -1 = \frac{2d+2}{2d -(2d-1)} -1 = 2d+1.
		\end{align*}
		Since $\dot{\mathcal{F}}_{d,k,\ell+1} = 0$ for $\vert k-\ell-1\vert >d$, the infinite sum can be bounded by a finite sum:
		\begin{equation*}
			\max_{k\geq 0} \sum_{\ell=0}^{\infty}  \frac{\sqrt{2k+1}}{\sqrt{2\ell+1}} \left\vert  \dot{\mathcal{F}}_{d,k,\ell+1} \right\vert \leq \sqrt{2d+1}  \sum_{k=0}^{d} \dot{\mathcal{F}}_{d,k,d-k}.
		\end{equation*}
		Since $\vert \dot{\mathcal{F}}_{d,d,-1}\vert = \vert \dot{\mathcal{F}}_{d,d,0}\vert$, similar arguments lead to the following bound for the second term,
		\begin{equation*}
			\max_{k\geq 0} \sum_{\ell= 0}^{\infty} \frac{\sqrt{2k+1}}{\sqrt{2\ell+1}} \vert \dot{\mathcal{F}}_{d,k,\ell-1} \vert \leq \sqrt{2d+1} \dot{\mathcal{F}}_{d,d,0} + \sqrt{2d+1} \sum_{k=0}^{d} \dot{\mathcal{F}}_{d,k,d-k}.
		\end{equation*}
		Hence, we obtain the bound
		\begin{align*}
			\Vert B^{(d)} \Vert_\infty  \leq \frac{2d+1}{\sqrt{8}} \left( \dot{\mathcal{F}}_{d,d,0} + 2 \sum_{k=0}^{d} \dot{\mathcal{F}}_{d,k,d-k} \right).
		\end{align*}
		Now we plug in the formula for $\dot{\mathcal{F}}_{d,k,d-k}$, note that $d+k+d-k = 2d$ and $-d+k+d-k = 0$ and let $\alpha(k) := \vert 2k-d\vert$, then 
 \begin{align*}
  			\Vert B^{(d)} \Vert_\infty  &\leq \frac{2d+1}{\sqrt{8}} \left( \dot{\mathcal{F}}_{d,d,0} + 2 \sum_{k=0}^{d} \dot{\mathcal{F}}_{d,k,d-k} \right) \leq \frac{2d+1}{\sqrt{8}} \left( 1 + 2 \sum_{k=0}^{d} \frac{2}{2d+1} \right)\\
			&\leq \frac{2d+1}{\sqrt{8}} \left( 1 + \frac{4(d+1)}{2d+1} \right)  \leq \frac{6d+5}{\sqrt{8}} < 3d+2,
 \end{align*} 
 which proves Lemma~\ref{lemma:infNormBd}. 
	\end{proof}

	\section*{Acknowledgments}

	The second author thanks Marcus Webb for discussion in an early stage of this project.

	\section*{Funding}

	This  work  was  supported  by  Charles  University  Research programs PRIMUS/21/SCI/009 and UNCE/SCI/023 and by the Magica project ANR-20-CE29-0007 funded by the French National Research Agency.

	\bibliographystyle{siam}

\input{references.bbl}
\end{document}

%% file: figs/ITVOLT_contour_tend=4.tikz
	\begin{tikzpicture}
	\begin{axis}[
		width=0.95\figurewidth,
		height=\figureheight,
		at={(0\figurewidth,0\figureheight)},
		xlabel = {$k$},
		ylabel = {$\ell$},
		view={90}{90},shader=interp,
		colormap={parula}{
			rgb255=(53,42,135)
			rgb255=(15,92,221)
			rgb255=(18,125,216)
			rgb255=(7,156,207)
			rgb255=(21,177,180)
			rgb255=(89,189,140)
			rgb255=(165,190,107)
			rgb255=(225,185,82)
			rgb255=(252,206,46)
			rgb255=(249,251,14)},
		colorbar,
		colorbar style={
			width=0.2cm},
		point meta min = -4,
		point meta max = 0,
		]
		\addplot3[
		contour filled={
			levels={-3,-2.2,...,0},
		},
		colormap access=piecewise const,
		patch type=bilinear
		]
		table {figs/contour_ITVOLT_m=50tend=4.dat};
	\end{axis}
	
 \end{tikzpicture}

%% file: figs/diagdecay_ITVOLT.tikz
%
\begin{tikzpicture}

\begin{axis}[%
width=0.951\figurewidth,
height=\figureheight,
at={(0\figurewidth,0\figureheight)},
scale only axis,
unbounded coords=jump,
xmode=log,
xmin=1,
xmax=1e+03,
xminorticks=true,
xlabel style={font=\color{white!15!black}},
xlabel={$\ell$},
ymode=log,
ymin=0.001,
ymax=10,
yminorticks=true,
axis background/.style={fill=white},
legend style={legend cell align=left, align=left, draw=white!15!black}
]
\addplot [color=blue, draw=none, mark=asterisk, mark options={solid, blue}]
  table[row sep=crcr]{%
0	2.3\\
1	1\\
2	0.68\\
3	0.5\\
4	0.4\\
5	0.33\\
6	0.29\\
7	0.25\\
8	0.22\\
9	0.2\\
10	0.18\\
11	0.17\\
12	0.15\\
13	0.14\\
14	0.13\\
15	0.13\\
16	0.12\\
17	0.11\\
18	0.11\\
19	0.1\\
20	0.095\\
21	0.091\\
22	0.087\\
23	0.083\\
24	0.08\\
25	0.077\\
26	0.074\\
27	0.071\\
28	0.069\\
29	0.067\\
30	0.065\\
31	0.063\\
32	0.061\\
33	0.059\\
34	0.057\\
35	0.056\\
36	0.054\\
37	0.053\\
38	0.051\\
39	0.05\\
40	0.049\\
41	0.048\\
42	0.047\\
43	0.045\\
44	0.044\\
45	0.043\\
46	0.043\\
47	0.042\\
48	0.041\\
49	0.04\\
50	0.039\\
51	0.038\\
52	0.038\\
53	0.037\\
54	0.036\\
55	0.036\\
56	0.035\\
57	0.034\\
58	0.034\\
59	0.033\\
60	0.033\\
61	0.032\\
62	0.032\\
63	0.031\\
64	0.031\\
65	0.03\\
66	0.03\\
67	0.029\\
68	0.029\\
69	0.029\\
70	0.028\\
71	0.028\\
72	0.027\\
73	0.027\\
74	0.027\\
75	0.026\\
76	0.026\\
77	0.026\\
78	0.025\\
79	0.025\\
80	0.025\\
81	0.024\\
82	0.024\\
83	0.024\\
84	0.024\\
85	0.023\\
86	0.023\\
87	0.023\\
88	0.022\\
89	0.022\\
90	0.022\\
91	0.022\\
92	0.022\\
93	0.021\\
94	0.021\\
95	0.021\\
96	0.021\\
97	0.02\\
98	0.02\\
99	0.02\\
1e+02	0.02\\
1e+02	0.02\\
1e+02	0.019\\
1e+02	0.019\\
1e+02	0.019\\
1e+02	0.019\\
1.1e+02	0.019\\
1.1e+02	0.019\\
1.1e+02	0.018\\
1.1e+02	0.018\\
1.1e+02	0.018\\
1.1e+02	0.018\\
1.1e+02	0.018\\
1.1e+02	0.018\\
1.1e+02	0.017\\
1.2e+02	0.017\\
1.2e+02	0.017\\
1.2e+02	0.017\\
1.2e+02	0.017\\
1.2e+02	0.017\\
1.2e+02	0.017\\
1.2e+02	0.016\\
1.2e+02	0.016\\
1.2e+02	0.016\\
1.2e+02	0.016\\
1.2e+02	0.016\\
1.3e+02	0.016\\
1.3e+02	0.016\\
1.3e+02	0.016\\
1.3e+02	0.015\\
1.3e+02	0.015\\
1.3e+02	0.015\\
1.3e+02	0.015\\
1.3e+02	0.015\\
1.3e+02	0.015\\
1.4e+02	0.015\\
1.4e+02	0.015\\
1.4e+02	0.014\\
1.4e+02	0.014\\
1.4e+02	0.014\\
1.4e+02	0.014\\
1.4e+02	0.014\\
1.4e+02	0.014\\
1.4e+02	0.014\\
1.4e+02	0.014\\
1.4e+02	0.014\\
1.5e+02	0.014\\
1.5e+02	0.014\\
1.5e+02	0.013\\
1.5e+02	0.013\\
1.5e+02	0.013\\
1.5e+02	0.013\\
1.5e+02	0.013\\
1.5e+02	0.013\\
1.5e+02	0.013\\
1.6e+02	0.013\\
1.6e+02	0.013\\
1.6e+02	0.013\\
1.6e+02	0.013\\
1.6e+02	0.013\\
1.6e+02	0.012\\
1.6e+02	0.012\\
1.6e+02	0.012\\
1.6e+02	0.012\\
1.6e+02	0.012\\
1.6e+02	0.012\\
1.7e+02	0.012\\
1.7e+02	0.012\\
1.7e+02	0.012\\
1.7e+02	0.012\\
1.7e+02	0.012\\
1.7e+02	0.012\\
1.7e+02	0.012\\
1.7e+02	0.011\\
1.7e+02	0.011\\
1.8e+02	0.011\\
1.8e+02	0.011\\
1.8e+02	0.011\\
1.8e+02	0.011\\
1.8e+02	0.011\\
1.8e+02	0.011\\
1.8e+02	0.011\\
1.8e+02	0.011\\
1.8e+02	0.011\\
1.8e+02	0.011\\
1.8e+02	0.011\\
1.9e+02	0.011\\
1.9e+02	0.011\\
1.9e+02	0.011\\
1.9e+02	0.011\\
1.9e+02	0.01\\
1.9e+02	0.01\\
1.9e+02	0.01\\
1.9e+02	0.01\\
1.9e+02	0.01\\
2e+02	0.01\\
2e+02	0.01\\
2e+02	0.01\\
2e+02	0.01\\
2e+02	0.01\\
2e+02	0.01\\
2e+02	0.0099\\
2e+02	0.0099\\
2e+02	0.0098\\
2e+02	0.0098\\
2e+02	0.0097\\
2.1e+02	0.0097\\
2.1e+02	0.0096\\
2.1e+02	0.0096\\
2.1e+02	0.0095\\
2.1e+02	0.0095\\
2.1e+02	0.0094\\
2.1e+02	0.0094\\
2.1e+02	0.0093\\
2.1e+02	0.0093\\
2.2e+02	0.0093\\
2.2e+02	0.0092\\
2.2e+02	0.0092\\
2.2e+02	0.0091\\
2.2e+02	0.0091\\
2.2e+02	0.009\\
2.2e+02	0.009\\
2.2e+02	0.009\\
2.2e+02	0.0089\\
2.2e+02	0.0089\\
2.2e+02	0.0088\\
2.3e+02	0.0088\\
2.3e+02	0.0088\\
2.3e+02	0.0087\\
2.3e+02	0.0087\\
2.3e+02	0.0087\\
2.3e+02	0.0086\\
2.3e+02	0.0086\\
2.3e+02	0.0085\\
2.3e+02	0.0085\\
2.4e+02	0.0085\\
2.4e+02	0.0084\\
2.4e+02	0.0084\\
2.4e+02	0.0084\\
2.4e+02	0.0083\\
2.4e+02	0.0083\\
2.4e+02	0.0083\\
2.4e+02	0.0082\\
2.4e+02	0.0082\\
2.4e+02	0.0082\\
2.4e+02	0.0081\\
2.5e+02	0.0081\\
2.5e+02	0.0081\\
2.5e+02	0.008\\
2.5e+02	0.008\\
2.5e+02	0.008\\
2.5e+02	0.0079\\
2.5e+02	0.0079\\
2.5e+02	0.0079\\
2.5e+02	0.0078\\
2.6e+02	0.0078\\
2.6e+02	0.0078\\
2.6e+02	0.0078\\
2.6e+02	0.0077\\
2.6e+02	0.0077\\
2.6e+02	0.0077\\
2.6e+02	0.0076\\
2.6e+02	0.0076\\
2.6e+02	0.0076\\
2.6e+02	0.0075\\
2.6e+02	0.0075\\
2.7e+02	0.0075\\
2.7e+02	0.0075\\
2.7e+02	0.0074\\
2.7e+02	0.0074\\
2.7e+02	0.0074\\
2.7e+02	0.0074\\
2.7e+02	0.0073\\
2.7e+02	0.0073\\
2.7e+02	0.0073\\
2.8e+02	0.0072\\
2.8e+02	0.0072\\
2.8e+02	0.0072\\
2.8e+02	0.0072\\
2.8e+02	0.0071\\
2.8e+02	0.0071\\
2.8e+02	0.0071\\
2.8e+02	0.0071\\
2.8e+02	0.007\\
2.8e+02	0.007\\
2.8e+02	0.007\\
2.9e+02	0.007\\
2.9e+02	0.0069\\
2.9e+02	0.0069\\
2.9e+02	0.0069\\
2.9e+02	0.0069\\
2.9e+02	0.0068\\
2.9e+02	0.0068\\
2.9e+02	0.0068\\
2.9e+02	0.0068\\
3e+02	0.0068\\
3e+02	0.0067\\
3e+02	0.0067\\
3e+02	0.0067\\
3e+02	0.0067\\
3e+02	0.0066\\
3e+02	0.0066\\
3e+02	0.0066\\
3e+02	0.0066\\
3e+02	0.0066\\
3e+02	0.0065\\
3.1e+02	0.0065\\
3.1e+02	0.0065\\
3.1e+02	0.0065\\
3.1e+02	0.0065\\
3.1e+02	0.0064\\
3.1e+02	0.0064\\
3.1e+02	0.0064\\
3.1e+02	0.0064\\
3.1e+02	0.0063\\
3.2e+02	0.0063\\
3.2e+02	0.0063\\
3.2e+02	0.0063\\
3.2e+02	0.0063\\
3.2e+02	0.0063\\
3.2e+02	0.0062\\
3.2e+02	0.0062\\
3.2e+02	0.0062\\
3.2e+02	0.0062\\
3.2e+02	0.0062\\
3.2e+02	0.0061\\
3.3e+02	0.0061\\
3.3e+02	0.0061\\
3.3e+02	0.0061\\
3.3e+02	0.0061\\
3.3e+02	0.006\\
3.3e+02	0.006\\
3.3e+02	0.006\\
3.3e+02	0.006\\
3.3e+02	0.006\\
3.4e+02	0.006\\
3.4e+02	0.0059\\
3.4e+02	0.0059\\
3.4e+02	0.0059\\
3.4e+02	0.0059\\
3.4e+02	0.0059\\
3.4e+02	0.0058\\
3.4e+02	0.0058\\
3.4e+02	0.0058\\
3.4e+02	0.0058\\
3.4e+02	0.0058\\
3.5e+02	0.0058\\
3.5e+02	0.0057\\
3.5e+02	0.0057\\
3.5e+02	0.0057\\
3.5e+02	0.0057\\
3.5e+02	0.0057\\
3.5e+02	0.0057\\
3.5e+02	0.0056\\
3.5e+02	0.0056\\
3.6e+02	0.0056\\
3.6e+02	0.0056\\
3.6e+02	0.0056\\
3.6e+02	0.0056\\
3.6e+02	0.0056\\
3.6e+02	0.0055\\
3.6e+02	0.0055\\
3.6e+02	0.0055\\
3.6e+02	0.0055\\
3.6e+02	0.0055\\
3.6e+02	0.0055\\
3.7e+02	0.0054\\
3.7e+02	0.0054\\
3.7e+02	0.0054\\
3.7e+02	0.0054\\
3.7e+02	0.0054\\
3.7e+02	0.0054\\
3.7e+02	0.0054\\
3.7e+02	0.0053\\
3.7e+02	0.0053\\
3.8e+02	0.0053\\
3.8e+02	0.0053\\
3.8e+02	0.0053\\
3.8e+02	0.0053\\
3.8e+02	0.0053\\
3.8e+02	0.0052\\
3.8e+02	0.0052\\
3.8e+02	0.0052\\
3.8e+02	0.0052\\
3.8e+02	0.0052\\
3.8e+02	0.0052\\
3.9e+02	0.0052\\
3.9e+02	0.0052\\
3.9e+02	0.0051\\
3.9e+02	0.0051\\
3.9e+02	0.0051\\
3.9e+02	0.0051\\
3.9e+02	0.0051\\
3.9e+02	0.0051\\
3.9e+02	0.0051\\
4e+02	0.0051\\
4e+02	0.005\\
4e+02	0.005\\
4e+02	0.005\\
4e+02	0.005\\
4e+02	0.005\\
4e+02	0.005\\
4e+02	0.005\\
4e+02	0.005\\
4e+02	0.0049\\
4e+02	0.0049\\
4.1e+02	0.0049\\
4.1e+02	0.0049\\
4.1e+02	0.0049\\
4.1e+02	0.0049\\
4.1e+02	0.0049\\
4.1e+02	0.0049\\
4.1e+02	0.0048\\
4.1e+02	0.0048\\
4.1e+02	0.0048\\
4.2e+02	0.0048\\
4.2e+02	0.0048\\
4.2e+02	0.0048\\
4.2e+02	0.0048\\
4.2e+02	0.0048\\
4.2e+02	0.0048\\
4.2e+02	0.0047\\
4.2e+02	0.0047\\
4.2e+02	0.0047\\
4.2e+02	0.0047\\
4.2e+02	0.0047\\
4.3e+02	0.0047\\
4.3e+02	0.0047\\
4.3e+02	0.0047\\
4.3e+02	0.0047\\
4.3e+02	0.0046\\
4.3e+02	0.0046\\
4.3e+02	0.0046\\
4.3e+02	0.0046\\
4.3e+02	0.0046\\
4.4e+02	0.0046\\
4.4e+02	0.0046\\
4.4e+02	0.0046\\
4.4e+02	0.0046\\
4.4e+02	0.0045\\
4.4e+02	0.0045\\
4.4e+02	0.0045\\
4.4e+02	0.0045\\
4.4e+02	0.0045\\
4.4e+02	0.0045\\
4.4e+02	0.0045\\
4.5e+02	0.0045\\
4.5e+02	0.0045\\
4.5e+02	0.0045\\
4.5e+02	0.0044\\
4.5e+02	0.0044\\
4.5e+02	0.0044\\
4.5e+02	0.0044\\
4.5e+02	0.0044\\
4.5e+02	0.0044\\
4.6e+02	0.0044\\
4.6e+02	0.0044\\
4.6e+02	0.0044\\
4.6e+02	0.0044\\
4.6e+02	0.0043\\
4.6e+02	0.0043\\
4.6e+02	0.0043\\
4.6e+02	0.0043\\
4.6e+02	0.0043\\
4.6e+02	0.0043\\
4.6e+02	0.0043\\
4.7e+02	0.0043\\
4.7e+02	0.0043\\
4.7e+02	0.0043\\
4.7e+02	0.0043\\
4.7e+02	0.0042\\
4.7e+02	0.0042\\
4.7e+02	0.0042\\
4.7e+02	0.0042\\
4.7e+02	0.0042\\
4.8e+02	0.0042\\
4.8e+02	0.0042\\
4.8e+02	0.0042\\
4.8e+02	0.0042\\
4.8e+02	0.0042\\
4.8e+02	0.0042\\
4.8e+02	0.0041\\
4.8e+02	0.0041\\
4.8e+02	0.0041\\
4.8e+02	0.0041\\
4.8e+02	0.0041\\
4.9e+02	0.0041\\
4.9e+02	0.0041\\
4.9e+02	0.0041\\
4.9e+02	0.0041\\
4.9e+02	0.0041\\
4.9e+02	0.0041\\
4.9e+02	0.0041\\
4.9e+02	0.004\\
4.9e+02	0.004\\
5e+02	0.004\\
5e+02	0.004\\
5e+02	0.004\\
5e+02	0.004\\
5e+02	0.004\\
5e+02	0.004\\
5e+02	0.004\\
5e+02	0.004\\
5e+02	0.004\\
5e+02	0.004\\
5e+02	0.004\\
5.1e+02	0.0039\\
5.1e+02	0.0039\\
5.1e+02	0.0039\\
5.1e+02	0.0039\\
5.1e+02	0.0039\\
5.1e+02	0.0039\\
5.1e+02	0.0039\\
5.1e+02	0.0039\\
5.1e+02	0.0039\\
5.2e+02	0.0039\\
5.2e+02	0.0039\\
5.2e+02	0.0039\\
5.2e+02	0.0039\\
5.2e+02	0.0038\\
5.2e+02	0.0038\\
5.2e+02	0.0038\\
5.2e+02	0.0038\\
5.2e+02	0.0038\\
5.2e+02	0.0038\\
5.2e+02	0.0038\\
5.3e+02	0.0038\\
5.3e+02	0.0038\\
5.3e+02	0.0038\\
5.3e+02	0.0038\\
5.3e+02	0.0038\\
5.3e+02	0.0038\\
5.3e+02	0.0038\\
5.3e+02	0.0037\\
5.3e+02	0.0037\\
5.4e+02	0.0037\\
5.4e+02	0.0037\\
5.4e+02	0.0037\\
5.4e+02	0.0037\\
5.4e+02	0.0037\\
5.4e+02	0.0037\\
5.4e+02	0.0037\\
5.4e+02	0.0037\\
5.4e+02	0.0037\\
5.4e+02	0.0037\\
5.4e+02	0.0037\\
5.5e+02	0.0037\\
5.5e+02	0.0036\\
5.5e+02	0.0036\\
5.5e+02	0.0036\\
5.5e+02	0.0036\\
5.5e+02	0.0036\\
5.5e+02	0.0036\\
5.5e+02	0.0036\\
5.5e+02	0.0036\\
5.6e+02	0.0036\\
5.6e+02	0.0036\\
5.6e+02	0.0036\\
5.6e+02	0.0036\\
5.6e+02	0.0036\\
5.6e+02	0.0036\\
5.6e+02	0.0036\\
5.6e+02	0.0036\\
5.6e+02	0.0035\\
5.6e+02	0.0035\\
5.6e+02	0.0035\\
5.7e+02	0.0035\\
5.7e+02	0.0035\\
5.7e+02	0.0035\\
5.7e+02	0.0035\\
5.7e+02	0.0035\\
5.7e+02	0.0035\\
5.7e+02	0.0035\\
5.7e+02	0.0035\\
5.7e+02	0.0035\\
5.8e+02	0.0035\\
5.8e+02	0.0035\\
5.8e+02	0.0035\\
5.8e+02	0.0035\\
5.8e+02	0.0034\\
5.8e+02	0.0034\\
5.8e+02	0.0034\\
5.8e+02	0.0034\\
5.8e+02	0.0034\\
5.8e+02	0.0034\\
5.8e+02	0.0034\\
5.9e+02	0.0034\\
5.9e+02	0.0034\\
5.9e+02	0.0034\\
5.9e+02	0.0034\\
5.9e+02	0.0034\\
5.9e+02	0.0034\\
5.9e+02	0.0034\\
5.9e+02	0.0034\\
5.9e+02	0.0034\\
6e+02	0.0034\\
6e+02	0.0034\\
6e+02	0.0033\\
6e+02	0.0033\\
6e+02	0.0033\\
6e+02	0.0033\\
6e+02	0.0033\\
6e+02	0.0033\\
6e+02	0.0033\\
6e+02	0.0033\\
6e+02	0.0033\\
6.1e+02	0.0033\\
6.1e+02	0.0033\\
6.1e+02	0.0033\\
6.1e+02	0.0033\\
6.1e+02	0.0033\\
6.1e+02	0.0033\\
6.1e+02	0.0033\\
6.1e+02	0.0033\\
6.1e+02	0.0033\\
6.2e+02	0.0032\\
6.2e+02	0.0032\\
6.2e+02	0.0032\\
6.2e+02	0.0032\\
6.2e+02	0.0032\\
6.2e+02	0.0032\\
6.2e+02	0.0032\\
6.2e+02	0.0032\\
6.2e+02	0.0032\\
6.2e+02	0.0032\\
6.2e+02	0.0032\\
6.3e+02	0.0032\\
6.3e+02	0.0032\\
6.3e+02	0.0032\\
6.3e+02	0.0032\\
6.3e+02	0.0032\\
6.3e+02	0.0032\\
6.3e+02	0.0032\\
6.3e+02	0.0032\\
6.3e+02	0.0031\\
6.4e+02	0.0031\\
6.4e+02	0.0031\\
6.4e+02	0.0031\\
6.4e+02	0.0031\\
6.4e+02	0.0031\\
6.4e+02	0.0031\\
6.4e+02	0.0031\\
6.4e+02	0.0031\\
6.4e+02	0.0031\\
6.4e+02	0.0031\\
6.4e+02	0.0031\\
6.5e+02	0.0031\\
6.5e+02	0.0031\\
6.5e+02	0.0031\\
6.5e+02	0.0031\\
6.5e+02	0.0031\\
6.5e+02	0.0031\\
6.5e+02	0.0031\\
6.5e+02	0.0031\\
6.5e+02	0.0031\\
6.6e+02	0.003\\
6.6e+02	0.003\\
6.6e+02	0.003\\
6.6e+02	0.003\\
6.6e+02	0.003\\
6.6e+02	0.003\\
6.6e+02	0.003\\
6.6e+02	0.003\\
6.6e+02	0.003\\
6.6e+02	0.003\\
6.6e+02	0.003\\
6.7e+02	0.003\\
6.7e+02	0.003\\
6.7e+02	0.003\\
6.7e+02	0.003\\
6.7e+02	0.003\\
6.7e+02	0.003\\
6.7e+02	0.003\\
6.7e+02	0.003\\
6.7e+02	0.003\\
6.8e+02	0.003\\
6.8e+02	0.003\\
6.8e+02	0.0029\\
6.8e+02	0.0029\\
6.8e+02	0.0029\\
6.8e+02	0.0029\\
6.8e+02	0.0029\\
6.8e+02	0.0029\\
6.8e+02	0.0029\\
6.8e+02	0.0029\\
6.8e+02	0.0029\\
6.9e+02	0.0029\\
6.9e+02	0.0029\\
6.9e+02	0.0029\\
6.9e+02	0.0029\\
6.9e+02	0.0029\\
6.9e+02	0.0029\\
6.9e+02	0.0029\\
6.9e+02	0.0029\\
6.9e+02	0.0029\\
7e+02	0.0029\\
7e+02	0.0029\\
7e+02	0.0029\\
7e+02	0.0029\\
7e+02	0.0029\\
7e+02	0.0029\\
7e+02	0.0028\\
7e+02	0.0028\\
7e+02	0.0028\\
7e+02	0.0028\\
7e+02	0.0028\\
7.1e+02	0.0028\\
7.1e+02	0.0028\\
7.1e+02	0.0028\\
7.1e+02	0.0028\\
7.1e+02	0.0028\\
7.1e+02	0.0028\\
7.1e+02	0.0028\\
7.1e+02	0.0028\\
7.1e+02	0.0028\\
7.2e+02	0.0028\\
7.2e+02	0.0028\\
7.2e+02	0.0028\\
7.2e+02	0.0028\\
7.2e+02	0.0028\\
7.2e+02	0.0028\\
7.2e+02	0.0028\\
7.2e+02	0.0028\\
7.2e+02	0.0028\\
7.2e+02	0.0028\\
7.2e+02	0.0028\\
7.3e+02	0.0028\\
7.3e+02	0.0027\\
7.3e+02	0.0027\\
7.3e+02	0.0027\\
7.3e+02	0.0027\\
7.3e+02	0.0027\\
7.3e+02	0.0027\\
7.3e+02	0.0027\\
7.3e+02	0.0027\\
7.4e+02	0.0027\\
7.4e+02	0.0027\\
7.4e+02	0.0027\\
7.4e+02	0.0027\\
7.4e+02	0.0027\\
7.4e+02	0.0027\\
7.4e+02	0.0027\\
7.4e+02	0.0027\\
7.4e+02	0.0027\\
7.4e+02	0.0027\\
7.4e+02	0.0027\\
7.5e+02	0.0027\\
7.5e+02	0.0027\\
7.5e+02	0.0027\\
7.5e+02	0.0027\\
7.5e+02	0.0027\\
7.5e+02	0.0027\\
7.5e+02	0.0027\\
7.5e+02	0.0027\\
7.5e+02	0.0026\\
7.6e+02	0.0026\\
7.6e+02	0.0026\\
7.6e+02	0.0026\\
7.6e+02	0.0026\\
7.6e+02	0.0026\\
7.6e+02	0.0026\\
7.6e+02	0.0026\\
7.6e+02	0.0026\\
7.6e+02	0.0026\\
7.6e+02	0.0026\\
7.6e+02	0.0026\\
7.7e+02	0.0026\\
7.7e+02	0.0026\\
7.7e+02	0.0026\\
7.7e+02	0.0026\\
7.7e+02	0.0026\\
7.7e+02	0.0026\\
7.7e+02	0.0026\\
7.7e+02	0.0026\\
7.7e+02	0.0026\\
7.8e+02	0.0026\\
7.8e+02	0.0026\\
7.8e+02	0.0026\\
7.8e+02	0.0026\\
7.8e+02	0.0026\\
7.8e+02	0.0026\\
7.8e+02	0.0026\\
7.8e+02	0.0026\\
7.8e+02	0.0026\\
7.8e+02	0.0025\\
7.8e+02	0.0025\\
7.9e+02	0.0025\\
7.9e+02	0.0025\\
7.9e+02	0.0025\\
7.9e+02	0.0025\\
7.9e+02	0.0025\\
7.9e+02	0.0025\\
7.9e+02	0.0025\\
7.9e+02	0.0025\\
7.9e+02	0.0025\\
8e+02	0.0025\\
8e+02	0.0025\\
8e+02	0.0025\\
8e+02	0.0025\\
8e+02	0.0025\\
8e+02	0.0025\\
8e+02	0.0025\\
8e+02	0.0025\\
8e+02	0.0025\\
8e+02	0.0025\\
8e+02	0.0025\\
8.1e+02	0.0025\\
8.1e+02	0.0025\\
8.1e+02	0.0025\\
8.1e+02	0.0025\\
8.1e+02	0.0025\\
8.1e+02	0.0025\\
8.1e+02	0.0025\\
8.1e+02	0.0025\\
8.1e+02	0.0025\\
8.2e+02	0.0025\\
8.2e+02	0.0024\\
8.2e+02	0.0024\\
8.2e+02	0.0024\\
8.2e+02	0.0024\\
8.2e+02	0.0024\\
8.2e+02	0.0024\\
8.2e+02	0.0024\\
8.2e+02	0.0024\\
8.2e+02	0.0024\\
8.2e+02	0.0024\\
8.3e+02	0.0024\\
8.3e+02	0.0024\\
8.3e+02	0.0024\\
8.3e+02	0.0024\\
8.3e+02	0.0024\\
8.3e+02	0.0024\\
8.3e+02	0.0024\\
8.3e+02	0.0024\\
8.3e+02	0.0024\\
8.4e+02	0.0024\\
8.4e+02	0.0024\\
8.4e+02	0.0024\\
8.4e+02	0.0024\\
8.4e+02	0.0024\\
8.4e+02	0.0024\\
8.4e+02	0.0024\\
8.4e+02	0.0024\\
8.4e+02	0.0024\\
8.4e+02	0.0024\\
8.4e+02	0.0024\\
8.5e+02	0.0024\\
8.5e+02	0.0024\\
8.5e+02	0.0024\\
8.5e+02	0.0024\\
8.5e+02	0.0024\\
8.5e+02	0.0023\\
8.5e+02	0.0023\\
8.5e+02	0.0023\\
8.5e+02	0.0023\\
8.6e+02	0.0023\\
8.6e+02	0.0023\\
8.6e+02	0.0023\\
8.6e+02	0.0023\\
8.6e+02	0.0023\\
8.6e+02	0.0023\\
8.6e+02	0.0023\\
8.6e+02	0.0023\\
8.6e+02	0.0023\\
8.6e+02	0.0023\\
8.6e+02	0.0023\\
8.7e+02	0.0023\\
8.7e+02	0.0023\\
8.7e+02	0.0023\\
8.7e+02	0.0023\\
8.7e+02	0.0023\\
8.7e+02	0.0023\\
8.7e+02	0.0023\\
8.7e+02	0.0023\\
8.7e+02	0.0023\\
8.8e+02	0.0023\\
8.8e+02	0.0023\\
8.8e+02	0.0023\\
8.8e+02	0.0023\\
8.8e+02	0.0023\\
8.8e+02	0.0023\\
8.8e+02	0.0023\\
8.8e+02	0.0023\\
8.8e+02	0.0023\\
8.8e+02	0.0023\\
8.8e+02	0.0023\\
8.9e+02	0.0023\\
8.9e+02	0.0023\\
8.9e+02	0.0022\\
8.9e+02	0.0022\\
8.9e+02	0.0022\\
8.9e+02	0.0022\\
8.9e+02	0.0022\\
8.9e+02	0.0022\\
8.9e+02	0.0022\\
9e+02	0.0022\\
9e+02	0.0022\\
9e+02	0.0022\\
9e+02	0.0022\\
9e+02	0.0022\\
9e+02	0.0022\\
9e+02	0.0022\\
9e+02	0.0022\\
9e+02	0.0022\\
9e+02	0.0022\\
9e+02	0.0022\\
9.1e+02	0.0022\\
9.1e+02	0.0022\\
9.1e+02	0.0022\\
9.1e+02	0.0022\\
9.1e+02	0.0022\\
9.1e+02	0.0022\\
9.1e+02	0.0022\\
9.1e+02	0.0022\\
9.1e+02	0.0022\\
9.2e+02	0.0022\\
9.2e+02	0.0022\\
9.2e+02	0.0022\\
9.2e+02	0.0022\\
9.2e+02	0.0022\\
9.2e+02	0.0022\\
9.2e+02	0.0022\\
9.2e+02	0.0022\\
9.2e+02	0.0022\\
9.2e+02	0.0022\\
9.2e+02	0.0022\\
9.3e+02	0.0022\\
9.3e+02	0.0022\\
9.3e+02	0.0022\\
9.3e+02	0.0022\\
9.3e+02	0.0021\\
9.3e+02	0.0021\\
9.3e+02	0.0021\\
9.3e+02	0.0021\\
9.3e+02	0.0021\\
9.4e+02	0.0021\\
9.4e+02	0.0021\\
9.4e+02	0.0021\\
9.4e+02	0.0021\\
9.4e+02	0.0021\\
9.4e+02	0.0021\\
9.4e+02	0.0021\\
9.4e+02	0.0021\\
9.4e+02	0.0021\\
9.4e+02	0.0021\\
9.4e+02	0.0021\\
9.5e+02	0.0021\\
9.5e+02	0.0021\\
9.5e+02	0.0021\\
9.5e+02	0.0021\\
9.5e+02	0.0021\\
9.5e+02	0.0021\\
9.5e+02	0.0021\\
9.5e+02	0.0021\\
9.5e+02	0.0021\\
9.6e+02	0.0021\\
9.6e+02	0.0021\\
9.6e+02	0.0021\\
9.6e+02	0.0021\\
9.6e+02	0.0021\\
9.6e+02	0.0021\\
9.6e+02	0.0021\\
9.6e+02	0.0021\\
9.6e+02	0.0021\\
9.6e+02	0.0021\\
9.6e+02	0.0021\\
9.7e+02	0.0021\\
9.7e+02	0.0021\\
9.7e+02	0.0021\\
9.7e+02	0.0021\\
9.7e+02	0.0021\\
9.7e+02	0.0021\\
9.7e+02	0.0021\\
9.7e+02	0.0021\\
9.7e+02	0.0021\\
9.8e+02	0.002\\
9.8e+02	0.002\\
9.8e+02	0.002\\
9.8e+02	0.002\\
9.8e+02	0.002\\
9.8e+02	0.002\\
9.8e+02	0.002\\
9.8e+02	0.002\\
9.8e+02	0.002\\
9.8e+02	0.002\\
9.8e+02	0.002\\
9.9e+02	0.002\\
9.9e+02	0.002\\
9.9e+02	0.002\\
9.9e+02	0.002\\
9.9e+02	0.002\\
9.9e+02	0.002\\
9.9e+02	0.002\\
9.9e+02	0.002\\
9.9e+02	0.002\\
1e+03	0.002\\
1e+03	0.002\\
1e+03	0.002\\
1e+03	0.002\\
};
\addlegendentry{$\vert f_{\ell,\ell+1} \vert$}

\addplot [color=red, dashed, thick]
  table[row sep=crcr]{%
0	inf\\
1	2.8\\
2	1.4\\
3	0.93\\
4	0.7\\
5	0.56\\
6	0.47\\
7	0.4\\
8	0.35\\
9	0.31\\
10	0.28\\
11	0.25\\
12	0.23\\
13	0.22\\
14	0.2\\
15	0.19\\
16	0.18\\
17	0.16\\
18	0.16\\
19	0.15\\
20	0.14\\
21	0.13\\
22	0.13\\
23	0.12\\
24	0.12\\
25	0.11\\
26	0.11\\
27	0.1\\
28	0.1\\
29	0.097\\
30	0.093\\
31	0.09\\
32	0.088\\
33	0.085\\
34	0.082\\
35	0.08\\
36	0.078\\
37	0.076\\
38	0.074\\
39	0.072\\
40	0.07\\
41	0.068\\
42	0.067\\
43	0.065\\
44	0.064\\
45	0.062\\
46	0.061\\
47	0.06\\
48	0.058\\
49	0.057\\
50	0.056\\
51	0.055\\
52	0.054\\
53	0.053\\
54	0.052\\
55	0.051\\
56	0.05\\
57	0.049\\
58	0.048\\
59	0.047\\
60	0.047\\
61	0.046\\
62	0.045\\
63	0.044\\
64	0.044\\
65	0.043\\
66	0.042\\
67	0.042\\
68	0.041\\
69	0.041\\
70	0.04\\
71	0.039\\
72	0.039\\
73	0.038\\
74	0.038\\
75	0.037\\
76	0.037\\
77	0.036\\
78	0.036\\
79	0.035\\
80	0.035\\
81	0.035\\
82	0.034\\
83	0.034\\
84	0.033\\
85	0.033\\
86	0.033\\
87	0.032\\
88	0.032\\
89	0.031\\
90	0.031\\
91	0.031\\
92	0.03\\
93	0.03\\
94	0.03\\
95	0.029\\
96	0.029\\
97	0.029\\
98	0.029\\
99	0.028\\
1e+02	0.028\\
1e+02	0.028\\
1e+02	0.027\\
1e+02	0.027\\
1e+02	0.027\\
1e+02	0.027\\
1.1e+02	0.026\\
1.1e+02	0.026\\
1.1e+02	0.026\\
1.1e+02	0.026\\
1.1e+02	0.025\\
1.1e+02	0.025\\
1.1e+02	0.025\\
1.1e+02	0.025\\
1.1e+02	0.025\\
1.2e+02	0.024\\
1.2e+02	0.024\\
1.2e+02	0.024\\
1.2e+02	0.024\\
1.2e+02	0.024\\
1.2e+02	0.023\\
1.2e+02	0.023\\
1.2e+02	0.023\\
1.2e+02	0.023\\
1.2e+02	0.023\\
1.2e+02	0.022\\
1.3e+02	0.022\\
1.3e+02	0.022\\
1.3e+02	0.022\\
1.3e+02	0.022\\
1.3e+02	0.022\\
1.3e+02	0.021\\
1.3e+02	0.021\\
1.3e+02	0.021\\
1.3e+02	0.021\\
1.4e+02	0.021\\
1.4e+02	0.021\\
1.4e+02	0.02\\
1.4e+02	0.02\\
1.4e+02	0.02\\
1.4e+02	0.02\\
1.4e+02	0.02\\
1.4e+02	0.02\\
1.4e+02	0.02\\
1.4e+02	0.019\\
1.4e+02	0.019\\
1.5e+02	0.019\\
1.5e+02	0.019\\
1.5e+02	0.019\\
1.5e+02	0.019\\
1.5e+02	0.019\\
1.5e+02	0.019\\
1.5e+02	0.018\\
1.5e+02	0.018\\
1.5e+02	0.018\\
1.6e+02	0.018\\
1.6e+02	0.018\\
1.6e+02	0.018\\
1.6e+02	0.018\\
1.6e+02	0.018\\
1.6e+02	0.018\\
1.6e+02	0.017\\
1.6e+02	0.017\\
1.6e+02	0.017\\
1.6e+02	0.017\\
1.6e+02	0.017\\
1.7e+02	0.017\\
1.7e+02	0.017\\
1.7e+02	0.017\\
1.7e+02	0.017\\
1.7e+02	0.016\\
1.7e+02	0.016\\
1.7e+02	0.016\\
1.7e+02	0.016\\
1.7e+02	0.016\\
1.8e+02	0.016\\
1.8e+02	0.016\\
1.8e+02	0.016\\
1.8e+02	0.016\\
1.8e+02	0.016\\
1.8e+02	0.016\\
1.8e+02	0.015\\
1.8e+02	0.015\\
1.8e+02	0.015\\
1.8e+02	0.015\\
1.8e+02	0.015\\
1.9e+02	0.015\\
1.9e+02	0.015\\
1.9e+02	0.015\\
1.9e+02	0.015\\
1.9e+02	0.015\\
1.9e+02	0.015\\
1.9e+02	0.015\\
1.9e+02	0.015\\
1.9e+02	0.014\\
2e+02	0.014\\
2e+02	0.014\\
2e+02	0.014\\
2e+02	0.014\\
2e+02	0.014\\
2e+02	0.014\\
2e+02	0.014\\
2e+02	0.014\\
2e+02	0.014\\
2e+02	0.014\\
2e+02	0.014\\
2.1e+02	0.014\\
2.1e+02	0.014\\
2.1e+02	0.013\\
2.1e+02	0.013\\
2.1e+02	0.013\\
2.1e+02	0.013\\
2.1e+02	0.013\\
2.1e+02	0.013\\
2.1e+02	0.013\\
2.2e+02	0.013\\
2.2e+02	0.013\\
2.2e+02	0.013\\
2.2e+02	0.013\\
2.2e+02	0.013\\
2.2e+02	0.013\\
2.2e+02	0.013\\
2.2e+02	0.013\\
2.2e+02	0.013\\
2.2e+02	0.013\\
2.2e+02	0.012\\
2.3e+02	0.012\\
2.3e+02	0.012\\
2.3e+02	0.012\\
2.3e+02	0.012\\
2.3e+02	0.012\\
2.3e+02	0.012\\
2.3e+02	0.012\\
2.3e+02	0.012\\
2.3e+02	0.012\\
2.4e+02	0.012\\
2.4e+02	0.012\\
2.4e+02	0.012\\
2.4e+02	0.012\\
2.4e+02	0.012\\
2.4e+02	0.012\\
2.4e+02	0.012\\
2.4e+02	0.012\\
2.4e+02	0.012\\
2.4e+02	0.011\\
2.4e+02	0.011\\
2.5e+02	0.011\\
2.5e+02	0.011\\
2.5e+02	0.011\\
2.5e+02	0.011\\
2.5e+02	0.011\\
2.5e+02	0.011\\
2.5e+02	0.011\\
2.5e+02	0.011\\
2.5e+02	0.011\\
2.6e+02	0.011\\
2.6e+02	0.011\\
2.6e+02	0.011\\
2.6e+02	0.011\\
2.6e+02	0.011\\
2.6e+02	0.011\\
2.6e+02	0.011\\
2.6e+02	0.011\\
2.6e+02	0.011\\
2.6e+02	0.011\\
2.6e+02	0.011\\
2.7e+02	0.011\\
2.7e+02	0.01\\
2.7e+02	0.01\\
2.7e+02	0.01\\
2.7e+02	0.01\\
2.7e+02	0.01\\
2.7e+02	0.01\\
2.7e+02	0.01\\
2.7e+02	0.01\\
2.8e+02	0.01\\
2.8e+02	0.01\\
2.8e+02	0.01\\
2.8e+02	0.01\\
2.8e+02	0.01\\
2.8e+02	0.01\\
2.8e+02	0.01\\
2.8e+02	0.0099\\
2.8e+02	0.0099\\
2.8e+02	0.0099\\
2.8e+02	0.0098\\
2.9e+02	0.0098\\
2.9e+02	0.0098\\
2.9e+02	0.0097\\
2.9e+02	0.0097\\
2.9e+02	0.0097\\
2.9e+02	0.0096\\
2.9e+02	0.0096\\
2.9e+02	0.0096\\
2.9e+02	0.0095\\
3e+02	0.0095\\
3e+02	0.0095\\
3e+02	0.0094\\
3e+02	0.0094\\
3e+02	0.0094\\
3e+02	0.0093\\
3e+02	0.0093\\
3e+02	0.0093\\
3e+02	0.0092\\
3e+02	0.0092\\
3e+02	0.0092\\
3.1e+02	0.0092\\
3.1e+02	0.0091\\
3.1e+02	0.0091\\
3.1e+02	0.0091\\
3.1e+02	0.009\\
3.1e+02	0.009\\
3.1e+02	0.009\\
3.1e+02	0.0089\\
3.1e+02	0.0089\\
3.2e+02	0.0089\\
3.2e+02	0.0089\\
3.2e+02	0.0088\\
3.2e+02	0.0088\\
3.2e+02	0.0088\\
3.2e+02	0.0088\\
3.2e+02	0.0087\\
3.2e+02	0.0087\\
3.2e+02	0.0087\\
3.2e+02	0.0086\\
3.2e+02	0.0086\\
3.3e+02	0.0086\\
3.3e+02	0.0086\\
3.3e+02	0.0085\\
3.3e+02	0.0085\\
3.3e+02	0.0085\\
3.3e+02	0.0085\\
3.3e+02	0.0084\\
3.3e+02	0.0084\\
3.3e+02	0.0084\\
3.4e+02	0.0084\\
3.4e+02	0.0083\\
3.4e+02	0.0083\\
3.4e+02	0.0083\\
3.4e+02	0.0083\\
3.4e+02	0.0082\\
3.4e+02	0.0082\\
3.4e+02	0.0082\\
3.4e+02	0.0082\\
3.4e+02	0.0081\\
3.4e+02	0.0081\\
3.5e+02	0.0081\\
3.5e+02	0.0081\\
3.5e+02	0.008\\
3.5e+02	0.008\\
3.5e+02	0.008\\
3.5e+02	0.008\\
3.5e+02	0.008\\
3.5e+02	0.0079\\
3.5e+02	0.0079\\
3.6e+02	0.0079\\
3.6e+02	0.0079\\
3.6e+02	0.0078\\
3.6e+02	0.0078\\
3.6e+02	0.0078\\
3.6e+02	0.0078\\
3.6e+02	0.0078\\
3.6e+02	0.0077\\
3.6e+02	0.0077\\
3.6e+02	0.0077\\
3.6e+02	0.0077\\
3.7e+02	0.0077\\
3.7e+02	0.0076\\
3.7e+02	0.0076\\
3.7e+02	0.0076\\
3.7e+02	0.0076\\
3.7e+02	0.0075\\
3.7e+02	0.0075\\
3.7e+02	0.0075\\
3.7e+02	0.0075\\
3.8e+02	0.0075\\
3.8e+02	0.0074\\
3.8e+02	0.0074\\
3.8e+02	0.0074\\
3.8e+02	0.0074\\
3.8e+02	0.0074\\
3.8e+02	0.0073\\
3.8e+02	0.0073\\
3.8e+02	0.0073\\
3.8e+02	0.0073\\
3.8e+02	0.0073\\
3.9e+02	0.0073\\
3.9e+02	0.0072\\
3.9e+02	0.0072\\
3.9e+02	0.0072\\
3.9e+02	0.0072\\
3.9e+02	0.0072\\
3.9e+02	0.0071\\
3.9e+02	0.0071\\
3.9e+02	0.0071\\
4e+02	0.0071\\
4e+02	0.0071\\
4e+02	0.0071\\
4e+02	0.007\\
4e+02	0.007\\
4e+02	0.007\\
4e+02	0.007\\
4e+02	0.007\\
4e+02	0.0069\\
4e+02	0.0069\\
4e+02	0.0069\\
4.1e+02	0.0069\\
4.1e+02	0.0069\\
4.1e+02	0.0069\\
4.1e+02	0.0068\\
4.1e+02	0.0068\\
4.1e+02	0.0068\\
4.1e+02	0.0068\\
4.1e+02	0.0068\\
4.1e+02	0.0068\\
4.2e+02	0.0067\\
4.2e+02	0.0067\\
4.2e+02	0.0067\\
4.2e+02	0.0067\\
4.2e+02	0.0067\\
4.2e+02	0.0067\\
4.2e+02	0.0067\\
4.2e+02	0.0066\\
4.2e+02	0.0066\\
4.2e+02	0.0066\\
4.2e+02	0.0066\\
4.3e+02	0.0066\\
4.3e+02	0.0066\\
4.3e+02	0.0065\\
4.3e+02	0.0065\\
4.3e+02	0.0065\\
4.3e+02	0.0065\\
4.3e+02	0.0065\\
4.3e+02	0.0065\\
4.3e+02	0.0065\\
4.4e+02	0.0064\\
4.4e+02	0.0064\\
4.4e+02	0.0064\\
4.4e+02	0.0064\\
4.4e+02	0.0064\\
4.4e+02	0.0064\\
4.4e+02	0.0063\\
4.4e+02	0.0063\\
4.4e+02	0.0063\\
4.4e+02	0.0063\\
4.4e+02	0.0063\\
4.5e+02	0.0063\\
4.5e+02	0.0063\\
4.5e+02	0.0063\\
4.5e+02	0.0062\\
4.5e+02	0.0062\\
4.5e+02	0.0062\\
4.5e+02	0.0062\\
4.5e+02	0.0062\\
4.5e+02	0.0062\\
4.6e+02	0.0062\\
4.6e+02	0.0061\\
4.6e+02	0.0061\\
4.6e+02	0.0061\\
4.6e+02	0.0061\\
4.6e+02	0.0061\\
4.6e+02	0.0061\\
4.6e+02	0.0061\\
4.6e+02	0.006\\
4.6e+02	0.006\\
4.6e+02	0.006\\
4.7e+02	0.006\\
4.7e+02	0.006\\
4.7e+02	0.006\\
4.7e+02	0.006\\
4.7e+02	0.006\\
4.7e+02	0.0059\\
4.7e+02	0.0059\\
4.7e+02	0.0059\\
4.7e+02	0.0059\\
4.8e+02	0.0059\\
4.8e+02	0.0059\\
4.8e+02	0.0059\\
4.8e+02	0.0059\\
4.8e+02	0.0058\\
4.8e+02	0.0058\\
4.8e+02	0.0058\\
4.8e+02	0.0058\\
4.8e+02	0.0058\\
4.8e+02	0.0058\\
4.8e+02	0.0058\\
4.9e+02	0.0058\\
4.9e+02	0.0057\\
4.9e+02	0.0057\\
4.9e+02	0.0057\\
4.9e+02	0.0057\\
4.9e+02	0.0057\\
4.9e+02	0.0057\\
4.9e+02	0.0057\\
4.9e+02	0.0057\\
5e+02	0.0057\\
5e+02	0.0056\\
5e+02	0.0056\\
5e+02	0.0056\\
5e+02	0.0056\\
5e+02	0.0056\\
5e+02	0.0056\\
5e+02	0.0056\\
5e+02	0.0056\\
5e+02	0.0056\\
5e+02	0.0055\\
5.1e+02	0.0055\\
5.1e+02	0.0055\\
5.1e+02	0.0055\\
5.1e+02	0.0055\\
5.1e+02	0.0055\\
5.1e+02	0.0055\\
5.1e+02	0.0055\\
5.1e+02	0.0055\\
5.1e+02	0.0054\\
5.2e+02	0.0054\\
5.2e+02	0.0054\\
5.2e+02	0.0054\\
5.2e+02	0.0054\\
5.2e+02	0.0054\\
5.2e+02	0.0054\\
5.2e+02	0.0054\\
5.2e+02	0.0054\\
5.2e+02	0.0054\\
5.2e+02	0.0053\\
5.2e+02	0.0053\\
5.3e+02	0.0053\\
5.3e+02	0.0053\\
5.3e+02	0.0053\\
5.3e+02	0.0053\\
5.3e+02	0.0053\\
5.3e+02	0.0053\\
5.3e+02	0.0053\\
5.3e+02	0.0053\\
5.3e+02	0.0052\\
5.4e+02	0.0052\\
5.4e+02	0.0052\\
5.4e+02	0.0052\\
5.4e+02	0.0052\\
5.4e+02	0.0052\\
5.4e+02	0.0052\\
5.4e+02	0.0052\\
5.4e+02	0.0052\\
5.4e+02	0.0052\\
5.4e+02	0.0051\\
5.4e+02	0.0051\\
5.5e+02	0.0051\\
5.5e+02	0.0051\\
5.5e+02	0.0051\\
5.5e+02	0.0051\\
5.5e+02	0.0051\\
5.5e+02	0.0051\\
5.5e+02	0.0051\\
5.5e+02	0.0051\\
5.5e+02	0.0051\\
5.6e+02	0.005\\
5.6e+02	0.005\\
5.6e+02	0.005\\
5.6e+02	0.005\\
5.6e+02	0.005\\
5.6e+02	0.005\\
5.6e+02	0.005\\
5.6e+02	0.005\\
5.6e+02	0.005\\
5.6e+02	0.005\\
5.6e+02	0.005\\
5.7e+02	0.0049\\
5.7e+02	0.0049\\
5.7e+02	0.0049\\
5.7e+02	0.0049\\
5.7e+02	0.0049\\
5.7e+02	0.0049\\
5.7e+02	0.0049\\
5.7e+02	0.0049\\
5.7e+02	0.0049\\
5.8e+02	0.0049\\
5.8e+02	0.0049\\
5.8e+02	0.0049\\
5.8e+02	0.0048\\
5.8e+02	0.0048\\
5.8e+02	0.0048\\
5.8e+02	0.0048\\
5.8e+02	0.0048\\
5.8e+02	0.0048\\
5.8e+02	0.0048\\
5.8e+02	0.0048\\
5.9e+02	0.0048\\
5.9e+02	0.0048\\
5.9e+02	0.0048\\
5.9e+02	0.0048\\
5.9e+02	0.0047\\
5.9e+02	0.0047\\
5.9e+02	0.0047\\
5.9e+02	0.0047\\
5.9e+02	0.0047\\
6e+02	0.0047\\
6e+02	0.0047\\
6e+02	0.0047\\
6e+02	0.0047\\
6e+02	0.0047\\
6e+02	0.0047\\
6e+02	0.0047\\
6e+02	0.0047\\
6e+02	0.0046\\
6e+02	0.0046\\
6e+02	0.0046\\
6.1e+02	0.0046\\
6.1e+02	0.0046\\
6.1e+02	0.0046\\
6.1e+02	0.0046\\
6.1e+02	0.0046\\
6.1e+02	0.0046\\
6.1e+02	0.0046\\
6.1e+02	0.0046\\
6.1e+02	0.0046\\
6.2e+02	0.0046\\
6.2e+02	0.0045\\
6.2e+02	0.0045\\
6.2e+02	0.0045\\
6.2e+02	0.0045\\
6.2e+02	0.0045\\
6.2e+02	0.0045\\
6.2e+02	0.0045\\
6.2e+02	0.0045\\
6.2e+02	0.0045\\
6.2e+02	0.0045\\
6.3e+02	0.0045\\
6.3e+02	0.0045\\
6.3e+02	0.0045\\
6.3e+02	0.0045\\
6.3e+02	0.0044\\
6.3e+02	0.0044\\
6.3e+02	0.0044\\
6.3e+02	0.0044\\
6.3e+02	0.0044\\
6.4e+02	0.0044\\
6.4e+02	0.0044\\
6.4e+02	0.0044\\
6.4e+02	0.0044\\
6.4e+02	0.0044\\
6.4e+02	0.0044\\
6.4e+02	0.0044\\
6.4e+02	0.0044\\
6.4e+02	0.0044\\
6.4e+02	0.0043\\
6.4e+02	0.0043\\
6.5e+02	0.0043\\
6.5e+02	0.0043\\
6.5e+02	0.0043\\
6.5e+02	0.0043\\
6.5e+02	0.0043\\
6.5e+02	0.0043\\
6.5e+02	0.0043\\
6.5e+02	0.0043\\
6.5e+02	0.0043\\
6.6e+02	0.0043\\
6.6e+02	0.0043\\
6.6e+02	0.0043\\
6.6e+02	0.0043\\
6.6e+02	0.0042\\
6.6e+02	0.0042\\
6.6e+02	0.0042\\
6.6e+02	0.0042\\
6.6e+02	0.0042\\
6.6e+02	0.0042\\
6.6e+02	0.0042\\
6.7e+02	0.0042\\
6.7e+02	0.0042\\
6.7e+02	0.0042\\
6.7e+02	0.0042\\
6.7e+02	0.0042\\
6.7e+02	0.0042\\
6.7e+02	0.0042\\
6.7e+02	0.0042\\
6.7e+02	0.0042\\
6.8e+02	0.0041\\
6.8e+02	0.0041\\
6.8e+02	0.0041\\
6.8e+02	0.0041\\
6.8e+02	0.0041\\
6.8e+02	0.0041\\
6.8e+02	0.0041\\
6.8e+02	0.0041\\
6.8e+02	0.0041\\
6.8e+02	0.0041\\
6.8e+02	0.0041\\
6.9e+02	0.0041\\
6.9e+02	0.0041\\
6.9e+02	0.0041\\
6.9e+02	0.0041\\
6.9e+02	0.0041\\
6.9e+02	0.0041\\
6.9e+02	0.004\\
6.9e+02	0.004\\
6.9e+02	0.004\\
7e+02	0.004\\
7e+02	0.004\\
7e+02	0.004\\
7e+02	0.004\\
7e+02	0.004\\
7e+02	0.004\\
7e+02	0.004\\
7e+02	0.004\\
7e+02	0.004\\
7e+02	0.004\\
7e+02	0.004\\
7.1e+02	0.004\\
7.1e+02	0.004\\
7.1e+02	0.004\\
7.1e+02	0.0039\\
7.1e+02	0.0039\\
7.1e+02	0.0039\\
7.1e+02	0.0039\\
7.1e+02	0.0039\\
7.1e+02	0.0039\\
7.2e+02	0.0039\\
7.2e+02	0.0039\\
7.2e+02	0.0039\\
7.2e+02	0.0039\\
7.2e+02	0.0039\\
7.2e+02	0.0039\\
7.2e+02	0.0039\\
7.2e+02	0.0039\\
7.2e+02	0.0039\\
7.2e+02	0.0039\\
7.2e+02	0.0039\\
7.3e+02	0.0039\\
7.3e+02	0.0039\\
7.3e+02	0.0038\\
7.3e+02	0.0038\\
7.3e+02	0.0038\\
7.3e+02	0.0038\\
7.3e+02	0.0038\\
7.3e+02	0.0038\\
7.3e+02	0.0038\\
7.4e+02	0.0038\\
7.4e+02	0.0038\\
7.4e+02	0.0038\\
7.4e+02	0.0038\\
7.4e+02	0.0038\\
7.4e+02	0.0038\\
7.4e+02	0.0038\\
7.4e+02	0.0038\\
7.4e+02	0.0038\\
7.4e+02	0.0038\\
7.4e+02	0.0038\\
7.5e+02	0.0038\\
7.5e+02	0.0037\\
7.5e+02	0.0037\\
7.5e+02	0.0037\\
7.5e+02	0.0037\\
7.5e+02	0.0037\\
7.5e+02	0.0037\\
7.5e+02	0.0037\\
7.5e+02	0.0037\\
7.6e+02	0.0037\\
7.6e+02	0.0037\\
7.6e+02	0.0037\\
7.6e+02	0.0037\\
7.6e+02	0.0037\\
7.6e+02	0.0037\\
7.6e+02	0.0037\\
7.6e+02	0.0037\\
7.6e+02	0.0037\\
7.6e+02	0.0037\\
7.6e+02	0.0037\\
7.7e+02	0.0037\\
7.7e+02	0.0037\\
7.7e+02	0.0036\\
7.7e+02	0.0036\\
7.7e+02	0.0036\\
7.7e+02	0.0036\\
7.7e+02	0.0036\\
7.7e+02	0.0036\\
7.7e+02	0.0036\\
7.8e+02	0.0036\\
7.8e+02	0.0036\\
7.8e+02	0.0036\\
7.8e+02	0.0036\\
7.8e+02	0.0036\\
7.8e+02	0.0036\\
7.8e+02	0.0036\\
7.8e+02	0.0036\\
7.8e+02	0.0036\\
7.8e+02	0.0036\\
7.8e+02	0.0036\\
7.9e+02	0.0036\\
7.9e+02	0.0036\\
7.9e+02	0.0036\\
7.9e+02	0.0035\\
7.9e+02	0.0035\\
7.9e+02	0.0035\\
7.9e+02	0.0035\\
7.9e+02	0.0035\\
7.9e+02	0.0035\\
8e+02	0.0035\\
8e+02	0.0035\\
8e+02	0.0035\\
8e+02	0.0035\\
8e+02	0.0035\\
8e+02	0.0035\\
8e+02	0.0035\\
8e+02	0.0035\\
8e+02	0.0035\\
8e+02	0.0035\\
8e+02	0.0035\\
8.1e+02	0.0035\\
8.1e+02	0.0035\\
8.1e+02	0.0035\\
8.1e+02	0.0035\\
8.1e+02	0.0035\\
8.1e+02	0.0035\\
8.1e+02	0.0034\\
8.1e+02	0.0034\\
8.1e+02	0.0034\\
8.2e+02	0.0034\\
8.2e+02	0.0034\\
8.2e+02	0.0034\\
8.2e+02	0.0034\\
8.2e+02	0.0034\\
8.2e+02	0.0034\\
8.2e+02	0.0034\\
8.2e+02	0.0034\\
8.2e+02	0.0034\\
8.2e+02	0.0034\\
8.2e+02	0.0034\\
8.3e+02	0.0034\\
8.3e+02	0.0034\\
8.3e+02	0.0034\\
8.3e+02	0.0034\\
8.3e+02	0.0034\\
8.3e+02	0.0034\\
8.3e+02	0.0034\\
8.3e+02	0.0034\\
8.3e+02	0.0034\\
8.4e+02	0.0034\\
8.4e+02	0.0033\\
8.4e+02	0.0033\\
8.4e+02	0.0033\\
8.4e+02	0.0033\\
8.4e+02	0.0033\\
8.4e+02	0.0033\\
8.4e+02	0.0033\\
8.4e+02	0.0033\\
8.4e+02	0.0033\\
8.4e+02	0.0033\\
8.5e+02	0.0033\\
8.5e+02	0.0033\\
8.5e+02	0.0033\\
8.5e+02	0.0033\\
8.5e+02	0.0033\\
8.5e+02	0.0033\\
8.5e+02	0.0033\\
8.5e+02	0.0033\\
8.5e+02	0.0033\\
8.6e+02	0.0033\\
8.6e+02	0.0033\\
8.6e+02	0.0033\\
8.6e+02	0.0033\\
8.6e+02	0.0033\\
8.6e+02	0.0033\\
8.6e+02	0.0033\\
8.6e+02	0.0032\\
8.6e+02	0.0032\\
8.6e+02	0.0032\\
8.6e+02	0.0032\\
8.7e+02	0.0032\\
8.7e+02	0.0032\\
8.7e+02	0.0032\\
8.7e+02	0.0032\\
8.7e+02	0.0032\\
8.7e+02	0.0032\\
8.7e+02	0.0032\\
8.7e+02	0.0032\\
8.7e+02	0.0032\\
8.8e+02	0.0032\\
8.8e+02	0.0032\\
8.8e+02	0.0032\\
8.8e+02	0.0032\\
8.8e+02	0.0032\\
8.8e+02	0.0032\\
8.8e+02	0.0032\\
8.8e+02	0.0032\\
8.8e+02	0.0032\\
8.8e+02	0.0032\\
8.8e+02	0.0032\\
8.9e+02	0.0032\\
8.9e+02	0.0032\\
8.9e+02	0.0032\\
8.9e+02	0.0031\\
8.9e+02	0.0031\\
8.9e+02	0.0031\\
8.9e+02	0.0031\\
8.9e+02	0.0031\\
8.9e+02	0.0031\\
9e+02	0.0031\\
9e+02	0.0031\\
9e+02	0.0031\\
9e+02	0.0031\\
9e+02	0.0031\\
9e+02	0.0031\\
9e+02	0.0031\\
9e+02	0.0031\\
9e+02	0.0031\\
9e+02	0.0031\\
9e+02	0.0031\\
9.1e+02	0.0031\\
9.1e+02	0.0031\\
9.1e+02	0.0031\\
9.1e+02	0.0031\\
9.1e+02	0.0031\\
9.1e+02	0.0031\\
9.1e+02	0.0031\\
9.1e+02	0.0031\\
9.1e+02	0.0031\\
9.2e+02	0.0031\\
9.2e+02	0.0031\\
9.2e+02	0.0031\\
9.2e+02	0.0031\\
9.2e+02	0.003\\
9.2e+02	0.003\\
9.2e+02	0.003\\
9.2e+02	0.003\\
9.2e+02	0.003\\
9.2e+02	0.003\\
9.2e+02	0.003\\
9.3e+02	0.003\\
9.3e+02	0.003\\
9.3e+02	0.003\\
9.3e+02	0.003\\
9.3e+02	0.003\\
9.3e+02	0.003\\
9.3e+02	0.003\\
9.3e+02	0.003\\
9.3e+02	0.003\\
9.4e+02	0.003\\
9.4e+02	0.003\\
9.4e+02	0.003\\
9.4e+02	0.003\\
9.4e+02	0.003\\
9.4e+02	0.003\\
9.4e+02	0.003\\
9.4e+02	0.003\\
9.4e+02	0.003\\
9.4e+02	0.003\\
9.4e+02	0.003\\
9.5e+02	0.003\\
9.5e+02	0.003\\
9.5e+02	0.003\\
9.5e+02	0.003\\
9.5e+02	0.0029\\
9.5e+02	0.0029\\
9.5e+02	0.0029\\
9.5e+02	0.0029\\
9.5e+02	0.0029\\
9.6e+02	0.0029\\
9.6e+02	0.0029\\
9.6e+02	0.0029\\
9.6e+02	0.0029\\
9.6e+02	0.0029\\
9.6e+02	0.0029\\
9.6e+02	0.0029\\
9.6e+02	0.0029\\
9.6e+02	0.0029\\
9.6e+02	0.0029\\
9.6e+02	0.0029\\
9.7e+02	0.0029\\
9.7e+02	0.0029\\
9.7e+02	0.0029\\
9.7e+02	0.0029\\
9.7e+02	0.0029\\
9.7e+02	0.0029\\
9.7e+02	0.0029\\
9.7e+02	0.0029\\
9.7e+02	0.0029\\
9.8e+02	0.0029\\
9.8e+02	0.0029\\
9.8e+02	0.0029\\
9.8e+02	0.0029\\
9.8e+02	0.0029\\
9.8e+02	0.0029\\
9.8e+02	0.0029\\
9.8e+02	0.0029\\
9.8e+02	0.0028\\
9.8e+02	0.0028\\
9.8e+02	0.0028\\
9.9e+02	0.0028\\
9.9e+02	0.0028\\
9.9e+02	0.0028\\
9.9e+02	0.0028\\
9.9e+02	0.0028\\
9.9e+02	0.0028\\
9.9e+02	0.0028\\
9.9e+02	0.0028\\
9.9e+02	0.0028\\
1e+03	0.0028\\
1e+03	0.0028\\
1e+03	0.0028\\
1e+03	0.0028\\
1e+03	0.0028\\
};
\addlegendentry{$\frac{2.2}{\ell}$}

\end{axis}
\end{tikzpicture}%

%% file: figs/maxnorm_DE_M=50nu=1.tikz
%
\begin{tikzpicture}

\begin{axis}[%
width=0.951\figurewidth,
height=\figureheight,
at={(0\figurewidth,0\figureheight)},
scale only axis,
xmin=0,
xmax=50,
xlabel style={font=\color{white!15!black}},
xlabel={$N$},
ymode=log,
ymin=1e-15,
ymax=10,
yminorticks=true,
axis background/.style={fill=white},
legend style={legend cell align=left, align=left, draw=white!15!black}
]
\addplot [color=blue, draw=none, mark=asterisk, mark options={solid, blue}]
  table[row sep=crcr]{%
1	0.42\\
2	0.21\\
3	0.022\\
4	0.0043\\
5	0.00027\\
6	3.6e-05\\
7	1.6e-06\\
8	1.6e-07\\
9	5.6e-09\\
10	4.4e-10\\
11	1.3e-11\\
12	8.4e-13\\
13	2.1e-14\\
14	1e-14\\
15	1e-14\\
16	1e-14\\
17	1e-14\\
18	1e-14\\
19	1e-14\\
20	1e-14\\
21	1e-14\\
22	1e-14\\
23	1e-14\\
24	1e-14\\
25	1e-14\\
26	1e-14\\
27	1e-14\\
28	1e-14\\
29	1e-14\\
30	1e-14\\
31	1e-14\\
32	1e-14\\
33	1e-14\\
34	1e-14\\
35	1e-14\\
36	1e-14\\
37	1e-14\\
38	1e-14\\
39	1e-14\\
40	1e-14\\
41	1e-14\\
42	1e-14\\
43	1e-14\\
44	1e-14\\
45	1e-14\\
46	1e-14\\
47	1e-14\\
48	1e-14\\
49	1e-14\\
50	1e-14\\
};

\addplot [color=red, draw=none, mark=o, mark options={solid, red}]
  table[row sep=crcr]{%
1	10\\
2	5.3\\
3	2.1\\
4	0.32\\
5	0.067\\
6	0.0054\\
7	0.00075\\
8	4.1e-05\\
9	4.2e-06\\
10	1.7e-07\\
11	1.4e-08\\
12	7.6e-10\\
13	3.3e-10\\
14	3e-10\\
15	3e-10\\
16	3e-10\\
17	3e-10\\
18	3e-10\\
19	3e-10\\
20	3e-10\\
21	3e-10\\
22	3e-10\\
23	3e-10\\
24	3e-10\\
25	3e-10\\
26	3e-10\\
27	3e-10\\
28	3e-10\\
29	3e-10\\
30	3e-10\\
31	3e-10\\
32	3e-10\\
33	3e-10\\
34	3e-10\\
35	3e-10\\
36	3e-10\\
37	3e-10\\
38	3e-10\\
39	3e-10\\
40	3e-10\\
41	3e-10\\
42	3e-10\\
43	3e-10\\
44	3e-10\\
45	3e-10\\
46	3e-10\\
47	3e-10\\
48	3e-10\\
49	3e-10\\
50	3e-10\\
};

\end{axis}

\begin{axis}[%
width=1.227\figurewidth,
height=1.227\figureheight,
at={(-0.16\figurewidth,-0.135\figureheight)},
scale only axis,
xmin=0,
xmax=1,
ymin=0,
ymax=1,
axis line style={draw=none},
ticks=none,
axis x line*=bottom,
axis y line*=left,
legend style={legend cell align=left, align=left, draw=white!15!black}
]
\end{axis}
\end{tikzpicture}%

%% file: figs/contour_DE_M=50nu=1.tikz
	\begin{tikzpicture}
	\begin{axis}[
		width=0.95\figurewidth,
		height=\figureheight,
		at={(0\figurewidth,0\figureheight)},
		xlabel = {$k$},
		ylabel = {$\ell$},
		view={90}{90},shader=interp,
		colormap={parula}{
			rgb255=(53,42,135)
			rgb255=(15,92,221)
			rgb255=(18,125,216)
			rgb255=(7,156,207)
			rgb255=(21,177,180)
			rgb255=(89,189,140)
			rgb255=(165,190,107)
			rgb255=(225,185,82)
			rgb255=(252,206,46)
			rgb255=(249,251,14)},
		colorbar,
		colorbar style={
			width=0.2cm},
		point meta min = -20,
		point meta max = 0,
		]
		\addplot3[
		contour filled={
			levels={-17,-16,...,0},
		},
		colormap access=piecewise const,
		patch type=bilinear
		]
		table {figs/contour_DE_m=50nu=1.dat};
	\end{axis}
	
\end{tikzpicture}

%% file: figs/maxnorm_DE_M=50nu=5.tikz
%
\begin{tikzpicture}

\begin{axis}[%
width=0.951\figurewidth,
height=\figureheight,
at={(0\figurewidth,0\figureheight)},
scale only axis,
xmin=0,
xmax=50,
xlabel style={font=\color{white!15!black}},
xlabel={$N$},
ymode=log,
ymin=1e-15,
ymax=1e+05,
yminorticks=true,
axis background/.style={fill=white},
legend style={legend cell align=left, align=left, draw=white!15!black}
]
\addplot [color=blue, draw=none, mark=asterisk, mark options={solid, blue}]
  table[row sep=crcr]{%
1	5\\
2	5.1\\
3	5\\
4	4.6\\
5	1.7\\
6	1.4\\
7	0.25\\
8	0.18\\
9	0.022\\
10	0.014\\
11	0.0012\\
12	0.00074\\
13	5e-05\\
14	2.7e-05\\
15	1.5e-06\\
16	7.3e-07\\
17	3.6e-08\\
18	1.6e-08\\
19	6.8e-10\\
20	2.6e-10\\
21	1e-11\\
22	3.7e-12\\
23	1.3e-13\\
24	4.7e-14\\
25	4.3e-14\\
26	4.3e-14\\
27	4.3e-14\\
28	4.3e-14\\
29	4.3e-14\\
30	4.3e-14\\
31	4.3e-14\\
32	4.3e-14\\
33	4.3e-14\\
34	4.3e-14\\
35	4.3e-14\\
36	4.3e-14\\
37	4.3e-14\\
38	4.3e-14\\
39	4.3e-14\\
40	4.3e-14\\
41	4.3e-14\\
42	4.3e-14\\
43	4.3e-14\\
44	4.3e-14\\
45	4.3e-14\\
46	4.3e-14\\
47	4.3e-14\\
48	4.3e-14\\
49	4.3e-14\\
50	4.3e-14\\
};

\addplot [color=red, draw=none, mark=o, mark options={solid, red}]
  table[row sep=crcr]{%
1	1.6e+02\\
2	1.6e+02\\
3	1.5e+02\\
4	1.3e+02\\
5	1.2e+02\\
6	50\\
7	36\\
8	9.2\\
9	5.6\\
10	0.96\\
11	0.51\\
12	0.064\\
13	0.03\\
14	0.0029\\
15	0.0012\\
16	9.8e-05\\
17	3.7e-05\\
18	2.5e-06\\
19	8.7e-07\\
20	5.2e-08\\
21	1.8e-08\\
22	2.2e-09\\
23	1.6e-09\\
24	1.4e-09\\
25	1.4e-09\\
26	1.4e-09\\
27	1.4e-09\\
28	1.4e-09\\
29	1.4e-09\\
30	1.4e-09\\
31	1.4e-09\\
32	1.4e-09\\
33	1.4e-09\\
34	1.4e-09\\
35	1.4e-09\\
36	1.4e-09\\
37	1.4e-09\\
38	1.4e-09\\
39	1.4e-09\\
40	1.4e-09\\
41	1.4e-09\\
42	1.4e-09\\
43	1.4e-09\\
44	1.4e-09\\
45	1.4e-09\\
46	1.4e-09\\
47	1.4e-09\\
48	1.4e-09\\
49	1.4e-09\\
50	1.4e-09\\
};

\end{axis}
\end{tikzpicture}%

%% file: figs/contour_DE_M=50nu=5.tikz
	\begin{tikzpicture}
	\begin{axis}[
		width=0.95\figurewidth,
		height=\figureheight,
		at={(0\figurewidth,0\figureheight)},
		xlabel = {$k$},
		ylabel = {$\ell$},
		view={90}{90},shader=interp,
		colormap={parula}{
			rgb255=(53,42,135)
			rgb255=(15,92,221)
			rgb255=(18,125,216)
			rgb255=(7,156,207)
			rgb255=(21,177,180)
			rgb255=(89,189,140)
			rgb255=(165,190,107)
			rgb255=(225,185,82)
			rgb255=(252,206,46)
			rgb255=(249,251,14)},
		colorbar,
		colorbar style={
			width=0.2cm},
		point meta min = -20,
		point meta max = 0,
		]
		\addplot3[
		contour filled={
			levels={-17,-16,...,0},
		},
		colormap access=piecewise const,
		patch type=bilinear
		]
		table {figs/contour_DE_m=50nu=5.dat};
	\end{axis}
	
\end{tikzpicture}

%% file: figs/SchurComplement.tikz
\newcommand\M{2} 
\newcommand\xo{-0.1} 
\newcommand\yo{0.1}

\newcommand\bandsize{0.8}
\newcommand\length{4}

	\begin{tikzpicture}
		\begin{scope}[scale=0.6]
			\draw[fill=blue!23] (\length,-\bandsize-\length)--(0,-\bandsize)--(0,0)--(\bandsize,0)--(\bandsize+\length,-\length);			
			\draw[dashed] (\xo,\yo-\M)--(\xo,\yo)--(\xo+\M,\yo)--(\xo+\M,\yo-\M)--cycle;
			\node at (\xo+\M/2,\yo-\M/2) {$A$};
			\draw[dashdotted] (\xo,\yo-\bandsize-\length)--(\xo,\yo-\M);
			\draw[dashdotted](\xo+\M,\yo-\M)--(\xo+\M,\yo-\bandsize-\length);
			\node at (\xo+\M/2,\yo-\M/2-\bandsize/2-\length/2) {$C$};
			\draw[dashdotted] (\xo+\M,\yo-\M)--(\xo+\length+\bandsize,\yo-\M);
			\draw[dashdotted] (\xo+\M,\yo)--(\xo+\length+\bandsize,\yo);
			\node at (\xo+\M/2+\bandsize/2+\length/2,\yo-\M/2) {$B$};
			\node at (\xo+\M/2+\bandsize/2+\length/2,\yo-\M/2-\bandsize/2-\length/2) {$D$};
			\draw[<->,thick] (0,0.3)--(\bandsize,0.3) node[midway,above,sloped] {\small $N+2$};
			\draw[thin] (0,0)--(0,0.3);			
			\draw[thin] (\bandsize,0)--(\bandsize,0.3);
			
			\draw[<->,thick] (-0.3,\yo-\M)--(-0.3,\yo) node[midway,above,sloped] {\small $M$};
			\draw[thin] (0,\yo)--(-0.3,\yo);			
			\draw[thin] (0,\yo-\M)--(-0.3,\yo-\M);
		\end{scope}
	\end{tikzpicture}

%% file: figs/BD.tikz
\newcommand\M{3} 
\newcommand\xo{-0.1} 
\newcommand\yo{0.1}

\newcommand{\LL}{0.65}
\newcommand\bandsize{0.6}
\newcommand\length{3}

\newcommand{\oneindig}{4}

\begin{tikzpicture}
	\begin{scope}[scale=0.6]
		\draw[dashdotted] (-0.5,-\M)--(-\oneindig,-\M)--(-\oneindig,0)--(-0.5,0);
		\draw[fill=blue!23](-\oneindig,-\M)--(-\oneindig+\bandsize,-\M)--(-\oneindig,-\M+\bandsize);
		\node at (-0.5-\oneindig/2,-\M/2) {$B$};
		\draw[<->,thick] (-\oneindig,-\M-0.3)--(-\oneindig+\bandsize,-\M-0.3) node[midway,below,sloped] {\small $N+2$};
		\draw[thin] (-\oneindig,-\M)--(-\oneindig,-\M-0.3);			
		\draw[thin] (-\oneindig+\bandsize,-\M)--(-\oneindig+\bandsize,-\M-0.3);			
		

		\draw[fill=blue!23] (\length,-\LL-\length)--(0,-\LL)--(0,0)--(\LL,0)--(\LL+\length,-\length);	
		\draw[dashed] (\xo,\yo-\LL-\length)--(\xo,\yo)--(\xo+\LL+\length,\yo);
		\node at (\LL/2+\length/3,-\LL/2-\length/3) {$D^{-1}$};
		\draw[<->,thick] (\length/2,-\LL-\length/2)--(\LL+\length/2,-\length/2)node[midway,below,sloped] {\small $L$};
		\draw[<->,thick] (0,0.3)--(\LL,0.3) node[midway,above,sloped] {\small $L$};
		\draw[thin] (0,0)--(0,0.3);			
		\draw[thin] (\LL,0)--(\LL,0.3);
		
		\node at (\length+1.8,-\length/1.65) {\Large$=$};

		\draw[dashdotted] (\length+3.5+\oneindig,0)--(\length+3.5,0)--(\length+3.5,-\M)--(\length+3.5+\oneindig,-\M);
		\draw[fill=green!23] (\length+3.5,-\M+\bandsize)--(\length+3.5,-\M)--(\length+3.5+\bandsize+\LL,-\M)--(\length+3.5+\LL,-\M+\bandsize)--cycle;

		\node at (\length+3.5+\oneindig/2,-\M/2) {$BD^{-1}$};
		\draw[<->,thick] (\length+3.5,-\M-0.3)--(\length+3.5+\bandsize+\LL,-\M-0.3) node[midway,below] {\small $N+2+L$};
		\draw[thin] (\length+3.5,-\M-0.3)--(\length+3.5,-\M);			
		\draw[thin] (\length+3.5+\bandsize+\LL,-\M-0.3)--(\length+3.5+\bandsize+\LL,-\M);			
		
		\draw[<->,thick] (\length+3.5-0.2,-\M)--(\length+3.5-0.2,-\M+\bandsize) node[midway,sloped,above]{\small $N+2$};
		\draw[thin] (\length+3.5-0.2,-\M)--(\length+3.5,-\M);			
		\draw[thin] (\length+3.5-0.2,-\M+\bandsize)--(\length+3.5,-\M+\bandsize);
		

	\end{scope}
\end{tikzpicture}

%% file: figs/ABD.tikz
\newcommand\M{3} 
\newcommand\xo{-0.1} 
\newcommand\yo{0.1}

\newcommand{\K}{0.65}
\newcommand\bandsize{0.6}
\newcommand\length{2}

\newcommand{\oneindig}{4}

\newcommand{\YtwoK}{\K+\bandsize} 

\begin{tikzpicture}
	\begin{scope}[scale=0.6]
		
		\draw[fill=blue!23] (\M-\K,-\M)--(0,-\K)--(0,0)--(\K,0)--(\M,-\M+\K)--(\M,-\M)--cycle;	
		\draw[dashed] (\xo,-\yo-\M)--(\xo,\yo)--(-\xo+\M,\yo)--(-\xo+\M,-\yo-\M)--cycle;
		\node at (\K/2+\M/3,-\K/2-\M/3) {$A^{-1}$};
		\draw[<->,thick] (\M/2,-\K-\M/2)--(\K+\M/2,-\M/2)node[midway,below,sloped] {$K$};
		\draw[<->,thick] (0,0.3)--(\K,0.3) node[midway,above,sloped] {$K$};
		\draw[thin] (0,0)--(0,0.3);			
		\draw[thin] (\K,0)--(\K,0.3);

		\draw[dashdotted] (\length+2+\oneindig,0)--(\length+2,0)--(\length+2,-\M)--(\length+2+\oneindig,-\M);
		\draw[fill=green!23] (\length+2,-\M+\bandsize)--(\length+2,-\M)--(\length+2+\bandsize+\K,-\M)--(\length+2+\K,-\M+\bandsize)--cycle;

		\node at (\length+2+\oneindig/2,-\M/2) {$BD^{-1}$};
		\draw[<->,thick] (\length+2,-\M-0.3)--(\length+2+\bandsize+\K,-\M-0.3) node[midway,below,xshift=0.1cm] {\small $N+2+L$};
		\draw[thin] (\length+2,-\M-0.3)--(\length+2,-\M);			
		\draw[thin] (\length+2+\bandsize+\K,-\M-0.3)--(\length+2+\bandsize+\K,-\M);			
		
		\draw[<->,thick] (\length+2-0.2,-\M)--(\length+2-0.2,-\M+\bandsize) node[midway,above,sloped]{\small$N+2$};
		\draw[thin] (\length+2-0.2,-\M)--(\length+2,-\M);			
		\draw[thin] (\length+2-0.2,-\M+\bandsize)--(\length+2,-\M+\bandsize);
		

		\node at (\length+7,-\length/1.25) {\Large$=$};

		\draw[dashdotted] (\length+9+\oneindig,0)--(\length+9,0)--(\length+9,-\M)--(\length+9+\oneindig,-\M);
		\draw[fill=red!23] (\length+9,-\M+\YtwoK)--(\length+9,-\M)--(\length+9+\bandsize+\K,-\M)--(\length+9+\bandsize+\K,-\M+\K)--(\length+9+\K,-\M+\YtwoK)--cycle;

		\node at (\length+9+\oneindig/2+0.5,-\M/2) {$A^{-1}BD^{-1}$};
		\draw[<->,thick] (\length+9,-\M-0.3)--(\length+9+\bandsize+\K,-\M-0.3) node[midway,below,xshift = 0.1cm] {\small $N+2+L$};
		\draw[thin] (\length+9,-\M-0.3)--(\length+9,-\M);			
		\draw[thin] (\length+9+\bandsize+\K,-\M-0.3)--(\length+9+\bandsize+\K,-\M);			
		
		\draw[<->,thick] (\length+9-0.2,-\M)--(\length+9-0.2,-\M+\YtwoK) node[midway,above,sloped,yshift=0.1cm]{\small$N+2+K$};
		\draw[thin] (\length+9-0.2,-\M)--(\length+9,-\M);			
		\draw[thin] (\length+9-0.2,-\M+\YtwoK)--(\length+9,-\M+\YtwoK);
		
		\draw[<->,thick] (\length+9,-\M+\YtwoK+0.2)--(\length+9+\K,-\M+\YtwoK+0.2) node[above,midway]{\small $K$};
		\draw[<->,thick] (\length+9+\bandsize+\K+0.2,-\M)--(\length+9+\bandsize+\K+0.2,-\M+\K) node[right,midway]{\small $K$};

	\end{scope}
\end{tikzpicture}

%% file: figs/ABDC2.tikz
\newcommand\M{3} 
\newcommand\xo{-0.1} 
\newcommand\yo{0.1}

\newcommand{\K}{0.65}
\newcommand\bandsize{0.6}
\newcommand\length{2}

\newcommand{\oneindig}{4}

\newcommand{\YtwoK}{\K+\bandsize} 

\begin{tikzpicture}
	\begin{scope}[scale=0.6]
		
		\draw[dashdotted] (0+\oneindig,0)--(0,0)--(0,-\M)--(0+\oneindig,-\M);
		\draw[fill=red!23] (0,-\M+\YtwoK)--(0,-\M)--(0+\bandsize+\K,-\M)--(0+\bandsize+\K,-\M+\K)--(0+\K,-\M+\YtwoK)--cycle;

		\node at (0+\oneindig/2+0.5,-\M/2) {$A^{-1}BD^{-1}$};
		\draw[<->,thick] (0,-\M-0.3)--(0+\bandsize+\K,-\M-0.3) node[midway,below,xshift = 0.1cm] {\small $N+2+L$};
		\draw[thin] (0,-\M-0.3)--(0,-\M);			
		\draw[thin] (0+\bandsize+\K,-\M-0.3)--(0+\bandsize+\K,-\M);			
		
		\draw[<->,thick] (0-0.2,-\M)--(0-0.2,-\M+\YtwoK) node[midway,above,sloped,yshift=0.1cm]{\small$N+2+K$};
		\draw[thin] (0-0.2,-\M)--(0,-\M);			
		\draw[thin] (0-0.2,-\M+\YtwoK)--(0,-\M+\YtwoK);
		
		\draw[<->,thick] (0,-\M+\YtwoK+0.2)--(0+\K,-\M+\YtwoK+0.2) node[above,midway]{\small $K$};
		\draw[<->,thick] (0+\bandsize+\K+0.2,-\M)--(0+\bandsize+\K+0.2,-\M+\K) node[right,midway]{\small $K$};

		\draw[dashdotted] (\length+2.5,-\oneindig)--(\length+2.5,0)--(\length+2.5+\M,0)--(\length+2.5+\M,-\oneindig);
		\draw[fill=blue!23] (\length+2.5+\M,0)--(\length+2.5+\M,-\bandsize)--(\length+2.5+\M-\bandsize,0)--cycle;
		
		\draw[<->,thick] (\length+2.5+\M-\bandsize,0.2)--(\length+2.5+\M,0.2) node[midway,above] {\small $N+2$};
		\draw[thin] (\length+2.5+\M-\bandsize,0.2)--(\length+2.5+\M-\bandsize,0);			
		\draw[thin] ((\length+2.5+\M,0.2)--(\length+2.5+\M,0);
		
		\draw[<->,thick] (\length+2.5+\M+0.2,0)--(\length+2.5+\M+0.2,-\bandsize) node[midway,above, sloped] {\small $N+2$};
		
		
		\node at (\length+2.5+\M/2,-\oneindig/2+0.45){$C$};

		%
		%
		%
		%

		\node at (\length+3.6+\M,-\length/1.25) {\Large$=$};

		\draw[dashdotted] (\length+4.5+\M,-\M)--(\length+4.5+\M,0)--(\length+4.5+\M+\M,0)--(\length+4.5+\M+\M,-\M)--cycle;
		\node at (\length+4.5+\M+\M/2,-\M/2) {$A^{-1}BD^{-1}C$};
		
		\draw[fill=green!23] (\length+4.5+\M+\M,-\M)--(\length+4.5+\M+\M,-\M+\bandsize+\K)--(\length+4.5+\M+\M-\bandsize,-\M+\bandsize+\K)--(\length+4.5+\M+\M-\bandsize,-\M);
		
		\draw[<->,thick] (\length+4.5+\M+\M,-\M-0.2)--(\length+4.5+\M+\M-\bandsize,-\M-0.2) node[midway,below] {\small  $N+2$};
		\draw[thin] (\length+4.5+\M+\M,-\M)--(\length+4.5+\M+\M,-\M-0.2);
		\draw[thin] (\length+4.5+\M+\M-\bandsize,-\M)--(\length+4.5+\M+\M-\bandsize,-\M-0.2);
		
		\draw[<->,thick] (\length+4.5+\M+\M+0.2,-\M)--(\length+4.5+\M+\M+0.2,-\M+\bandsize+\K) node[midway,sloped,below] {\small $N+K+2$};
		\draw[thin] (\length+4.5+\M+\M+0.2,-\M)--(\length+4.5+\M+\M,-\M);
		\draw[thin] (\length+4.5+\M+\M+0.2,-\M+\bandsize+\K)--(\length+4.5+\M+\M,-\M+\bandsize+\K);

	\end{scope}
\end{tikzpicture}

%% file: figs/structure_xdotError.tikz
\newcommand\M{3} 
\newcommand\xo{-0.1} 
\newcommand\yo{0.1}

\newcommand{\K}{0.65}
\newcommand\bandsize{0.6}
\newcommand\length{2}

\newcommand{\oneindig}{4}
\newcommand{\one}{0.45} 

\newcommand{\YtwoK}{\K+\bandsize} 

\begin{tikzpicture}
	\begin{scope}[scale=0.6]
		
		\draw[dashed] (\xo,-\yo-\M)--(\xo,\yo)--(-\xo+\M,\yo)--(-\xo+\M,-\yo-\M)--cycle;
		\node at (\xo+\M/2-\bandsize/2-\K/2,-\yo-\M/2+\bandsize/2+\K/2) {$0$};
		\draw[<->,thick] (\xo,\yo+0.2)--(\xo+\M-\bandsize-\K,\yo+0.2) node[midway,above] {\small $M-N-K-2$};
		\draw[thin] (\M-\bandsize-\K,-\M-0.2-\yo)--(\M-\bandsize-\K,-\M);
		\draw[thin] (\M,-\M-0.2-\yo)--(\M,-\M);
		
		\draw (0,0-\M+\bandsize+\K)--(0+\M-\bandsize-\K,0-\M+\bandsize+\K)--(0+\M-\bandsize-\K,0);
		\draw (0+\M-\bandsize-\K,0-\M)--(0+\M-\bandsize-\K,0-\M+\bandsize+\K)--(0+\M,0-\M+\bandsize+\K);	
		\draw[fill=blue!23] (\M,-0-\M+\bandsize+\K)--(-0+\M-\bandsize-\K,-0-\M+\bandsize+\K)--(-0+\M-\bandsize-\K,-\M)--(\M,-\M)--cycle;	
		\draw[<->,thick] (\M-\bandsize-\K,-\M-0.2-\yo)--(\M,-\M-0.2-\yo) node[midway,below] {\small $N+K+2$};
		\draw[thin] (\M-\bandsize-\K,-\M-0.2-\yo)--(\M-\bandsize-\K,-\M);
		\draw[thin] (\M,-\M-0.2-\yo)--(\M,-\M);

		\draw[fill=blue!23] (\length+1.5+0.05,-\M+0.05)--(\length+1.5+0.05,-0.05)--(\length+1.5+\one-0.05,-0.05)--(\length+1.5+\one-0.05,-\M+0.05)--cycle;
		\draw[dashdotted] (\length+1.5,-\M)--(\length+1.5,0)--(\length+1.5+\one,0)--(\length+1.5+\one,-\M)--cycle;
		\node at (\length+1.5+\one/2,-\M/2) {$\dot{x}$};

		\node at (\length+2.5,-\length/1.4) {\Large$-$};

		\draw[dashdotted] (\length+3.8+\oneindig,0)--(\length+3.8,0)--(\length+3.8,-\M)--(\length+3.8+\oneindig,-\M);
		\draw[fill=blue!23] (\length+3.8,-\M+\YtwoK)--(\length+3.8,-\M)--(\length+3.8+\bandsize+\K,-\M)--(\length+3.8+\bandsize+\K,-\M+\K+\bandsize)--cycle;

		\draw[<->,thick] (\length+3.8,-\M-0.3)--(\length+3.8+\bandsize+\K,-\M-0.3) node[midway,below] {\small $N+L+2$};
		\draw[thin] (\length+3.8,-\M-0.3)--(\length+3.8,-\M);			
		\draw[thin] (\length+3.8+\bandsize+\K,-\M-0.3)--(\length+3.8+\bandsize+\K,-\M);			
		
		\draw[<->,thick] (\length+3.8-0.2,-\M)--(\length+3.8-0.2,-\M+\YtwoK) node[midway,above,sloped]{\small $N+K+2$};
		\draw[thin] (\length+3.8-0.2,-\M)--(\length+3.8,-\M);			
		\draw[thin] (\length+3.8-0.2,-\M+\YtwoK)--(\length+3.8,-\M+\YtwoK);

		\draw[fill=blue!23] (\length+5.3+\M+0.05,-\M+0.05)--(\length+5.3+\M+0.05,-0.05)--(\length+5.3+\M+\one-0.05,-0.05)--(\length+5.3+\one+\M-0.05,-\M+0.05);
		\draw[dashdotted] (\length+5.3+\M,-\M)--(\length+5.3+\M,0)--(\length+5.3+\M+\one,0)--(\length+5.3+\one+\M,-\M);
		\node at (\length+5.3+\one/2+\M,-\M/2) {$v$};

		\node at (\length+6.35+\M,-\length/1.3) {\Large$=$};

		\draw[fill=blue!23] (\length+7.0+\M+0.05,-\M+0.05)--(\length+7.0+\M+0.05,-\M+\bandsize+\K)--(\length+7.0+\M+\one-0.05,-\M+\bandsize+\K)--(\length+7.0+\one+\M-0.05,-\M+0.05)--cycle;
		\draw[dashdotted] (\length+7.0+\M,-\M)--(\length+7.0+\M,0)--(\length+7.0+\M+\one,0)--(\length+7.0+\one+\M,-\M)--cycle;
		\node at (\length+7.0+\one/2+\M,-\M/4) {$0$};
		
		\draw[<->, thick] (\length+7.0+\M+0.05+0.5,-\M+0.05)--(\length+7.0+\M+0.05+0.5,-\M+\bandsize+\K) node[midway,sloped,below]{\small $N+K+2$};	
		\draw[thin] (\length+7.0+\M+0.05+0.5,-\M+0.05)--(\length+7.0+\M+0.05+0.2,-\M+0.05);			
		\draw[thin] (\length+7.0+\M+0.05+0.5,-\M+\bandsize+\K)--(\length+7.0+\M+0.05+0.2,-\M+\bandsize+\K);
		
	\end{scope}
\end{tikzpicture}

%% file: figs/Structure_contour_DE_m=100nu=1.tikz
	\begin{tikzpicture}
	\begin{axis}[
		width=0.95\figurewidth,
		height=\figureheight,
		at={(0\figurewidth,0\figureheight)},
		xlabel = {$k$},
		ylabel = {$\ell$},
		view={90}{90},shader=interp,
		colormap={parula}{
			rgb255=(53,42,135)
			rgb255=(15,92,221)
			rgb255=(18,125,216)
			rgb255=(7,156,207)
			rgb255=(21,177,180)
			rgb255=(89,189,140)
			rgb255=(165,190,107)
			rgb255=(225,185,82)
			rgb255=(252,206,46)
			rgb255=(249,251,14)},
		]
		\addplot3[
		contour filled={
			levels={-17,-16,...,0},
		},
		colormap access=piecewise const,
		patch type=bilinear
		]
		table {figs/contour_IF_DE_m=100nu=1.dat};
	\end{axis}
	
	\begin{axis}[
		width=0.95\figurewidth,
		height=\figureheight,
		at={(0.9\figurewidth,0\figureheight)},
		xlabel = {$k$},
		ylabel = {$\ell$},
		view={90}{90},shader=interp,
		colormap={parula}{
			rgb255=(53,42,135)
			rgb255=(15,92,221)
			rgb255=(18,125,216)
			rgb255=(7,156,207)
			rgb255=(21,177,180)
			rgb255=(89,189,140)
			rgb255=(165,190,107)
			rgb255=(225,185,82)
			rgb255=(252,206,46)
			rgb255=(249,251,14)},
		]
		\addplot3[
		contour filled={
			levels={-17,-16,...,0},
		},
		colormap access=piecewise const,
		patch type=bilinear
		]
		table {figs/contour_IFinv_DE_m=100nu=1.dat};
	\end{axis}
	
	\begin{axis}[
		width=0.95\figurewidth,
		height=\figureheight,
		at={(1.8\figurewidth,0\figureheight)},
		xlabel = {$k$},
		ylabel = {$\ell$},
		view={90}{90},shader=interp,
		colormap={parula}{
			rgb255=(53,42,135)
			rgb255=(15,92,221)
			rgb255=(18,125,216)
			rgb255=(7,156,207)
			rgb255=(21,177,180)
			rgb255=(89,189,140)
			rgb255=(165,190,107)
			rgb255=(225,185,82)
			rgb255=(252,206,46)
			rgb255=(249,251,14)},
		colorbar,
		colorbar style={
			width=0.2cm, at={(2.5\figurewidth,0.623\figureheight)}, ylabel=$\log_{10}(\vert a_{k,\ell}\vert)$},
		point meta min = -20,
		point meta max = 0,
		]
		\addplot3[
		contour filled={
			levels={-17,-16,...,0},
		},
		colormap access=piecewise const,
		patch type=bilinear
		]
		table {figs/contour_Dinv_DE_m=100nu=1.dat};
	\end{axis}
	
\end{tikzpicture}

%% file: figs/truncError.tikz
	\begin{tikzpicture}
	\begin{axis}[
		width=0.95\figurewidth,
		height=\figureheight,
		at={(0\figurewidth,0\figureheight)},
		xlabel = {$k$},
		ylabel = {$\ell$},
		view={90}{90},shader=interp,
		colormap={parula}{
			rgb255=(255,255,255)
			rgb255=(53,42,135)
			rgb255=(15,92,221)
			rgb255=(18,125,216)
			rgb255=(7,156,207)
			rgb255=(21,177,180)
			rgb255=(89,189,140)
			rgb255=(165,190,107)
			rgb255=(225,185,82)
			rgb255=(252,206,46)
			rgb255=(249,251,14)},
		]
		\addplot3[
		contour filled={
			levels={-19,-18,...,0},
		},
		colormap access=piecewise const,
		patch type=bilinear
		]
		table {figs/contour_term1_DE_m=100nu=1.dat};
		\addplot [color=red, draw=red, thick,dashed,  forget plot]
		table[row sep=crcr]{%
			12 	49\\
			12	12\\
			49	12\\
		};
	\end{axis}
	
	\begin{axis}[
		width=0.95\figurewidth,
		height=\figureheight,
		at={(1\figurewidth,0\figureheight)},
		xlabel = {$k$},
		ylabel = {$\ell$},
		view={90}{90},shader=interp,
		colormap={parula}{
			rgb255=(255,255,255)
			rgb255=(53,42,135)
			rgb255=(15,92,221)
			rgb255=(18,125,216)
			rgb255=(7,156,207)
			rgb255=(21,177,180)
			rgb255=(89,189,140)
			rgb255=(165,190,107)
			rgb255=(225,185,82)
			rgb255=(252,206,46)
			rgb255=(249,251,14)},
		]
		\addplot3[
		contour filled={
			levels={-19,-18,...,0},
		},
		colormap access=piecewise const,
		patch type=bilinear
		]
		table {figs/contour_term2_DE_m=100nu=1.dat};
		\addplot [color=red, draw=red, thick,dashed,  forget plot]
		table[row sep=crcr]{%
			49	32\\
			12	32\\
			12 0\\
		};
	\end{axis}
	
	\begin{axis}[
		width=50pt,
		height=\figureheight,
		at={(2\figurewidth,0\figureheight)},
		xlabel = {$k$},
		xtick = {0,12,20,30,40},
		xticklabels = {0,{\footnotesize \color{red} 12},20,{\footnotesize 30},40},		
		yticklabels={},
		view={90}{90},shader=interp,
		colormap={parula}{
			rgb255=(255,255,255)
			rgb255=(53,42,135)
			rgb255=(15,92,221)
			rgb255=(18,125,216)
			rgb255=(7,156,207)
			rgb255=(21,177,180)
			rgb255=(89,189,140)
			rgb255=(165,190,107)
			rgb255=(225,185,82)
			rgb255=(252,206,46)
			rgb255=(249,251,14)},
		]
		\addplot3[
		contour filled={
			levels={-19,-18,...,0},
		},
		colormap access=piecewise const,
		patch type=bilinear
		]
		table {figs/contour_vec_DE_m=100nu=1.dat};
			\addplot [color=red, draw=red, thick,  forget plot]
		table[row sep=crcr]{%
			12 	1\\
			12	2\\
		};
	\end{axis}

	\begin{axis}[
	width=50pt,
	height=\figureheight,
	at={(2.3\figurewidth,0\figureheight)},
	xtick = {0,10,20},
	xticklabels = {0,-10,-20},		
	yticklabels={},
	xlabel={$\log_{10}(\vert a_{k,\ell}\vert)$},
	x label style={rotate=90},
	xticklabel pos=right,
	view={90}{90},shader=interp,
	colormap={parula}{
		rgb255=(255,255,255)
		rgb255=(53,42,135)
		rgb255=(15,92,221)
		rgb255=(18,125,216)
		rgb255=(7,156,207)
		rgb255=(21,177,180)
		rgb255=(89,189,140)
		rgb255=(165,190,107)
		rgb255=(225,185,82)
		rgb255=(252,206,46)
		rgb255=(249,251,14)},
	]
	\addplot3[
	contour filled={
		levels={-20,-19,...,0},
	},
	colormap access=piecewise const,
	patch type=bilinear
	]
	table {figs/colorbar.dat};
\end{axis}
\end{tikzpicture}

%% file: figs/coeffs_DE_M=50nu=1.tikz
%
\begin{tikzpicture}

\begin{axis}[%
width=0.411\figurewidth,
height=\figureheight,
at={(0\figurewidth,0\figureheight)},
scale only axis,
xmin=0,
xmax=50,
xlabel style={font=\color{white!15!black}},
xlabel={$i$},
ymode=log,
ymin=1e-20,
ymax=1.2779,
yminorticks=true,
axis background/.style={fill=white}
]
\addplot [color=blue, draw=none, mark=asterisk, mark options={solid, blue}, forget plot]
  table[row sep=crcr]{%
1	1.2779\\
2	0.58622\\
3	0.14196\\
4	0.05502\\
5	0.012049\\
6	0.0035012\\
7	0.00084205\\
8	0.00017829\\
9	4.496e-05\\
10	8.8653e-06\\
11	1.9715e-06\\
12	4.0179e-07\\
13	7.9209e-08\\
14	1.6153e-08\\
15	3.0095e-09\\
16	5.8979e-10\\
17	1.0825e-10\\
18	2.005e-11\\
19	3.6617e-12\\
20	6.4826e-13\\
21	1.1725e-13\\
22	1.9774e-14\\
23	3.3466e-15\\
24	6.5571e-16\\
25	8.0562e-17\\
26	2.1265e-16\\
27	2.3902e-16\\
28	1.8721e-16\\
29	6.3543e-16\\
30	1.8623e-16\\
31	1.8397e-16\\
32	3.9301e-16\\
33	1.1821e-15\\
34	1.3405e-15\\
35	1.4141e-15\\
36	1.3505e-15\\
37	4.2914e-16\\
38	4.9204e-16\\
39	5.2594e-16\\
40	1.7469e-16\\
41	8.0767e-16\\
42	5.3266e-16\\
43	2.759e-16\\
44	4.1512e-16\\
45	8.987e-16\\
46	1.2071e-16\\
47	1.0409e-15\\
48	8.9827e-16\\
49	8.9766e-16\\
50	9.2646e-16\\
};
\addplot [color=green, draw=none, mark=triangle, mark options={solid, green}, forget plot]
  table[row sep=crcr]{%
1	1.2779\\
2	0.58622\\
3	0.14196\\
4	0.05502\\
5	0.012049\\
6	0.0035012\\
7	0.00084205\\
8	0.00017829\\
9	4.496e-05\\
10	8.8653e-06\\
11	1.9715e-06\\
12	4.0179e-07\\
13	7.9209e-08\\
14	1.6153e-08\\
15	3.0095e-09\\
16	5.8979e-10\\
17	1.0825e-10\\
18	2.005e-11\\
19	3.6618e-12\\
20	6.4825e-13\\
21	1.1649e-13\\
22	2.0109e-14\\
23	3.5251e-15\\
24	6.0044e-16\\
25	9.7735e-17\\
26	3.0803e-17\\
27	1.4072e-17\\
28	1.3911e-17\\
29	2.7756e-17\\
30	1.3889e-17\\
31	1.3882e-17\\
32	1.3985e-18\\
33	1.3884e-17\\
34	6.912e-19\\
35	1.5472e-19\\
36	4.1634e-17\\
37	0\\
38	1.3878e-17\\
39	1.3878e-17\\
40	0\\
41	1.3878e-17\\
42	0\\
43	0\\
44	0\\
45	0\\
46	0\\
47	0\\
48	0\\
49	1.3878e-17\\
50	0\\
};
\addplot [color=red, draw=none, mark=o, mark options={solid, red}, forget plot]
  table[row sep=crcr]{%
1	1.2779\\
2	0.58622\\
3	0.14196\\
4	0.05502\\
5	0.012049\\
6	0.0035012\\
7	0.00084205\\
8	0.00017829\\
9	4.496e-05\\
10	8.8653e-06\\
11	1.9715e-06\\
12	4.0179e-07\\
13	7.9209e-08\\
14	1.6153e-08\\
15	3.0095e-09\\
16	5.8979e-10\\
17	1.0825e-10\\
18	2.005e-11\\
19	3.6618e-12\\
20	6.4825e-13\\
21	1.1649e-13\\
22	2.0109e-14\\
23	3.5251e-15\\
24	6.0044e-16\\
25	9.7735e-17\\
26	3.0803e-17\\
27	1.4072e-17\\
28	1.3911e-17\\
29	2.776e-17\\
30	1.39e-17\\
31	2.8357e-17\\
32	2.1797e-17\\
33	2.4361e-16\\
34	1.6015e-15\\
35	7.774e-15\\
36	6.9574e-14\\
37	2.9079e-13\\
38	2.6769e-12\\
39	1.5108e-11\\
40	8.2047e-11\\
41	8.4627e-10\\
42	2.5468e-09\\
43	3.4598e-08\\
44	1.3909e-07\\
45	1.6613e-06\\
46	9.929e-06\\
47	6.6806e-05\\
48	0.00025654\\
49	0.00036014\\
50	0.07107\\
};
\end{axis}

\begin{axis}[%
width=0.411\figurewidth,
height=\figureheight,
at={(0.54\figurewidth,0\figureheight)},
scale only axis,
xmin=0,
xmax=50,
xlabel style={font=\color{white!15!black}},
xlabel={$i$},
ymode=log,
ymin=1.4e-17,
ymax=1,
yminorticks=true,
axis background/.style={fill=white}
]
\addplot [color=green, draw=none, mark=triangle, mark options={solid, green}, forget plot]
  table[row sep=crcr]{%
1	1.1102e-16\\
2	5.5511e-16\\
3	1.4592e-16\\
4	2.0304e-16\\
5	1.4206e-17\\
6	1.6839e-16\\
7	1.6792e-16\\
8	1.5403e-16\\
9	7.9685e-17\\
10	1.9712e-16\\
11	1.3916e-16\\
12	1.974e-16\\
13	5.4646e-16\\
14	1.9014e-16\\
15	4.1903e-16\\
16	6.0328e-16\\
17	6.1318e-16\\
18	2.7652e-16\\
19	2.1224e-16\\
20	6.8126e-16\\
21	7.9007e-16\\
22	7.11e-16\\
23	3.1997e-16\\
24	1.1611e-16\\
25	1.6305e-16\\
26	2.1767e-16\\
27	2.2619e-16\\
28	1.9033e-16\\
29	6.1371e-16\\
30	1.9368e-16\\
31	1.7158e-16\\
32	3.9338e-16\\
33	1.1958e-15\\
34	1.3405e-15\\
35	1.4142e-15\\
36	1.3124e-15\\
37	4.2914e-16\\
38	4.7953e-16\\
39	5.1798e-16\\
40	1.7469e-16\\
41	8.1792e-16\\
42	5.3266e-16\\
43	2.759e-16\\
44	4.1512e-16\\
45	8.987e-16\\
46	1.2071e-16\\
47	1.0409e-15\\
48	8.9827e-16\\
49	8.9102e-16\\
50	9.2646e-16\\
};
\addplot [color=red, draw=none, mark=o, mark options={solid, red}, forget plot]
  table[row sep=crcr]{%
1	1.1102e-16\\
2	5.5511e-16\\
3	1.4592e-16\\
4	2.0304e-16\\
5	1.4206e-17\\
6	1.6839e-16\\
7	1.6792e-16\\
8	1.5403e-16\\
9	7.9685e-17\\
10	1.9712e-16\\
11	1.3916e-16\\
12	1.974e-16\\
13	5.4646e-16\\
14	1.9014e-16\\
15	4.1903e-16\\
16	6.0328e-16\\
17	6.1318e-16\\
18	2.7652e-16\\
19	2.1224e-16\\
20	6.8126e-16\\
21	7.9007e-16\\
22	7.11e-16\\
23	3.1997e-16\\
24	1.1611e-16\\
25	1.6305e-16\\
26	2.1767e-16\\
27	2.2619e-16\\
28	1.9033e-16\\
29	6.1348e-16\\
30	1.9349e-16\\
31	1.5641e-16\\
32	4.1095e-16\\
33	9.8583e-16\\
34	1.058e-15\\
35	8.8186e-15\\
36	6.965e-14\\
37	2.9113e-13\\
38	2.6768e-12\\
39	1.5109e-11\\
40	8.2047e-11\\
41	8.4627e-10\\
42	2.5468e-09\\
43	3.4598e-08\\
44	1.3909e-07\\
45	1.6613e-06\\
46	9.929e-06\\
47	6.6806e-05\\
48	0.00025654\\
49	0.00036014\\
50	0.07107\\
};
\end{axis}
\end{tikzpicture}%

%% file: figs/coeffs_DE_M=100nu=5.tikz
%
\begin{tikzpicture}

\begin{axis}[%
width=0.411\figurewidth,
height=\figureheight,
at={(0\figurewidth,0\figureheight)},
scale only axis,
xmin=0,
xmax=100,
xlabel style={font=\color{white!15!black}},
xlabel={$i$},
ymode=log,
ymin=1e-20,
ymax=1.0695,
yminorticks=true,
axis background/.style={fill=white}
]
\addplot [color=blue, draw=none, mark=asterisk, mark options={solid, blue}, forget plot]
  table[row sep=crcr]{%
1	1.0695\\
2	0.19471\\
3	0.11358\\
4	0.71983\\
5	0.21625\\
6	0.43187\\
7	0.080576\\
8	0.11009\\
9	0.14135\\
10	0.078552\\
11	0.081259\\
12	0.031691\\
13	0.029237\\
14	0.017108\\
15	0.016424\\
16	0.0097269\\
17	0.0069886\\
18	0.0043262\\
19	0.0022188\\
20	0.0023207\\
21	0.00089146\\
22	0.0010293\\
23	0.00044067\\
24	0.00033767\\
25	0.00023937\\
26	0.00010149\\
27	0.00011026\\
28	4.015e-05\\
29	3.9926e-05\\
30	1.9976e-05\\
31	1.2522e-05\\
32	9.2472e-06\\
33	4.125e-06\\
34	3.5906e-06\\
35	1.6089e-06\\
36	1.2198e-06\\
37	6.7246e-07\\
38	3.9927e-07\\
39	2.6371e-07\\
40	1.3551e-07\\
41	9.4347e-08\\
42	4.8283e-08\\
43	3.1952e-08\\
44	1.7514e-08\\
45	1.0606e-08\\
46	6.2262e-09\\
47	3.4907e-09\\
48	2.1456e-09\\
49	1.1508e-09\\
50	7.1869e-10\\
51	3.8343e-10\\
52	2.3357e-10\\
53	1.2855e-10\\
54	7.3927e-11\\
55	4.2783e-11\\
56	2.3192e-11\\
57	1.3911e-11\\
58	7.3434e-12\\
59	4.38e-12\\
60	2.3535e-12\\
61	1.3446e-12\\
62	7.5138e-13\\
63	4.0781e-13\\
64	2.3548e-13\\
65	1.2422e-13\\
66	7.2019e-14\\
67	3.9083e-14\\
68	2.0129e-14\\
69	1.1614e-14\\
70	5.324e-15\\
71	4.2055e-15\\
72	8.3022e-16\\
73	2.5681e-15\\
74	2.9316e-16\\
75	2.0515e-15\\
76	8.0248e-16\\
77	1.7567e-15\\
78	7.5574e-16\\
79	6.5207e-16\\
80	8.8577e-16\\
81	1.8052e-16\\
82	1.7796e-15\\
83	1.6474e-15\\
84	8.949e-16\\
85	1.4022e-15\\
86	1.0757e-15\\
87	1.3328e-15\\
88	6.1024e-16\\
89	1.0488e-15\\
90	1.6175e-15\\
91	3.5955e-16\\
92	6.2602e-16\\
93	8.7881e-16\\
94	5.6e-16\\
95	8.0707e-16\\
96	8.918e-16\\
97	6.5055e-16\\
98	3.3018e-16\\
99	9.7946e-16\\
100	1.7607e-15\\
};
\addplot [color=green, draw=none, mark=triangle, mark options={solid, green}, forget plot]
  table[row sep=crcr]{%
1	1.0695\\
2	0.19471\\
3	0.11358\\
4	0.71983\\
5	0.21625\\
6	0.43187\\
7	0.080576\\
8	0.11009\\
9	0.14135\\
10	0.078552\\
11	0.081259\\
12	0.031691\\
13	0.029237\\
14	0.017108\\
15	0.016424\\
16	0.0097269\\
17	0.0069886\\
18	0.0043262\\
19	0.0022188\\
20	0.0023207\\
21	0.00089146\\
22	0.0010293\\
23	0.00044067\\
24	0.00033767\\
25	0.00023937\\
26	0.00010149\\
27	0.00011026\\
28	4.015e-05\\
29	3.9926e-05\\
30	1.9976e-05\\
31	1.2522e-05\\
32	9.2472e-06\\
33	4.125e-06\\
34	3.5906e-06\\
35	1.6089e-06\\
36	1.2198e-06\\
37	6.7246e-07\\
38	3.9927e-07\\
39	2.6371e-07\\
40	1.3551e-07\\
41	9.4347e-08\\
42	4.8283e-08\\
43	3.1952e-08\\
44	1.7514e-08\\
45	1.0606e-08\\
46	6.2262e-09\\
47	3.4907e-09\\
48	2.1456e-09\\
49	1.1508e-09\\
50	7.1869e-10\\
51	3.8343e-10\\
52	2.3357e-10\\
53	1.2855e-10\\
54	7.3927e-11\\
55	4.2783e-11\\
56	2.3192e-11\\
57	1.391e-11\\
58	7.3444e-12\\
59	4.3803e-12\\
60	2.3525e-12\\
61	1.3439e-12\\
62	7.5145e-13\\
63	4.0839e-13\\
64	2.3554e-13\\
65	1.2472e-13\\
66	7.196e-14\\
67	3.8309e-14\\
68	2.1573e-14\\
69	1.1735e-14\\
70	6.407e-15\\
71	3.5618e-15\\
72	1.8696e-15\\
73	1.0945e-15\\
74	5.2974e-16\\
75	2.734e-16\\
76	1.1247e-16\\
77	6.9389e-18\\
78	1.3878e-17\\
79	1.3878e-17\\
80	1.3878e-17\\
81	2.0817e-17\\
82	0\\
83	6.9389e-18\\
84	6.9389e-18\\
85	2.0817e-17\\
86	6.9389e-18\\
87	1.3878e-17\\
88	6.9389e-18\\
89	0\\
90	6.9389e-18\\
91	6.9389e-18\\
92	6.9389e-18\\
93	6.9389e-18\\
94	2.0817e-17\\
95	0\\
96	6.9389e-18\\
97	1.3878e-17\\
98	6.9389e-18\\
99	6.9389e-18\\
100	0\\
};
\addplot [color=red, draw=none, mark=o, mark options={solid, red}, forget plot]
  table[row sep=crcr]{%
1	1.0695\\
2	0.19471\\
3	0.11358\\
4	0.71983\\
5	0.21625\\
6	0.43187\\
7	0.080576\\
8	0.11009\\
9	0.14135\\
10	0.078552\\
11	0.081259\\
12	0.031691\\
13	0.029237\\
14	0.017108\\
15	0.016424\\
16	0.0097269\\
17	0.0069886\\
18	0.0043262\\
19	0.0022188\\
20	0.0023207\\
21	0.00089146\\
22	0.0010293\\
23	0.00044067\\
24	0.00033767\\
25	0.00023937\\
26	0.00010149\\
27	0.00011026\\
28	4.015e-05\\
29	3.9926e-05\\
30	1.9976e-05\\
31	1.2522e-05\\
32	9.2472e-06\\
33	4.125e-06\\
34	3.5906e-06\\
35	1.6089e-06\\
36	1.2198e-06\\
37	6.7246e-07\\
38	3.9927e-07\\
39	2.6371e-07\\
40	1.3551e-07\\
41	9.4347e-08\\
42	4.8283e-08\\
43	3.1952e-08\\
44	1.7514e-08\\
45	1.0606e-08\\
46	6.2262e-09\\
47	3.4907e-09\\
48	2.1456e-09\\
49	1.1508e-09\\
50	7.1869e-10\\
51	3.8343e-10\\
52	2.3357e-10\\
53	1.2855e-10\\
54	7.3926e-11\\
55	4.2782e-11\\
56	2.3189e-11\\
57	1.3905e-11\\
58	7.3404e-12\\
59	4.3581e-12\\
60	2.36e-12\\
61	1.3334e-12\\
62	9.8785e-13\\
63	1.3866e-12\\
64	2.1532e-12\\
65	4.0738e-12\\
66	7.4688e-12\\
67	1.2282e-11\\
68	2.4998e-11\\
69	3.8936e-11\\
70	7.8961e-11\\
71	1.2761e-10\\
72	2.5566e-10\\
73	4.2993e-10\\
74	8.7521e-10\\
75	1.4251e-09\\
76	2.8022e-09\\
77	4.9044e-09\\
78	8.0612e-09\\
79	1.9807e-08\\
80	2.3417e-08\\
81	7.6255e-08\\
82	7.4617e-08\\
83	2.266e-07\\
84	3.4543e-07\\
85	4.8803e-07\\
86	1.8824e-06\\
87	8.5547e-07\\
88	7.4452e-06\\
89	2.5744e-06\\
90	1.7999e-05\\
91	2.8734e-05\\
92	2.1725e-05\\
93	0.00021627\\
94	7.1622e-05\\
95	0.0007805\\
96	0.00041988\\
97	0.00077714\\
98	0.00052454\\
99	0.00014405\\
100	0.050135\\
};
\end{axis}

\begin{axis}[%
width=0.411\figurewidth,
height=\figureheight,
at={(0.54\figurewidth,0\figureheight)},
scale only axis,
xmin=0,
xmax=100,
xlabel style={font=\color{white!15!black}},
xlabel={$i$},
ymode=log,
ymin=1.4e-17,
ymax=1,
yminorticks=true,
axis background/.style={fill=white}
]
\addplot [color=green, draw=none, mark=triangle, mark options={solid, green}, forget plot]
  table[row sep=crcr]{%
1	7.8505e-16\\
2	4.4841e-16\\
3	6.5096e-16\\
4	3.5108e-16\\
5	2.7894e-16\\
6	4.996e-16\\
7	1.1281e-16\\
8	1.8981e-16\\
9	1.3092e-16\\
10	3.2524e-16\\
11	8.5829e-17\\
12	6.8362e-17\\
13	2.5004e-16\\
14	2.9218e-16\\
15	8.6948e-17\\
16	3.4485e-16\\
17	3.4991e-16\\
18	3.6397e-16\\
19	2.8738e-16\\
20	1.6373e-16\\
21	4.2864e-16\\
22	3.1303e-16\\
23	2.6586e-16\\
24	5.7199e-16\\
25	2.4804e-16\\
26	2.4632e-16\\
27	8.0793e-16\\
28	5.3183e-16\\
29	2.4735e-16\\
30	5.2816e-16\\
31	1.5234e-16\\
32	1.6257e-16\\
33	2.2511e-16\\
34	1.8115e-16\\
35	6.669e-16\\
36	2.8379e-16\\
37	6.2403e-16\\
38	8.9401e-16\\
39	5.6945e-16\\
40	9.6932e-16\\
41	1.7158e-15\\
42	4.9795e-16\\
43	1.8977e-16\\
44	6.0425e-16\\
45	2.5653e-16\\
46	8.0142e-16\\
47	8.5107e-16\\
48	8.5174e-16\\
49	8.955e-16\\
50	5.9469e-16\\
51	9.9297e-16\\
52	1.1598e-15\\
53	2.7542e-16\\
54	8.9373e-16\\
55	5.7239e-16\\
56	4.0884e-16\\
57	1.3749e-15\\
58	1.3082e-15\\
59	4.697e-16\\
60	1.2934e-15\\
61	7.0872e-16\\
62	6.4956e-16\\
63	7.2448e-16\\
64	1.2476e-16\\
65	5.1817e-16\\
66	1.4929e-15\\
67	1.9915e-15\\
68	1.5859e-15\\
69	6.2782e-16\\
70	1.374e-15\\
71	6.7495e-16\\
72	1.1203e-15\\
73	1.5448e-15\\
74	5.1604e-16\\
75	1.8604e-15\\
76	7.7382e-16\\
77	1.7511e-15\\
78	7.6743e-16\\
79	6.6593e-16\\
80	8.9964e-16\\
81	1.8847e-16\\
82	1.7796e-15\\
83	1.654e-15\\
84	8.959e-16\\
85	1.4089e-15\\
86	1.0826e-15\\
87	1.323e-15\\
88	6.0644e-16\\
89	1.0488e-15\\
90	1.6205e-15\\
91	3.6002e-16\\
92	6.2805e-16\\
93	8.8166e-16\\
94	5.5223e-16\\
95	8.0707e-16\\
96	8.9451e-16\\
97	6.5945e-16\\
98	3.2517e-16\\
99	9.7796e-16\\
100	1.7607e-15\\
};
\addplot [color=red, draw=none, mark=o, mark options={solid, red}, forget plot]
  table[row sep=crcr]{%
1	7.8505e-16\\
2	4.4841e-16\\
3	6.5096e-16\\
4	3.5108e-16\\
5	2.7894e-16\\
6	4.996e-16\\
7	1.1281e-16\\
8	1.8981e-16\\
9	1.3092e-16\\
10	3.2524e-16\\
11	8.5829e-17\\
12	6.8362e-17\\
13	2.5004e-16\\
14	2.9218e-16\\
15	8.6948e-17\\
16	3.4485e-16\\
17	3.4991e-16\\
18	3.6397e-16\\
19	2.8738e-16\\
20	1.6373e-16\\
21	4.2864e-16\\
22	3.1303e-16\\
23	2.6586e-16\\
24	5.7199e-16\\
25	2.4804e-16\\
26	2.4632e-16\\
27	8.0793e-16\\
28	5.3206e-16\\
29	2.4744e-16\\
30	5.2909e-16\\
31	1.5234e-16\\
32	1.6255e-16\\
33	2.2527e-16\\
34	1.8114e-16\\
35	6.7398e-16\\
36	2.8417e-16\\
37	6.3778e-16\\
38	8.9485e-16\\
39	5.6837e-16\\
40	9.8242e-16\\
41	1.7135e-15\\
42	4.7987e-16\\
43	2.0637e-16\\
44	5.6586e-16\\
45	3.117e-16\\
46	7.4776e-16\\
47	9.572e-16\\
48	9.4789e-16\\
49	1.3921e-15\\
50	1.4617e-15\\
51	8.421e-16\\
52	3.8142e-15\\
53	4.7553e-15\\
54	8.1411e-15\\
55	1.3492e-14\\
56	2.7092e-14\\
57	4.1644e-14\\
58	7.5979e-14\\
59	1.3811e-13\\
60	2.1851e-13\\
61	4.5219e-13\\
62	6.5629e-13\\
63	1.3883e-12\\
64	2.1662e-12\\
65	4.077e-12\\
66	7.4834e-12\\
67	1.2281e-11\\
68	2.5005e-11\\
69	3.8934e-11\\
70	7.8961e-11\\
71	1.2761e-10\\
72	2.5566e-10\\
73	4.2993e-10\\
74	8.7521e-10\\
75	1.4251e-09\\
76	2.8022e-09\\
77	4.9044e-09\\
78	8.0612e-09\\
79	1.9807e-08\\
80	2.3417e-08\\
81	7.6255e-08\\
82	7.4617e-08\\
83	2.266e-07\\
84	3.4543e-07\\
85	4.8803e-07\\
86	1.8824e-06\\
87	8.5547e-07\\
88	7.4452e-06\\
89	2.5744e-06\\
90	1.7999e-05\\
91	2.8734e-05\\
92	2.1725e-05\\
93	0.00021627\\
94	7.1622e-05\\
95	0.0007805\\
96	0.00041988\\
97	0.00077714\\
98	0.00052454\\
99	0.00014405\\
100	0.050135\\
};
\end{axis}
\end{tikzpicture}%

%% file: figs/coeffs_DE_M=50nu=5.tikz
%
\begin{tikzpicture}

\begin{axis}[%
width=0.411\figurewidth,
height=\figureheight,
at={(0\figurewidth,0\figureheight)},
scale only axis,
xmin=0,
xmax=50,
xlabel style={font=\color{white!15!black}},
xlabel={$i$},
ymode=log,
ymin=1e-20,
ymax=1.0695,
yminorticks=true,
axis background/.style={fill=white}
]
\addplot [color=blue, draw=none, mark=asterisk, mark options={solid, blue}, forget plot]
  table[row sep=crcr]{%
1	1.0695\\
2	0.19471\\
3	0.11358\\
4	0.71983\\
5	0.21625\\
6	0.43187\\
7	0.080576\\
8	0.11009\\
9	0.14135\\
10	0.078552\\
11	0.081259\\
12	0.031691\\
13	0.029237\\
14	0.017108\\
15	0.016424\\
16	0.0097269\\
17	0.0069886\\
18	0.0043262\\
19	0.0022188\\
20	0.0023207\\
21	0.00089146\\
22	0.0010293\\
23	0.00044067\\
24	0.00033767\\
25	0.00023937\\
26	0.00010149\\
27	0.00011026\\
28	4.015e-05\\
29	3.9926e-05\\
30	1.9976e-05\\
31	1.2522e-05\\
32	9.2472e-06\\
33	4.125e-06\\
34	3.5906e-06\\
35	1.6089e-06\\
36	1.2198e-06\\
37	6.7246e-07\\
38	3.9927e-07\\
39	2.6371e-07\\
40	1.3551e-07\\
41	9.434e-08\\
42	4.8279e-08\\
43	3.1962e-08\\
44	1.752e-08\\
45	1.0588e-08\\
46	6.2161e-09\\
47	3.527e-09\\
48	2.1674e-09\\
49	9.9571e-10\\
50	6.2369e-10\\
};
\addplot [color=green, draw=none, mark=triangle, mark options={solid, green}, forget plot]
  table[row sep=crcr]{%
1	1.0695\\
2	0.19471\\
3	0.11358\\
4	0.71983\\
5	0.21625\\
6	0.43187\\
7	0.080575\\
8	0.11009\\
9	0.14135\\
10	0.078552\\
11	0.081259\\
12	0.03169\\
13	0.029236\\
14	0.017107\\
15	0.016423\\
16	0.0097258\\
17	0.0069871\\
18	0.0043238\\
19	0.0022168\\
20	0.0023129\\
21	0.00088717\\
22	0.0010181\\
23	0.0004354\\
24	0.00033136\\
25	0.00015613\\
26	7.3216e-05\\
27	1.3878e-17\\
28	0\\
29	0\\
30	0\\
31	1.3878e-17\\
32	1.3878e-17\\
33	0\\
34	0\\
35	0\\
36	1.3878e-17\\
37	0\\
38	1.3878e-17\\
39	1.3878e-17\\
40	0\\
41	1.3878e-17\\
42	0\\
43	0\\
44	0\\
45	0\\
46	0\\
47	0\\
48	0\\
49	1.3878e-17\\
50	0\\
};
\addplot [color=red, draw=none, mark=o, mark options={solid, red}, forget plot]
  table[row sep=crcr]{%
1	1.0695\\
2	0.19471\\
3	0.11358\\
4	0.71983\\
5	0.21625\\
6	0.43187\\
7	0.080576\\
8	0.11009\\
9	0.14135\\
10	0.078552\\
11	0.081259\\
12	0.031691\\
13	0.029237\\
14	0.017108\\
15	0.016424\\
16	0.0097269\\
17	0.0069886\\
18	0.0043262\\
19	0.0022188\\
20	0.0023207\\
21	0.00089145\\
22	0.0010293\\
23	0.00044064\\
24	0.00033761\\
25	0.00023929\\
26	0.00010139\\
27	0.00011006\\
28	3.9937e-05\\
29	3.9482e-05\\
30	1.9423e-05\\
31	1.1402e-05\\
32	7.9547e-06\\
33	2.8393e-06\\
34	1.2913e-06\\
35	6.4946e-06\\
36	1.3165e-05\\
37	9.3857e-06\\
38	5.2578e-05\\
39	1.2134e-05\\
40	0.00012263\\
41	9.2442e-05\\
42	0.00014173\\
43	0.00066911\\
44	0.00021871\\
45	0.0023496\\
46	0.0012397\\
47	0.002295\\
48	0.001525\\
49	0.00040927\\
50	0.071122\\
};
\end{axis}

\begin{axis}[%
width=0.411\figurewidth,
height=\figureheight,
at={(0.54\figurewidth,0\figureheight)},
scale only axis,
xmin=0,
xmax=50,
xlabel style={font=\color{white!15!black}},
xlabel={$i$},
ymode=log,
ymin=1e-10,
ymax=1,
yminorticks=true,
axis background/.style={fill=white}
]
\addplot [color=green, draw=none, mark=triangle, mark options={solid, green}, forget plot]
  table[row sep=crcr]{%
1	4.0182e-07\\
2	3.419e-07\\
3	2.5822e-07\\
4	3.1892e-07\\
5	3.4943e-08\\
6	2.748e-07\\
7	2.162e-07\\
8	2.023e-07\\
9	2.9768e-07\\
10	2.5042e-07\\
11	3.8582e-07\\
12	3.7063e-07\\
13	6.1847e-07\\
14	6.387e-07\\
15	1.1063e-06\\
16	1.1179e-06\\
17	1.513e-06\\
18	2.6678e-06\\
19	2.1686e-06\\
20	7.7966e-06\\
21	4.336e-06\\
22	1.1267e-05\\
23	5.2911e-06\\
24	6.3224e-06\\
25	8.4739e-05\\
26	2.8301e-05\\
27	0.00011026\\
28	4.015e-05\\
29	3.9926e-05\\
30	1.9976e-05\\
31	1.2522e-05\\
32	9.2472e-06\\
33	4.125e-06\\
34	3.5906e-06\\
35	1.6089e-06\\
36	1.2198e-06\\
37	6.7246e-07\\
38	3.9927e-07\\
39	2.6371e-07\\
40	1.3551e-07\\
41	9.434e-08\\
42	4.8279e-08\\
43	3.1962e-08\\
44	1.752e-08\\
45	1.0588e-08\\
46	6.2161e-09\\
47	3.527e-09\\
48	2.1674e-09\\
49	9.9571e-10\\
50	6.2369e-10\\
};
\addplot [color=red, draw=none, mark=o, mark options={solid, red}, forget plot]
  table[row sep=crcr]{%
1	1.4266e-09\\
2	1.218e-09\\
3	7.9253e-10\\
4	1.0834e-09\\
5	1.8414e-10\\
6	8.8249e-10\\
7	7.1016e-10\\
8	5.8769e-10\\
9	8.5991e-10\\
10	8.6427e-10\\
11	1.085e-09\\
12	1.3244e-09\\
13	1.7159e-09\\
14	1.9984e-09\\
15	3.0517e-09\\
16	3.4376e-09\\
17	5.3633e-09\\
18	6.902e-09\\
19	9.6487e-09\\
20	1.4701e-08\\
21	1.9032e-08\\
22	3.0435e-08\\
23	3.9962e-08\\
24	6.5075e-08\\
25	8.8302e-08\\
26	1.5359e-07\\
27	1.9947e-07\\
28	3.4153e-07\\
29	4.9316e-07\\
30	6.6743e-07\\
31	1.4865e-06\\
32	1.3348e-06\\
33	4.0059e-06\\
34	3.3638e-06\\
35	7.7162e-06\\
36	1.4139e-05\\
37	1.0004e-05\\
38	5.2915e-05\\
39	1.237e-05\\
40	0.00012275\\
41	9.251e-05\\
42	0.00014177\\
43	0.00066914\\
44	0.0002187\\
45	0.0023496\\
46	0.0012397\\
47	0.002295\\
48	0.001525\\
49	0.00040927\\
50	0.071122\\
};
\end{axis}
\end{tikzpicture}%